\documentclass[a4paper,11pt,reqno]{amsart}

\usepackage[utf8]{inputenc}
\usepackage[english]{babel}

\usepackage{amsmath, amsthm, amssymb}
\usepackage{amsfonts}
\usepackage{esint}
\usepackage[dvips]{epsfig}
\usepackage{graphicx}
\usepackage{enumitem}

\usepackage{pdfsync}
\usepackage{hyperref}
\usepackage{color}
\usepackage{array}
\usepackage{appendix}
\usepackage{cleveref, autonum}

\usepackage[
  hmarginratio={1:1},     
  vmarginratio={1:1},     
  textwidth=17.5cm,        
  textheight=25cm,
  heightrounded,          
]{geometry}

\newtheorem{thm}{Theorem}[section]
\newtheorem{cor}[thm]{Corollary}
\newtheorem{lem}[thm]{Lemma}
\newtheorem{prop}[thm]{Proposition}
\newtheorem{conj}[thm]{Conjecture}
\newtheorem*{thm*}{Theorem}
\newtheorem{defn}[thm]{Definition}
\theoremstyle{remark}
\newtheorem{rem}[thm]{Remark}

\newtheorem*{defn*}{Definition}

\numberwithin{equation}{section}

\newcommand{\dx}{\,{\rm d}x}

\newcommand{\dt}{\,{\rm d}t}

\newcommand{\rd}{{\rm d}}

\newcommand{\HH}{\mathbb{H}}
\newcommand{\NN}{\mathbb{N}}

\newcommand{\RR}{\mathbb{R}}
\newcommand\EE{{\mathcal{E}}}

\newcommand\JJ{{\mathcal{J}}}
\newcommand\KK{{\mathcal{K}}}
\newcommand\LL{{\mathcal{L}}}
\newcommand\UU{{\mathcal{U}}}

\newcommand\WW{{\mathcal{W}}}

\newcommand\ed{_{\varepsilon,\delta}}

\newcommand{\vep}{\varepsilon}

\newcommand{\nb}{\nabla}
\newcommand{\dv}{{\rm div}}
\def\dist{\mathrm{dist}} 
\def\supp{\mathrm{supp}} 

\newcommand{\loc}{\mathrm{loc}}

\renewcommand{\le}{\leqslant}
\renewcommand{\leq}{\leqslant}
\renewcommand{\ge}{\geqslant}
\renewcommand{\geq}{\geqslant}

\begin{document}

\title[]{On the existence of solutions to some singular\\ parabolic free boundary problems}

\author[A. Audrito]{Alessandro Audrito}
\address{Alessandro Audrito \newline \indent
Politecnico di Torino, DISMA, Corso Duca degli Abruzzi 24, 10129, Torino, Italia.
 }
\email{alessandro.audrito@polito.it}

\author[T. Sanz-Perela]{Tom\'as Sanz-Perela}
\address{T. Sanz-Perela \textsuperscript{1,2}
\newline
\textsuperscript{1}
Departament de Matem\`atiques i Inform\`atica,
Universitat de Barcelona,
Gran Via de les Corts Catalanes 585, 08007, Barcelona, Spain
\newline
\textsuperscript{2}
Centre de Recerca Matem\`atica, Edifici C, Campus Bellaterra, 08193 Bellaterra, Spain}
\email{tomas.sanz.perela@ub.edu}



%
%
%
%
%
%
%
%
%
%
%

\thanks{This project has received funding from the European Research Council (ERC) under the grant agreement 948029 and the Istituto Nazionale di Alta Matematica INdAM.
This work is supported by the Spanish State Research Agency, through the Severo Ochoa and Mar\'ia de Maeztu Program for Centers and Units of Excellence in R\&D (CEX2020-001084-M).
T. S.-P. is supported by the AEI grants PID2024-156429NB-I00 and PID2021-123903NB-I00.
}

\begin{abstract}
We construct nonnegative weak solutions to the singular parabolic free boundary problem
\[
\partial_t u - \Delta u = - \frac{\mathrm{d}}{\mathrm{d} u} u_+^\gamma , 
\]
where $\gamma \in (0,1]$, $u_+ := \max\{u,0\}$, and the term in the right-hand side denotes the formal derivative of the \emph{non-smooth} function $u \mapsto u_+^\gamma$. 
Weak solutions are obtained as limits of a suitable approximation procedure. We show uniform optimal regularity, optimal growth and nondegeneracy estimates, and a Weiss-type monotonicity formula for solutions to the approximating problem. 
Such uniform estimates are then passed to limit: we prove the existence of a class of weak solutions to the free boundary problem which is \emph{closed under blow-up} and whose weak formulation encodes the \emph{sharp} free boundary condition. 
Finally, we construct several examples of weak solutions with \emph{self-similar} and \emph{traveling wave} form.
\end{abstract}

\maketitle

\setcounter{tocdepth}{1}
\tableofcontents



%
%
%
%
\section{Introduction}
We study nonnegative weak solutions to the following singular semilinear parabolic equation 
\begin{equation}\label{eq:EQSing}\tag{E}
\partial_t u - \Delta u = - \frac{\rd}{\rd u} u_+^\gamma  , 
\end{equation}
where $\gamma \in (0,1]$, $u_+ := \max\{u,0\}$, and the right-hand side denotes the formal derivative of the \emph{non-smooth} function $u \mapsto u_+^\gamma$. 
Several versions of the above equation appear in numerous applications such as combustion theory (limit case $\gamma = 0$, see~\cite{BuckLud:book,BerLar1989:art,CafVaz95}), ice-melting (limit case $\gamma = 1$, see~\cite{Stefan,FriedKinder1975:art}), chemical engineering \cite{Phillips1987:art}, transport of thermal energy in plasma \cite{Martinson1976:art}, and can be interpreted as the \emph{gradient flow} of the energy
\begin{equation}\label{eq:EllipticEnergy}
\JJ_\gamma(u) =  \int |\nb u|^2 + 2 u_+^\gamma \, \rd x,
\end{equation}
originally studied in \cite{Phillips1983:art,Phillips1983bis:art,AltPhillips1986:art} in the elliptic setting. The limit cases $\gamma = 0$ and $\gamma = 1$ correspond to the two most classical \emph{free boundary problems}: the \emph{one-phase problem} (see Alt\&Caffarelli~\cite{AC81}) and the \emph{obstacle problem} (see Caffarelli~\cite{Caf78}), respectively. 

\

As it is well-known, these problems are characterized by the presence of some \emph{free interface}, called \emph{free boundary} (FB). For instance, in the elliptic framework, if $u$ is a nonnegative minimizer of the functional $\JJ_\gamma$ (in a suitable sense, see \cite{AltPhillips1986:art}), then
\begin{equation}
\Delta u = \gamma u^{\gamma-1} \quad \text{ in } \{u > 0\}, \quad \quad \text{ and } \quad \quad |\nb (u^{1/\beta})| = \frac{\sqrt{2}}{\beta}  \quad \text{ in } \partial \{u > 0\},  
\end{equation}
where
\begin{equation}
    \label{eq:Beta}
    \beta := \dfrac{2}{2-\gamma} \in (1,2]
\end{equation}
is the natural scaling power of the problem, and the FB condition must be intended in a suitable weak sense (see \cite{AltPhillips1986:art,KarakhanyanSanzPerela:art} or \Cref{rem:FBcondition} below). 
Thus, the space can be decomposed as the union of the sets $\{u > 0\}$ and $\{u = 0\}$ and, since neither of them is prescribed \emph{a priori}, the boundary $\partial\{u > 0\}$ is an \emph{unknown} of the problem (that is, a free boundary). The study of FB problems is usually very hard since \emph{a priori} both the solutions and their FB's can be very irregular: the main advances in the theory have been obtained through the combined use of techniques developed in different fields, like PDE's, Calculus of Variations, and Geometric Measure Theory.  

\

The mathematical study of the regularity of the solutions and their free interfaces was initiated in the elliptic framework with the seminal papers by Alt\&Caffarelli \cite{AC81} (case $\gamma = 0$) and Caffarelli \cite{Caf78} (case $\gamma = 1$), then followed by numerous outstanding works by several authors. When $\gamma \in (0,1)$, the first results appeared in the works of Phillips \cite{Phillips1983:art,Phillips1983bis:art} and Alt\&Phillips \cite{AltPhillips1986:art}. The development of the theory in the range $\gamma \in (0,1)$ have recently flourished: we mention the FB regularity theory developed in \cite{DeSilSav2021:art,FatKoch2023:art,RestrepoRos2025:art}, the ``generic regularity'' properties established in \cite{FerYu2023:art}, the stability condition obtained in \cite{KarakhanyanSanzPerela:art}, and the construction of singular minimizing cones in \cite{SavinYu25:art,SavinYu25Bis:art}. Lastly, we quote the papers \cite{DeSilSav2022:art,DeSilSav2022bis:art,DiKarVal2022:art,CarTor2025:art} treating the ``very singular'' range $\gamma \in (-2,0)$: the limit case $\gamma = -2$ is strongly connected with the theory of \emph{minimal surfaces}, cf. \cite{DeSilSav2022bis:art,CarTor2025:art} (see also \cite{art:CafYu18,art:Yu19} in the parabolic setting).

\

The parabolic framework is less studied. Various versions of the equation \eqref{eq:EQSing} appear in the theory of \emph{quenching} (see for instance \cite{AckWal78,Lev85}): the literature in this context is concerned with the existence of quenching points\footnote{Essentially, points at which $u$ vanish but $\partial_t u $ blows up, see \cite{AckWal78,Lev85}.} and continuation after quenching, but does not treat the FB theory at all. Some partial results about the existence of weak solutions and corresponding monotonicity formulas were obtained by Phillips \cite{Phillips1987:art} and Weiss \cite{Weiss1999:art}, respectively, and we refer the reader to the paragraphs after Theorem \ref{thm:MAIN1} for a more detailed discussion. In the special cases $\gamma = 0,1$, the theory has been developed in many papers, see for instance \cite{FriedKinder1975:art,Caf78,CafVaz95,art:Weiss2003,CafPetSha04,AndWeiss2009,FigRosSerra24} and partially extended to the range $\gamma \in [1,2)$ by Weiss \cite{Weiss1999:art,Weiss2000:art} (see also \cite{ChoeWeiss2003} and our recent work \cite{AudritoSanz2022:art}). 

To the best of our knowledge, if $\gamma \in (0,1)$, there are not further references except the recent papers \cite{Araujo24,Jeon24} where the authors study the existence, regularity, and quasi-convexity properties of viscosity solutions to versions of \eqref{eq:EQSing} with ``fully nonlinear'' diffusion. Both their approaches and ours are based on a regularization of the nonlinearity in \eqref{eq:EQSing} (see Subsection \ref{subsub:RegApproach}), but strongly differ in the techniques and results: we develop several methods having a ``variational-energetic'' flavor and the results we obtain are \emph{sharp} from different viewpoints (see the paragraphs below for further details).
\subsection{Leading ideas and the notion of weak solutions} 
In this paper, we study the \emph{existence} of nonnegative weak solutions to \eqref{eq:EQSing}, following two different approaches: on the one hand, we construct weak solutions to the initial-value problem
\begin{equation}\label{eq:ParReacDiff}\tag{P}
\begin{cases}
\partial_t u - \Delta u = - \frac{\rd}{\rd u} u_+^\gamma  \quad &\text{ in }  Q := \mathbb{R}^n \times (0,\infty) \\
u|_{t=0} = u_\circ     \quad &\text{ in } \mathbb{R}^n,
\end{cases}
\end{equation}
where $n \geq 1$ and $u_\circ$ is a suitable nonnegative initial data (\Cref{sec:GlobalUnifEnEst} to  \Cref{sec:WeissMonotonicity}); on the other hand, we drop the initial condition and prove existence of solutions with \emph{self-similar} and \emph{traveling wave} form  (\Cref{sec:SelfSimilar} and \Cref{sec:TravelingWaves}). In doing so, there are some critical aspects that one has to take into account:

\

(A) In contrast with the elliptic setting, weak solutions cannot be constructed by direct minimization of a suitable energy functional. Instead, we proceed with a regularization procedure of~\eqref{eq:ParReacDiff} in the spirit of Phillips \cite{Phillips1987:art},  Caffarelli\&V\'azquez \cite{CafVaz95}, and Weiss \cite{art:Weiss2003} (cf. Ilmanen \cite{Imanen93:art} for the Mean Curvature Flow framework): weak solutions will be obtained as \emph{limits} of a family of singular perturbation problems, see \eqref{eq:ProbEps} below.    

\    

(B) Since the function $u \mapsto u_+^\gamma$ is smooth when $u>0$, any weak solution $u$ must satisfy 
    \begin{equation}\label{eq:ParEqPosSetIntro}
    \partial_tu - \Delta u = -\gamma u^{\gamma-1} \quad \text{ in } \{u > 0\}.
    \end{equation}
    Further, any notion of weak solution must encode the FB condition characterizing the underlying FB problem. In our setting, we will show that
\begin{equation}\label{eq:FBCondIntro}
    |\nb (u^{1/\beta})| = \frac{\sqrt{2}}{\beta}  \quad \text{ in } \partial \{u > 0\},
\end{equation}
    in a suitable weak sense, where $\nb$ denotes the spatial gradient of $u$; see \Cref{rem:FBcondition}, \Cref{lem:FBCondition} and \Cref{lem:FBConditionBis}.
    We stress that such properties are quite delicate in the parabolic setting, whilst reasonably natural for minimizers in the elliptic one, since they can be deduced via standard variations of the functional $\JJ_\gamma$. Actually, the fact that \eqref{eq:FBCondIntro} is the natural (and sharp) FB condition is unknown, at least in the parabolic framework (range $\gamma > 0$): the weak solutions constructed in \cite{Araujo24,Jeon24} satisfy $|\nb u| = 0$ on $\partial\{u > 0\}$, a weaker and non-sharp formulation of the FB condition.

\    

(C) Once weak solutions are constructed, the main issue is to study the \emph{regularity} properties of the FB. The usual strategy is to \emph{blow-up} solutions around FB points: if $u$ is a weak solution and $(x_\circ,t_\circ) \in \partial\{u > 0\}$, one defines the \emph{blow-up family}
    \begin{equation}\label{eq:BlowUpFamIntro}
            u_r^{(x_\circ,t_\circ)}(x,t) := \frac{u(x_\circ +rx,t_\circ + r^2t)}{r^\beta}, \qquad r > 0,
    \end{equation}
    and studies the limits of $u_r^{(x_\circ,t_\circ)}$ as $r \downarrow 0$. From the analysis of the \emph{blow-up limits} (and a lot of extra work!) it is often possible to prove that $\partial\{u > 0\}$ is regular, in some suitable sense, see for instance \cite{AndWeiss2009} (case $\gamma = 0$) and \cite{FigRosSerra24} (case $\gamma = 1$).

    However, some subtle and delicate properties play a key role in the blow-up analysis. First of all, $u$ must satisfy some (optimal) regularity and non-degeneracy estimates to actually define the blow-up limits via compactness arguments and, moreover, the FB must be (at least) a closed set with Lebesgue measure zero (if not, one cannot expect $\partial\{u > 0\}$ to be regular, in any reasonable sense). Second, the class of weak solutions must be \emph{closed under blow-ups}, in the sense that any blow-up limit of a weak solution is still a weak solution: besides being quite natural, this property is crucial in the study of the FB (for example, in the majority of the \emph{dimension reduction} arguments). Third, since blow-up limits are expected to be (parabolically) $\beta$\emph{-homogeneous}, weak solutions must carry some information yielding the homogeneity of their blow-up limits. 
    We accomplish this by showing that the weak solutions that we construct enjoy a Weiss-type monotonicity formula in the spirit of the case $\gamma = 0$ (see \cite{art:Weiss2003} and \Cref{sec:WeissMonotonicity} below). In the parabolic framework, the proofs are quite involved: the monotonicity formula is derived at the level of the approximation ---i.e., for solutions to the singular perturbation problem \eqref{eq:ProbEps}--- and then pushed to the limit in a final step.

\    

(D) Lastly, the special solutions we construct (self-similar solutions and traveling waves) must be weak solutions. This is not obvious since such special solutions are obtained by solving \eqref{eq:ParEqPosSetIntro} and imposing the FB condition \eqref{eq:FBCondIntro}: in a second step, one has to show they are weak solutions as well, see \Cref{remark:solutionsAreWeakSol}. 

\

In light of the discussion above, we give the definition of weak solutions to \eqref{eq:EQSing} and \eqref{eq:ParReacDiff}. 
Regarding the equation, we consider:
\begin{defn}\label{def:WeakSolEIntro}
Let $n \geq 1$, $\gamma \in (0,1]$, and let $\mathcal{A} \subseteq \RR^{n+1}$ be an open set. 
We say that $u$ is a weak solution to \eqref{eq:EQSing} in $\mathcal{A}$ if for every open bounded interval $I$ and every open bounded set $\Omega \subset \RR^n$ such that $\Omega \times I \subset\subset \mathcal{A}$, we have
\begin{itemize}
    \item $u \in L^2(I;H^1(\Omega))$ with $\partial_tu \in L^2(\Omega\times I)$ and $u_+^{\gamma-1} \in L^1(\Omega\times I)$.\footnote{By definition, $u_+^{\gamma-1} := \chi_{\{u > 0\}} u^{\gamma-1}$, where $\chi_E$ deontes the characteristic function of the set $E$.}
    %
    \item $u$ satisfies 
\begin{equation}\label{eq:FirstOVLimIntro}
\int_{\Omega\times I} \partial_t u \varphi + \nabla u \cdot\nabla\varphi + \gamma u_+^{\gamma-1} \varphi = 0,
\end{equation}
for every $\varphi \in C_c^\infty(\Omega\times I)$.
    \item $u$ satisfies
\begin{equation}\label{eq:FirstIVLimIntro}
\int_{\Omega\times I} \big( |\nabla u|^2 + 2u_+^\gamma \big) \dv_x \Phi  -  2 \nb u \cdot D_x \Phi \cdot \nabla u - 2\partial_t u \,(\nb u \cdot \Phi) = 0,
\end{equation}
for every $\Phi \in C_c^\infty(\Omega\times I;\RR^{n+1})$.\footnote{Here and throughout the paper, we use the notation $v\cdot M\cdot v := v\cdot M\cdot v^\intercal$, for every vector $v \in \RR^{1\times n}$ and every matrix $M \in \RR^{n\times n}$. Further, see Subsection \ref{Subsec:Notation} for the definition of $\dv_x \Phi$ and $D_x \Phi$.}
\end{itemize}
\end{defn}

Some comments regarding the above definition are in order. Concerning what was discussed in paragraph (B) above, formulas \eqref{eq:FirstOVLimIntro} and \eqref{eq:FirstIVLimIntro} are usually referred to as the \emph{weak formulation} and the \emph{weak formulation with respect to domain variations} of \eqref{eq:EQSing}, respectively. We will show that, under suitable regularity assumptions  on $u$ and the FB $\partial\{u > 0\}$, they are equivalent in the sense that both imply \eqref{eq:ParEqPosSetIntro} and \eqref{eq:FBCondIntro}, and vice versa (see \Cref{rem:FBcondition}, \Cref{lem:FBCondition}, and \Cref{lem:FBConditionBis} for all the details).

However, without some \emph{a priori} regularity assumptions, the two weak formulations seem to be independent. This may be related with the assumption in \eqref{eq:FirstOVLimIntro}, requiring $u_+^{\gamma-1}$ being locally integrable only. Indeed, this implies that, when $\gamma \in (0,1)$, \eqref{eq:EQSing} can be interpreted as a parabolic equation with $L^1$ right-hand side, the critical energy space for the regularity theory, in a sense, similar to the \emph{Harmonic Maps Flow} framework, see \cite{ChStr89,LinWang08:book}; we recall that in the \emph{Harmonic Maps} theory, both the weak formulation and the weak formulation with respect to domain variations are necessary to develop any partial regularity results in dimension $n \geq 3$, see \cite{Riviere95}. Furthermore, while superfluous when $\gamma = 1$, checking that a candidate solution satisfies $u_+^{\gamma-1}$ locally in $L^1$ is highly nontrivial for $\gamma \in (0,1)$ ---note that in the case $\gamma = 0$ one must replace $\gamma u_+^{\gamma-1}$ by a measure, and also that solutions $u$ may not satisfy $u_+^{-1}$ being locally integrable, since they grow linearly from the FB; see \cite{KrivenWeiss25:art,art:Weiss2003}.

Finally, in connection to paragraph (C) above, we will show that the class of weak solutions defined above is closed under blow-up limits (see \Cref{thm:MAIN2Blowups}).

\

For what concerns the initial value problem \eqref{eq:ParReacDiff}, we consider weak solutions to \eqref{eq:EQSing} in $Q$ which are continuous flows in $L^2(\RR^n)$, up to the initial time $t = 0$. 
\begin{defn}\label{def:WeakSolPIntro}
Let $n \geq 1$, $\gamma \in (0,1]$, and let $u_\circ \in L^2(\RR^n)$.  
We say that $u$ is a weak solution to \eqref{eq:ParReacDiff} if $u$ is a weak solution to \eqref{eq:EQSing} in $Q$ and, in addition, $u \in \cap_{R>0} C([0,R]:L^2(\RR^n))$ with $u|_{t=0} = u_\circ$ in $L^2(\RR^n)$.   
\end{defn}
\subsection{Setting of the problem and main results}\label{subsub:RegApproach} 
As already mentioned, weak solutions will be obtained as limits of a suitable approximation procedure that we present next.

We consider $h \in C_c^1([0,1])$ satisfying 
\begin{equation}
\label{Eq:Propertiesh}
h \geq 0, \qquad h'(0) >0 \qquad \text{ and } \qquad \int_0^1 h(v)\rd v = 1, 
\end{equation}
and its integral function
\begin{equation}\label{eq:AprxChi}
H(u) := 
\begin{cases}
0 \quad &\text{if } u \leq 0 \\
\int_0^u h(v) \, \rd v \quad &\text{if } u > 0.
\end{cases}
\end{equation}
Then, for $\vep > 0$, we set
\begin{equation}
\label{eq:Hepsilon}
H_\vep(u) := H(u/\vep^\beta) \qquad 
\text{ and } 
\qquad
h_\vep(u) := \frac{\rd}{\rd u} H_\vep(u) = \vep^{-\beta} h(u/\vep^\beta).
\end{equation}
Notice that $H_\vep$ converges pointwise to $\chi_{(0,\infty)}$ in $\RR$ as $\vep \downarrow 0$ (as above, $\chi_E$ denotes the characteristic function of the set $E$). Using $H_\vep$, we define 
\begin{equation}
\label{eq:Fepsilon}
    F_\vep(u) := H_\vep(u)u^\gamma 
    \qquad \text{ and } \qquad
    f_\vep(u) := \frac{\rd}{\rd u} F_\vep(u) = \vep^{-\beta} h(u/\vep^\beta) u^\gamma +  \gamma H(u/\vep^\beta) u^{\gamma - 1}.
\end{equation}
Note that for every $r>0$ we have the following scaling relations:
\begin{equation}\label{eq:Scalingfeps}
    F_\vep (r^\beta u) = r^{\gamma \beta} F_{\vep/r} (u)
    \quad \text{ and } \quad 
    f_\vep(r^\beta u) = r^{\beta - 2} f_{\vep/r} (u).
\end{equation}
\begin{figure}[!ht]\label{fig:NLRHS}
\includegraphics[scale=0.6]{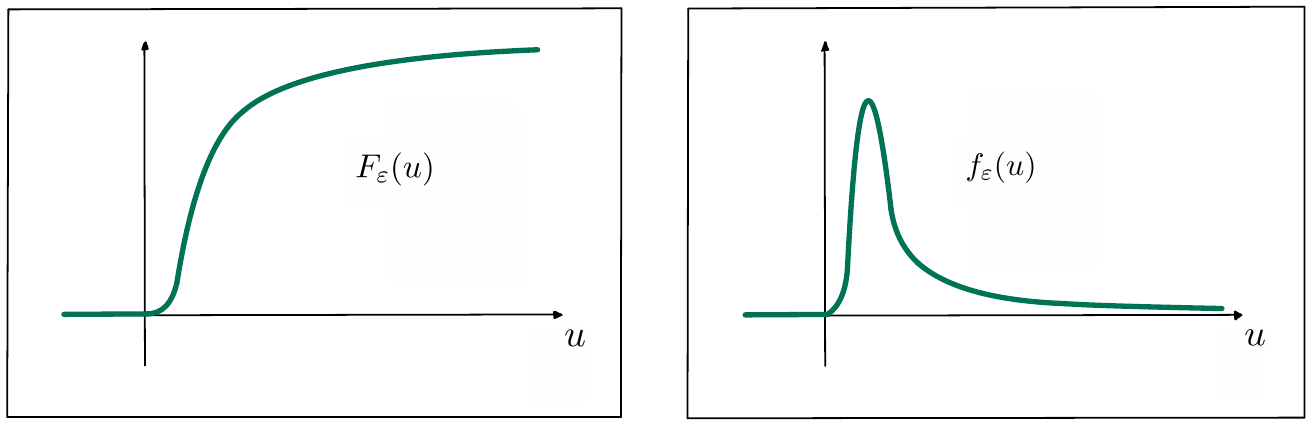}
\caption{The qualitative graphs of $F_\vep$ (left) and $f_\vep$ (right).}
\end{figure}

Taking the previous definitions into account, we introduce the approximating initial value problem that we will consider through most part of the article:
\begin{equation}
\label{eq:ProbEps}
\tag{P$_\vep$}
\begin{cases}
\partial_t u_\vep - \Delta u_\vep = - f_\vep(u_\vep) \quad &\text{ in }  Q \\
u_\vep|_{t=0} = u_\circ    \quad &\text{ in } \mathbb{R}^n.
\end{cases}
\end{equation}
Note that, by \eqref{Eq:Propertiesh}, $f_\vep \in C^{1+\gamma}$ (for every $\vep > 0$ fixed) and, in the limit as $\vep \downarrow 0$, it formally approximates the singular right-hand side in our FB problem \eqref{eq:EQSing} and/or \eqref{eq:ParReacDiff}. Therefore, for every $\vep > 0$, the solutions $u_\vep$ to \eqref{eq:ProbEps} are classical and positive in $Q$ (whenever $u_\circ$ is nontrivial and nonnegative, see Section \ref{sec:GlobalUnifEnEst}) and are expected to approach a weak solution $u$ to \eqref{eq:ParReacDiff} as $\vep \downarrow 0$: the leading idea is that, at the approximating level, the roles of $\{u>0\}$ and $\{u=0\}$ are played by $\{u_\vep > \vep^\beta\}$ and $\{u_\vep \leq \vep^\beta\}$, respectively, and a careful analysis of their properties is required to actually pass to the limit as $\vep \downarrow 0$. 

The main results of this paper, regarding the construction of solutions via the problem \eqref{eq:ProbEps}, are gathered in the following theorem.
\begin{thm}\label{thm:MAIN1}
Let $n \geq 1$, $\gamma \in (0,1]$, $\alpha \in (0,1)$, and let $u_\circ \in C_c^{2+\alpha}(\RR^n)$ be nonnegative and nontrivial. Let $\{u_\vep\}_{\vep > 0}$ be a family of solutions to \eqref{eq:ProbEps}.

Then there exist $\nu \in (0,1)$, $\vep_j \downarrow 0$, and a nonnegative weak solution $u \in C^\nu(\overline{Q})$ to \eqref{eq:ParReacDiff} such that\footnote{See Subsection \ref{Subsec:Notation} for our notation and terminology regarding local convergence.}
\[
u_{\vep_j} \to u \quad \text{ locally in } C^\nu(\overline{Q}) \cap L^2(0,\infty;H^1(\RR^n)),
\]
and
\[
\begin{aligned}
&\{u_{\vep_j} \geq \vep_j^\beta\} \to \overline{\{u > 0\}} \quad \text{ locally Hausdorff in } Q, \\
&H_{\vep_j}(u_{\vep_j}) \to \chi_{\{u > 0\}}  \qquad\;\, \text{ locally in } L^1(Q),    
\end{aligned}
\]
as $j \uparrow \infty$. Furthermore:  
\begin{enumerate}
\item[(i)] $u$ satisfies the energy estimates \eqref{eq:EnBound1Lim} and \eqref{eq:EnBound2Lim}, stated in \Cref{lem:StrCmp};

\medskip

\item[(ii)] $u$ satisfies the regularity estimates \eqref{eq:OptRegSpaceLim}, \eqref{eq:OptRegTimeLim}, and \eqref{eq:BdrHolderLim}, stated in  \Cref{lem:StrCmp};

\medskip

\item[(iii)] $u$ satisfies the optimal growth and non-degeneracy estimates \eqref{eq:OptGrLimLem} and \eqref{eq:NonDegLimLem}, stated in \Cref{lem:HausConv};

\medskip

\item[(iv)] $\{u>0\}$ has positive density and $\mathcal{L}^{n+1}(\partial\{u > 0\}) = 0$, see  \Cref{lem:L1ConvFBMeas0};

\medskip

\item[(v)] $u$ satisfies the Weiss-type monotonicity formula \eqref{eq:WeissForLim}, stated in \Cref{prop:WeissForSemLim}.

\medskip

\item[(vi)] $u$ has compact support in $\overline{Q}$, see Corollary \ref{cor:CmptSupp}.
\end{enumerate}
\end{thm}

The proofs of the results gathered in \Cref{thm:MAIN1} are presented in \Cref{sec:convergence} and \Cref{sec:WeissMonotonicity}, using the preliminary results from \Cref{sec:GlobalUnifEnEst} and \Cref{sec:OptRegNonDeg}. 
In particular, the statements regarding the convergence of the solution are given in \Cref{lem:StrCmp}, \Cref{lem:HausConv}, and \Cref{lem:L1ConvFBMeas0}.
The fact that the solutions we obtain are weak solutions to \eqref{eq:ParReacDiff} in the sense of \Cref{def:WeakSolPIntro} is given in \Cref{lem:L2H1StgConv} and \Cref{lem:FirstOVLim}, and the proof of each of the properties $(i)$-$(vi)$ is given in the results referred in the above statement.
Before proceeding, let us comment some important features that the above statement bears.

\

$\bullet$ First, it proves the existence of nontrivial weak solutions to \eqref{eq:ParReacDiff} and, furthermore, it establishes their basic, yet crucial, properties $(i)$-$(vi)$. In this regard, Theorem \ref{thm:MAIN1} generalizes and extends two previous works by Phillips \cite{Phillips1987:art} and Weiss \cite{Weiss1999:art}. 
    
    On the one hand, Phillips proved the existence of weak solutions to \eqref{eq:ParReacDiff} in the sense of \eqref{eq:FirstOVLimIntro}, using a special approximation term $f_\vep$ in the right-hand side of \eqref{eq:ProbEps} (see Remark \ref{rem:PhilProb}). However, even in the framework of \cite{Phillips1987:art}, a key step in the study of the limit of $u_\vep$ as $\vep \downarrow 0$ lacks of a complete justification (see Remark \ref{rem:PhilProb} again). Here we present a full proof, working in a more general setting (that is, under the only assumptions \eqref{Eq:Propertiesh} on the approximating term): to do so, we combine some of the Phillips' techniques with new fine uniform estimates for solutions to the approximating problem \eqref{eq:ProbEps} and their level sets $\{u_\vep \sim \vep^\beta\}$, and a delicate barrier argument. 

    On the other hand, the validity of a Weiss-type monotonicity formula was shown in \cite{Weiss1999:art} for a class of solutions, called ``variational solutions'', under the assumption $\gamma \in (2/3,1]$ (see \Cref{sec:WeissMonotonicity}): Weiss introduced the notion of variational solutions and proved that they satisfy the monotonicity formula; in a second step, he showed that the weak solutions constructed by Phillips in \cite{Phillips1987:art} are variational solutions whenever $\gamma \in (2/3,1]$. Here, we proceed in the spirit of \cite{art:Weiss2003}, establishing a monotonicity formula for the solutions of the approximating problem \eqref{eq:ProbEps} and obtaining, in the limit, a monotonicity formula for weak solutions to \eqref{eq:ParReacDiff}, valid for every $\gamma \in (0,1]$.   

\

$\bullet$ Second, the value $\gamma = 0$ is a critical threshold. Some remarkable, yet partial, results were obtained in \cite{CafVaz95,art:Weiss2003}. Nevertheless, the weak formulations of the equation of the limit solutions are derived only under stronger assumptions on the solutions themselves. This is essentially because the case $\gamma = 0$ is ``more degenerate'' and the family of approximating solutions $u_\vep$ may converge locally uniformly to $0$ as $\vep \downarrow 0$ (see \cite{art:Weiss2003}). This does not happen if $\gamma \in (0,1]$. Actually, assuming that $u_\vep$ satisfy the uniform non-degeneracy estimate \eqref{eq:NonDeg}, our proofs directly extend to the case $\gamma = 0$ and allow to show that the limit $u$ is a weak solution in the sense of \eqref{eq:FirstIVLimIntro} and the properties $(i)$-$(v)$ hold true as well.

\

$\bullet$ Finally, as already anticipated, we prevalently work at the level of the approximation, establishing uniform energy, regularity, and non-degeneracy estimates for the solutions $u_\vep$ to \eqref{eq:ProbEps}, which are then pushed to the limit as $\vep \downarrow 0$. This approach has two main features. The first is that it allows to \emph{quantitatively} describe how the solutions $u_\vep$ converge to a limit solution $u$ and how the sets $\{u_\vep \sim \vep^\beta\}$ converge to the FB $\partial\{u>0\}$. The second is the \emph{robustness} of the methods: the techniques we develop throughout the paper ``pass to the limit'' as $\vep \downarrow 0$, that is, they essentially apply when one studies the limit solution $u$. For example, in relation to point (C) above, the same techniques used to show  \Cref{thm:MAIN1} allow to prove that the class of weak solutions we consider is closed under blow-up limits, as stated in the following corollary (see \Cref{Subsec:Blowups} for the proof).
\begin{cor} \label{thm:MAIN2Blowups}
    Let $n \geq 1$, $\gamma \in (0,1]$, $\alpha \in (0,1)$, and let $u_\circ \in C_c^{2+\alpha}(\RR^n)$ be nonnegative and nontrivial. 
    Let $u$ be a nonnegative weak solution to \eqref{eq:ParReacDiff} given by \Cref{thm:MAIN1}, $(x_\circ,t_\circ) \in \partial\{ u > 0\}\cap Q$ and let $u_r := u_r^{(x_\circ,t_\circ)}$ be the blow-up family defined in \eqref{eq:BlowUpFamIntro}.

Then for every $\nu \in (0,\frac{\beta}{2})$, there exist $r_j \downarrow 0$ and a nonnegative weak solution $u_0 \in C^{\beta/2}(\RR^{n+1})$ to \eqref{eq:EQSing} in $\RR^{n+1}$ such that
\[
u_{r_j} \to u_0 \quad \text{ locally in } C^\nu(\RR^{n+1}) \cap L^2(\RR;H^1(\RR^n)),
\]
and
\[
\begin{aligned}
&\overline{\{u_{r_j} > 0\}} \to \overline{\{u_0 > 0\}} \quad \text{ locally Hausdorff in } \RR^{n+1}, \\
&\chi_{\{u_{r_j} > 0\}} \to \chi_{\{u_0 > 0\}}  \qquad\, \text{ locally in } L^1(\RR^{n+1}),    
\end{aligned}
\]
as $j \uparrow \infty$. Furthermore:  
\begin{enumerate}
\item[(i)] $u_0$ satisfies the energy estimates \eqref{eq:UnifEnBdBUFam1+2};

\medskip

\item[(ii)] $u_0$ satisfies the regularity estimates \eqref{eq:OptRegEStBU};

\medskip

\item[(iii)] $u_0$ satisfies the optimal growth and non-degeneracy estimates  \eqref{eq:OptGrNonDegBU};

\medskip

\item[(iv)] $\{u_0 > 0\}$ has positive density and $\mathcal{L}^{n+1}(\partial\{u_0 > 0\}) = 0$, see \eqref{eq:FB0MeasBU}.

\medskip

\item[(v)] $u_0$ is parabolically $\beta$-homogeneous with respect to $(0,0)$ ``backward in time'', that is,
    \begin{equation}\label{eq:HomogBUIntro}
    u_0(rx,r^2t) = r^\beta u_0(x,t),
    \end{equation}
    for every $(x,t) \in \RR^n \times(-\infty,0)$ and every $r > 0$, or, equivalently, 
    \[
    u_0(x,t) = |t|^{\frac{\beta}{2}} U\big(|t|^{-\frac{1}{2}}x\big),
    \]
    in $\RR^n \times(-\infty,0)$, for some $U: \RR^n \to [0,\infty)$.
\end{enumerate}
\end{cor}
The above corollary closes the study of weak solutions to the initial-value problem \eqref{eq:ParReacDiff}. The last two sections of this paper are devoted to the study of the existence of weak solutions to \eqref{eq:EQSing} with \emph{self-similar} and \emph{traveling wave} form, respectively.

\

In \Cref{sec:SelfSimilar} we construct \emph{radial} weak solutions to \eqref{eq:EQSing} in $Q$ with \emph{self-similar form} ``forward in time'', that is, solutions with form 
\begin{equation}\label{eq:SelfSimIntro}
u(x,t) = t^{\frac{\beta}{2}} U\big(t^{-\frac{1}{2}}|x|\big), \qquad (x,t) \in Q,
\end{equation}
for some $U: [0,\infty) \to [0,\infty)$, called the self-similar profile of $U$. By definition, $U$ characterizes the corresponding self-similar solution. In  \Cref{thm:ExSSProfiles}, we show that for every $\gamma \in (0,1]$ and every $R>0$, there exists a profile $U$ satisfying $\{ U(r) = 0\} = \{r \leq R\}$ and such that the function $u$ defined as in \eqref{eq:SelfSimIntro} is a weak solution to \eqref{eq:EQSing} in $Q$. Consequently, we deduce the existence of self-similar solutions $u$ with \emph{unbounded support} and \emph{expanding contact set}:
\[
\{u = 0\} = \{(x,t): t > 0, \,|x|^2 \leq R^2t\}.
\]
The FB of $u$ is the paraboloid $R^2t = |x|^2$ and, at each time-slice $t > 0$, $\{u = 0\}$ is a ball of radius $R\sqrt{|t|}$.

The construction of such self-similar profiles is obtained by combining fine ODE's methods and comparison arguments. The cases $\gamma = 0$ and $\gamma = 1$ are treated separately, since the equation of the profile is equivalent to a Confluent Hypergeometric Equation and the solutions are explicit in terms of special functions. Further, when $\gamma = 0$, the same techniques allow us to complement the construction of self-similar solutions with \emph{bounded-shrinking support} carried out in \cite[Section 1]{CafVaz95} while, when $\gamma = 1$, to show non-existence of this class of self-similar solutions; see also \cite{HadzicRaphael2019:art, FioravantiRosTorres2025:art} for related results. 
 Unfortunately, when $\gamma\in (0,1)$, the techniques we use to study self-similar solutions ``forward in time'' seem not to apply to construct self-similar solutions ``backward in time''. Actually, some partial analytic results and numerical computations had led us to propose a nonexistence conjecture in this range, which is left as an open problem.
We refer the reader to \Cref{sec:SelfSimilar} for all the details.

 \

Finally, in \Cref{sec:TravelingWaves} we construct weak solutions to \eqref{eq:EQSing} in $\RR^{n+1}$ with traveling wave (TW) form, that is, solutions with form 
\begin{equation}\label{eq:TWAnsatzIntro}
u(x,t) = \phi(e\cdot x-ct),
\end{equation}
where $\phi: \RR \to [0,\infty)$ is the wave's profile, $c \in \RR$ is the profile's speed, and $e \in \RR^n$ is a fixed unit vector: $u$ is an \emph{eternal} solution (i.e., defined for all times $t\in \RR$) identified by the fixed profile $\phi$ traveling along the direction $e$ with speed $c$.

In \Cref{thm:ExTW}, we classify the \emph{admissible} profiles via a phase-plane analysis while, in  \Cref{rem:CollTW}, we use the admissible TW profiles to build examples of ``colliding TW solutions'' (see \cite{art:Weiss2003}). They are weak solutions exhibiting non-standard singular FB points. An easy way to visualize them, is imagining two planar fronts with disjoint supports traveling in opposite directions, and ``colliding'' at some time $t = T$: for every $t < T$, the FB is made by two disjoint parallel lines that collapse to a ``multiplicity 2'' line at time $t=T$. It is worthwhile to notice that such kind of singularities do \emph{not} appear in the Mean Curvature Flow theory (see Remark \ref{rem:CollTW}): this is related to the ``multiplicity 1 conjecture'' (see \cite{Bamler2024:art}) and the validity of the Strong Maximum Principle (see for instance \cite{Mantegazza:book}). 
\subsection{Notation and terminology} 
\label{Subsec:Notation}
We recall here the notations we adopt throughout the paper.

We will always denote $Q := \RR^n \times (0,\infty)$ and, as usual, for every $r > 0$, we set
\begin{equation}
    \label{eq:DefQr}
    Q_r := B_r \times (-r^2, r^2), \qquad Q_r^+ := B_r \times (0, r^2) \quad \text{ and } \qquad Q_r^- := B_r \times (-r^2, 0).
\end{equation}
For $(x_\circ, t_\circ)\in \RR^{n + 1}$, as customary, we denote $Q_r(x_\circ, t_\circ) := (x_\circ, t_\circ) + Q_r$ and, analogously, $Q_r^\pm(x_\circ, t_\circ):= (x_\circ, t_\circ) + Q_r^\pm$.

Concerning the H\"older spaces, for $k=0,1,2,\ldots$ and $\alpha \in (0,1)$, we will consider the usual notation $C^{k + \alpha} = C^{k, \alpha}$.
Sometimes when using regularity results, we will refer to the H\"older spaces $\HH_{k + \alpha}$, defined as $H_{k + \alpha}$ in \cite[Chapter IV, Section 1]{Lieberman1996:book} (we will not use nor need the definition, just recall that these are essentially H\"older spaces in the natural parabolic metric).

Throughout the paper, whenever we say that $w_{j} \to w$ \textit{locally } in $C^\alpha(\Omega)$ it means that, for every compact set $\mathcal{K} \subset\subset \Omega$, we have $w_{j} \to w \text{ in } C^\alpha(\mathcal{K})$.
We use the analogous definitions of local convergence in $L^p(\Omega)$ for $p\geq 1$.
When we say that $w_{j} \to w$ \textit{ locally } in $L^2(I;H^1(\Omega))$, it means that, for every open bounded interval $J \subset\subset I$ and every open bounded $\omega \subset\subset \Omega$, we have $w_{j} \to w \text{ in } L^2(J;H^1(\omega))$. 

Finally, if $\mathcal{A} \subseteq \RR^{n+1}$ is an open set and $\Phi \in C_c^\infty(\mathcal{A};\RR^{n+1})$ with $\Phi = (\Phi^1,\ldots,\Phi^{n+1})$, we set
\[
\dv_x \Phi := \sum_{j=1}^n \partial_{x_j}\Phi^j, \qquad D_x\Phi := (\partial_{x_j}\Phi^i)_{i,j=1,\ldots,n}.
\]
%
%
%
%

%
%
%
%
%
%
%
%
\section{Energy estimates for the approximating problem}\label{sec:GlobalUnifEnEst}
This introductory section is devoted to show existence of weak solutions to \eqref{eq:ProbEps} together with some uniform energy estimates. 
The main result is  \Cref{prop:LimDelta} below, and the proof is an application of the energy estimates we obtained in \cite[Proposition 2.3]{AudritoSanz2022:art} for the range $\gamma \in [1,2)$, taking into account that they work also in our setting ---see \cite[Remark~ 2.6]{AudritoSanz2022:art}.
The tool used is \emph{elliptic regularization}, and formally works as follows (see also \cite{AudritoSerraTilli2021:art}). 
We approximate solutions to the parabolic problem \eqref{eq:ProbEps} by using suitable minimizers of the functional
\begin{equation}\label{eq:Functional0}
\EE\ed(u) = \int_0^\infty \frac{e^{- t/\delta}}{\delta} \int_{\RR^n} \delta|\partial_t u|^2 + |\nabla u|^2 + 2 F_\vep(u) \dx \dt,
\end{equation}
where $\vep,\delta > 0$ are free parameters. Under appropriate assumptions on $u_\circ$, it is not difficult to check that any minimizer $u\ed$ satisfies 
\begin{equation}\label{eq:PerProbEps}
\begin{cases}
-\delta \partial_{tt} u\ed + \partial_t u\ed - \Delta u\ed = - f_\vep(u\ed) \quad &\text{ in }  Q \\
u\ed|_{t=0} = u_\circ     \quad &\text{ in } \mathbb{R}^n,
\end{cases}
\end{equation}
in the weak sense, and thus, under appropriate energy boundedness assumptions, one can hope to pass to the limit as $\delta \downarrow 0$ and obtain a weak solution $u_\vep$ to \eqref{eq:ProbEps}. As mentioned above, this plan can be successfully carried out using the same techniques in \cite{AudritoSanz2022:art}, with minor changes. 
Therefore we omit the details\footnote{We only remark that the second estimate in \eqref{eq:EnBound2Eps} is not stated in the results but is established in the proofs in \cite{AudritoSanz2022:art}, and that \eqref{eq:LinftyBoundEps} is not established in \cite{AudritoSanz2022:art} but follows easily by a comparison argument.} of the proof and in the next result we state the existence of weak solutions and the uniform energy estimates we will use later on in the paper.

\begin{prop}\label{prop:LimDelta} 
Let $\gamma \in [0,1]$ and let $u_\circ \in H^1(\RR^n)$ be nonnegative with $u_\circ^\gamma \in L^1(\RR^n)$. Set
\[
\UU_\circ := \Big\{u \in  \bigcap_{R > 0} L^2((0,R): H^1(\RR^n))\cap C([0,R]:L^2(\RR^n)): \partial_tu \in L^2(Q), \; u|_{t=0} = u_\circ \,\text{ in } L^2(\RR^n) \Big\}.
\]
Then, for every $\vep > 0$, there exists a nonnegative weak solution $u_\vep$ to \eqref{eq:ProbEps}, in the sense that $u_\vep \in \UU_\circ$ and
\begin{equation}\label{eq:FirstOVvep}
\int_Q \partial_t u_\vep \varphi + \nabla u_\vep \cdot\nabla\varphi + f_\vep(u_\vep) \varphi = 0,
\end{equation}
for every $\varphi \in C_c^\infty(Q)$. Furthermore, for every $\vep > 0$,  
\begin{equation}\label{eq:EnBound1Eps}
\int_Q |\partial_t u_\vep|^2 \dx \rd t \leq \|u_\circ\|_{H^1(\RR^n)}^2 + 2\|u_\circ^\gamma\|_{L^1(\RR^n)},
\end{equation}
and, for every $R > 0$, there exists $C(u_\circ,R) > 0$, depending only on $\|u_\circ\|_{H^1(\RR^n)}$, $\|u_\circ^\gamma\|_{L^1(\RR^n)}$, and $R$, such that
\begin{equation}\label{eq:EnBound2Eps}
\begin{aligned}
&\int_0^R \int_{\mathbb{R}^n} u_\vep^2 + |\nabla u_\vep|^2 + F_\vep(u_\vep) \dx\rd t \leq C(u_\circ,R), \\
& \max_{t \in [0,R]} \int_{\mathbb{R}^n} u_\vep^2(t) \dx \leq C(u_\circ,R).
\end{aligned}
\end{equation}
If, in addition, $u_\circ \in L^\infty(\RR^n)$, then $u_\vep \in L^\infty(\RR^n)$ and
\begin{equation}\label{eq:LinftyBoundEps}
\|u_\vep\|_{L^\infty(Q)} \le \|u_\circ\|_{L^\infty(\RR^n)}.
\end{equation}
\end{prop}
\begin{rem}\label{rem:InteriorReg+Inner}
We remark that, since $f_\vep \in C^{1+\gamma}$ for every $\vep > 0$ fixed, every bounded weak solution $u_\vep$ to \eqref{eq:ProbEps} in the sense of \eqref{eq:FirstOVvep} is actually a classical solution in $Q$: $u_\vep$ is of class $\HH_{2+\alpha}$ in $Q$ for some $\alpha \in (0,1]$ (see \cite[Chapter IV, Section 1]{Lieberman1996:book} for the definition of the spaces $\HH_{k+\alpha}$ and \cite[Theorem 4.8 and Theorem 4.9]{Lieberman1996:book} for the regularity estimates). As a consequence, we may test \eqref{eq:FirstOVvep} with $\varphi := \nb u_\vep \cdot\Phi$ and check that each $u_\vep$ satisfies
\begin{equation}\label{eq:FirstIVvep}
\int_Q \big( |\nabla u_\vep|^2 + 2H_\vep(u_\vep)u_\vep^\gamma \big) \dv_x \Phi  -  2\nb u_\vep \cdot D_x \Phi \cdot\nabla u_\vep - 2\partial_tu_\vep \,(\nb u_\vep \cdot \Phi) = 0,
\end{equation}
where $\Phi \in C_c^\infty(Q;\RR^{n+1})$ is arbitrarily fixed. As common in the elliptic theory (see also \cite{Weiss1999:art,art:Weiss2003} for the parabolic setting), we call \eqref{eq:FirstIVvep} the \emph{weak formulation in the sense of domain variations} of the equation in \eqref{eq:ProbEps}.
\end{rem}
We next show continuity of $u_\vep$ and $\nb u_\vep$ up to $t=0$, needed in the proof of Lemma \ref{thm:OptRegSpace}.
\begin{lem}\label{lem:BdryEstimates}
Let $\gamma \in [0,1]$ and $\alpha \in (0,1)$, let $u_\circ \in C^{2+\alpha}_c(\RR^n)$ be nonnegative and let $\{u_\vep\}_{\vep > 0}$ be a family of weak solutions to \eqref{eq:ProbEps} as in  \Cref{prop:LimDelta}. Then, for every $\vep > 0$, both $u_\vep$ and $\nb u_\vep$ can be continuously extended up to $t=0$ by setting $u_\vep|_{t=0} = u_\circ$ and $\nb u_\vep|_{t=0} = \nb u_\circ$.    
\end{lem}
\begin{proof} Let $\vep > 0$ and $u_\vep$ as in the statement. We divide the proof in two main steps as follows, although the strategy in both cases consists of combining standard regularity results with a barrier for $|u-u_\circ|$ or $|\nabla u - \nabla u_\circ|$ of the form $t^\alpha$ for some $\alpha>0$.

\

\emph{Step 1: we show that $u_\vep \in C(\overline{Q})$.} Let us set $\tilde{u} := u_\vep - u_\circ$. Then, $\tilde{u}(t)$ satisfies $\tilde{u} \to 0$ in $L^2(\RR^n)$ as $t \downarrow 0$ and
\[
\int_Q \partial_t \tilde{u} \varphi + \nabla \tilde{u} \cdot\nabla\varphi = \int_Q g \varphi,
\]
for every $\varphi \in C_c^\infty(Q)$, where $g := \Delta u_\circ - f_\vep(\tilde{u}+u_\circ)$. Notice that $g \in L^\infty(Q)$ since $u_\vep$ and $f_\vep$ are so and $u_\circ \in C^{2+\alpha}_c(\RR^n)$. 
Now let $w := \|g\|_{L^\infty(Q)}t$ and $v := \tilde{u}-w$. 
First, it is not difficult to check that $v$ satisfies    
\[
\int_Q \partial_t v \varphi + \nabla v \cdot\nabla\varphi \le 0,
\]
for every nonnegative $\varphi \in C_c^\infty(Q)$.
Furthermore, since $w \ge 0$ and $u_\vep \in L^2((0,R): H^1(\RR^n)) \cap C([0,R]:L^2(\RR^n))$ for every $R>0$ with $\partial_tu_\vep \in L^2(Q)$, the same regularity holds for $v_+ = \max\{v,0\}$. 
Therefore, given $0 < s < \tau$ and $\psi \in W^{1,\infty}(\RR^n)$ with compact support, a standard approximation argument (see for instance \cite[Theorem 6.1]{Lieberman1996:book}) shows that we can test the above inequality with $\varphi := \chi_{(s,\tau)}(t)\psi^2(x)v_+$ to obtain 
\[
 \int_s^{\tau}\int_{\RR^n} \frac{1}{2} \partial_t (v_+^2) \psi^2 + |\nabla v_+|^2 \psi^2 + 2(\psi \nabla v_+)\cdot(v_+ \nb\psi)  \, \rd x \rd t \le 0,
\]
and thus, integrating by parts in time, 
\[
\int_{\RR^n} v_+^2(\tau) \psi^2 - \int_{\RR^n} v_+^2(s) \psi^2 + 4\int_s^{\tau}\int_{\RR^n}(\psi \nabla v_+)\cdot(v_+ \nb\psi) \, \rd x \rd t \le 0
\]
Now, for $\sigma \in (0,1)$, take $\psi = \eta_\sigma$, where $\eta_\sigma(x) := \eta(\sigma x)$ and $\eta(x) := \min\{1,(2-|x|)_+\}$. Noticing that $\eta_\sigma \to 1$ and $|\nb \eta_\sigma| \to 0$ locally uniformly in $\RR^n$ as $\sigma \downarrow 0$, we may pass to the limit as $\sigma \downarrow 0$ into the above inequality to deduce, by dominated convergence, that
\[
\int_{\RR^n} v_+^2(\tau) \, \rd x \leq \int_{\RR^n} v_+^2(s) \, \rd x.
\]
Letting $s\downarrow 0$ and using that $v(s)\to 0$ in $L^2(\RR^n)$ as $s \downarrow 0$, it follows that $v_+ = 0$ a.e. in $Q$ (and thus, by regularity, everywhere in $Q$). In particular, we deduce $u_\vep - u_\circ \leq \|g\|_{L^\infty(Q)}t$ in $Q$. The same argument, applied to $\tilde{u} := u_\circ - u_\vep$, $g := f_\vep(u_\circ-\tilde{u}) - \Delta u_\circ$, $w:= \|g\|_{L^\infty(Q)}t$ and $v := \tilde{u}-w$, shows that $u_\circ - u_\vep \leq \|g\|_{L^\infty(Q)}t$ in $Q$ and therefore $|u_\vep - u_\circ| \leq C_\vep t$ in $Q$ for some $C_\vep > 0$. 
This, combined with the interior estimates and a standard boundary regularity argument, shows that $u_\vep$ is of class $\HH_1$ up to $t = 0$ (again, see \cite[Chapter IV, Section 1]{Lieberman1996:book} for the definition of the space $\HH_1$). In particular, this implies that $u_\vep$ can be continuously extended up to $t = 0$, that is, $u_\vep \in C(\overline{Q})$. 

\

\emph{Step 2: we show that $\nb u_\vep \in C(\overline{Q})$.} As in the step above, we consider $\tilde{u} := u_\vep - u_\circ$ and $g := \Delta u_\circ - f_\vep(\tilde{u}+u_\circ)$. Since $u_\circ \in C^{2+\alpha}_c(\RR^n)$, $u_\vep$ is of class $\HH_1$ up to $t=0$ and $f_\vep \in C^{1+\gamma}$, $g$ is locally H\"older continuous up to $t=0$ and thus, by \cite[Theorem 10 and Theorem 16, Chapter 1]{Fried1964:book}, it follows that
\begin{equation}\label{eq:Gaussian}
\tilde{u}(x,t) = \int_0^t\int_{\RR^n} G(x-y,t-s)g(y,s) \,\rd y\, \rd s, \quad \text {where} \quad 
G(x,t) = \frac{1}{(4\pi t)^\frac{n}{2}} e^{-\frac{|x|^2}{4t}}
\end{equation}
is the fundamental solution to the heat equation. Then, for every $i \in \{1,\ldots,n\}$, we have
\[
\begin{aligned}
|\partial_{x_i} \tilde{u}(x,t)| &\leq \int_0^t\int_{\RR^n} |\partial_{x_i}G(x-y,t-s)g(y,s)| \rd y\, \rd s \leq \|g\|_{L^\infty(Q)} \int_0^t\int_{\RR^n} \frac{|y|}{s} G(y,s) \, \rd y\, \rd s \\
&= \|g\|_{L^\infty(Q)} \int_0^t s^{-\frac{1}{2}} \Big(\int_{\RR^n} |x| G(x,1) \, \rd x \Big)\rd s \leq C\|g\|_{L^\infty(Q)} \, t^\frac{1}{2},
\end{aligned}
\]
for some $C > 0$ depending only on $n$. Therefore, as in \emph{Step 1}, $\partial_{x_i} \tilde{u}$ can be continuously extended up to $t=0$ by setting $\partial_{x_i} \tilde{u}|_{t=0} = 0$. Our claim follows since $i \in \{1,\ldots,n\}$ is arbitrary and $\partial_{x_i}u_\vep = \partial_{x_i} \tilde{u} + \partial_{x_i}u_\circ$.  
\end{proof}
\begin{rem}\label{rem:ExpDecay}
By \Cref{rem:InteriorReg+Inner} and \Cref{lem:BdryEstimates}, we obtain decay estimates for $u_\vep$ and its derivatives, under suitable assumptions on $u_\circ$. More precisely, if $u_\circ \in C_c^{2+\alpha}(\RR^n)$, then
\begin{equation}\label{eq:GaussianBounduEps}
u_\vep(x,t) \leq M G(x,t+T) \quad \text{ in } Q,    
\end{equation}
where $G$ is the Gaussian defined in \eqref{eq:Gaussian} and $M,T > 0$ are chosen such that $u_\circ(x) \leq MG(x,T)$ in $\RR^n$. The bound \eqref{eq:GaussianBounduEps} is obtained by comparison since $f_\vep \geq 0$ (that is, $u_\vep$ is sub-caloric) and $u_\vep$ is continuous up to $t=0$.  Note also that, combining \eqref{eq:GaussianBounduEps} with the classical parabolic Schauder estimates and that $f_\vep(0) = 0$, one can easily prove that for every $t > 0$, $|\nb u_\vep(x,t)|$ also decays exponentially fast as $|x| \to \infty$ (this fact will be used later on in the proof of \Cref{thm:OptRegSpace}).
 We conclude the remark by noticing that, since $f_\vep$ is Lipschitz, a comparison principle holds in the class of nonnegative bounded solutions, and therefore the solutions $u_\vep$ that we build in \Cref{prop:LimDelta} are the unique solutions of \eqref{eq:ProbEps} in this class. In particular, if $u_\circ = 0$, then $u_\vep = 0$ for every $\vep > 0$; thus, the assumption $u_\circ$ nontrivial is natural to obtain meaningful results.
\end{rem}
%
%
%
%
%
%
%
\section{Uniform regularity estimates and non-degeneracy}\label{sec:OptRegNonDeg}
In this section, we establish some optimal regularity, optimal growth, and non-degeneracy estimates for weak solutions to the problem \eqref{eq:ProbEps}. 
Our main interest is to obtain bounds which are either uniform in $\vep$, or in which we can track explicitly the dependence on $\vep$ ---taking into account that we will take the limit $\vep\downarrow0 $ in the next section. We first recall how to establish an optimal regularity estimate in space independent of $\vep\in (0,1)$.
For this, we follow \cite[Lemma 2]{Phillips1987:art}.
\begin{lem}[\cite{Phillips1987:art}, Lemma 2]
\label{thm:OptRegSpace}  
Let $\gamma \in [0,1]$ and let $u_\circ \in C_c^{2+\alpha}(\RR^n)$ be nonnegative. 
Then there exists $C_\circ > 0$, depending only on  $n$, $\gamma$, $\|u_\circ\|_{L^\infty(\RR^n)}$, and $\|D^2 u_\circ\|_{L^\infty(\RR^n)}$,
such that for every $\vep \in (0,1)$ and every nonnegative weak solution $u_\vep$ to \eqref{eq:ProbEps} given by \Cref{prop:LimDelta}, we have
\begin{equation}\label{eq:OptRegSpace}
\sup_Q \, u_\vep^{2/\beta} + \sup_Q \, |\nabla (u_\vep^{1/\beta})|^2 \leq C_\circ.
\end{equation}
\end{lem}
\begin{proof} 
Since the $L^\infty$ estimate follows from \eqref{eq:LinftyBoundEps}, it suffices to prove the gradient bound.
For this, fix $\vep > 0$, $M := (\|D^2 u_\circ\|_\infty + 1)^2$, set $u := u_\vep$ and define
\[
\psi := |\nb u|^2 - 2F_\vep (u) - Mu.
\]
We will prove that $\psi \leq 0$ in $\overline{Q}$. 
Once this is established, a straightforward computation combined with \eqref{eq:LinftyBoundEps} shows that
\[
|\nb (u^{1/\beta})|^2 = \frac{1}{\beta^2} \frac{|\nb u|^2}{u^\gamma} \leq \frac{1}{\beta^2} \big[ 2H_\vep(u) + M u^{1-\gamma} \big] \leq \frac{1}{\beta^2} \big[ 2 + M \|u_\circ\|_{L^\infty(\RR^n)}^{1-\gamma} \big] \quad \text{ in } Q,
\]
and our statement follows.

It remains to show that $\psi \leq 0$ in $\overline{Q}$. 
This will follow from a maximum principle argument using that, since $u$ is a classical solution to $\partial_t u - \Delta u = - f_\vep(u)$ in $Q$ (see \Cref{rem:InteriorReg+Inner}), the function $\psi$ satisfies
\[
\partial_t\psi - \Delta \psi = - 2\sum_{i=1}^n |\nb u_i|^2 + 2 f_\vep(u)^2 + M f_\vep(u) \quad \text{in } Q, 
\]
where $u_i := \partial_{x_i} u$.
First, a general fact about nonnegative $C^2$ functions (see, for example, \cite[Lemma~1]{Phillips1987:art}) yields that $|\nb u_\circ(x)|^2 \leq M u_\circ(x)$ for $x\in \RR^n$, and thus since both $u$ and $\nabla u$ are continuous up to $t=0$ (see  \Cref{lem:BdryEstimates}), we deduce $\psi|_{t=0} \leq 0$.
Now, assume by contradiction that $\psi > 0$ somewhere in $Q$.
Then, since $\psi$ is continuous and, for each $t>0$, $\psi(\cdot, t)$ decays to zero at infinity by virtue of \Cref{rem:ExpDecay}, it follows that given $T>0$ large enough the function $\psi$ achieves a positive maximum in the strip $\RR^n \times [0, T]$ at some $(x_0,t_0) \in Q$ with $t_0 \in (0, T]$.
In particular, we have $|\nb u(x_0,t_0)| > 0$ according to the definition of $\psi$.
Thus, if we differentiate $\psi$ in the direction $e := \nb u(x_0,t_0)/ |\nb u(x_0,t_0)|$, at the maximum point $(x_0,t_0)$ we get
\[
0 = \partial_e \psi = 2 |\nb u| \partial_{ee}u - 2f_\vep(u)\partial_e u - M\partial_e u = |\nb u| (2 \partial_{ee}u - 2f(u) - M), 
\]
and thus
\[
\partial_{ee}u(x_0,t_0) = f_\vep(u(x_0,t_0)) + \frac{M}{2}. 
\]
Combining this with the equation of $\psi$ and the fact that $(x_0,t_0)$ is a maximum point for $\psi$, we deduce that, at $(x_0,t_0)$, 
\begin{equation}
    \begin{split}
        0 \geq \Delta \psi - \partial_t \psi &= 2\sum_{i=1}^n |\nb u_i|^2 - 2 f_\vep(u)^2 - M f_\vep(u) \geq 2 (\partial_{ee}u)^2 - 2 f_\vep(u)^2 - M f_\vep(u) \\
        &= 2\left(f_\vep(u) + \tfrac{M}{2}\right)^2 - 2 f_\vep(u)^2 - M f_\vep(u) = \tfrac{M^2}{2} + M f_\vep(u) > 0,
    \end{split}
\end{equation}
which is a contradiction.
\end{proof}

Note that the previous estimate is enough to obtain uniform Hölder bounds up to $t=0$, as proved in \cite[Lemma 4]{Phillips1987:art}. 

\begin{cor}[\cite{Phillips1987:art}, Lemma 4]\label{cor:UniformHolder}
Let $\gamma \in [0,1]$ and let $u_\circ \in C_c^{2+\alpha}(\RR^n)$ be nonnegative. 
Then, for every $\vep \in (0,1)$, every nonnegative weak solution $u_\vep$ to \eqref{eq:ProbEps} given by \Cref{prop:LimDelta}, and every compact set $\KK \subset \overline{Q}$, there exists $C > 0$, depending only on $\KK$ and the constant $C_\circ$ given in \Cref{thm:OptRegSpace}, such that
\begin{equation}\label{eq:BdrHolderEps}
\| u_\vep \|_{C^\frac{1}{3n}(\KK)} \leq C.
\end{equation}

In particular, for every $\nu \in (0,\frac{1}{3n})$, the sequence $\{u_\vep\}_{\vep>0}$ converges locally in $C^\nu(\overline{Q})$ to some function $u$ as $\vep \downarrow 0$, along a suitable subsequence.
\end{cor}

\begin{proof}
    Let $t> s\geq 0$, $x,y\in \RR^n$, and set $r := |x-y| + |t-s|^{\frac{1}{3n}}$. 
    First, note that there exists $x_\star\in B_r(x)$, depending on $s$ and $t$,
    such that
    \begin{equation}
        \dfrac{1}{|B_r|}\int_{B_r(x)} \left ( \int_s^t |\partial_t u_\vep (z, \tau)|^2 \, \rd \tau \right)  \rd z = \int_s^t |\partial_t u_\vep (x_\star, \tau)|^2 \, \rd \tau.
    \end{equation}
    Then, by the energy estimates \eqref{eq:EnBound1Eps} and using that $r^n \geq |t-s|^\frac{1}{3}$, we obtain that 
    \begin{equation}
        \int_s^t |\partial_t u_\vep (x_\star, \tau)|^2 \, \rd \tau \leq \dfrac{C_\star}{r^n} \leq C_\star |t-s|^{-\frac{1}{3}}
    \end{equation}
    for some constant $C_\star>0$ depending only on $n$, $u_\circ$, and $\gamma$ (and in particular independent of $s$ and $t$).
    Now, from this and the fact that, since by \eqref{eq:OptRegSpace} and \eqref{eq:LinftyBoundEps}, we have $|\nabla u_\vep | \leq \beta C_\circ^{1/2} \|u_\circ\|_\infty^{\gamma/2} =: C_1$ in~$Q$, we get
    \begin{equation}
    \begin{split}
        |u_\vep (x,t) - u_\vep (y,s)| 
        &\leq C_1|x - x_\star| + |u_\vep (x_\star,t) - u_\vep (x_\star,s)| + C_1|y - x_\star| \\
        & \leq 3C_1 r + \left |\int_s^t \partial_t u_\vep (x_\star, \tau) \, \rd \tau \right | \\
        &\leq 3C_1 r + |t-s|^{1/2} \left( \int_s^t |\partial_t u_\vep (x_\star, \tau)|^2 \, \rd \tau \right)^{1/2} \\
        &\leq 3C_1 r + C_\star^{1/2} |t-s|^{1/3} ,
    \end{split}
    \end{equation}
    from which a local $C^\frac{1}{3n}$ estimate follows for $x$ and $y$ close enough (and thus the stated estimate with the constant also depending on $\KK$ is obtained from an $L^\infty$ estimate as usual).
\end{proof}

\begin{rem}
    \label{rem:TransfLinftyRHS}
    A useful transformation in the following results will consist of considering the function $w_\vep:= u_\vep^{2/\beta} = u_\vep^{2 - \gamma}$.
    Indeed, it is not difficult to check that
    \begin{equation}\label{eq:EqwOG}
    \partial_t w_\vep - \Delta w_\vep = - g_\vep \quad \text{ in } Q,
    \end{equation}
    where 
    \begin{equation}\label{eq:DefRHSg}
    g_\vep := (2-\gamma) \left[ h_\vep(u_\vep)u_\vep + \gamma H_\vep(u_\vep) + (1-\gamma)\beta^2 |\nabla (u_\vep^{1/\beta})|^2 \right].
    \end{equation}
    Notice that since $\gamma \in [0,1]$, $g_\vep \geq 0$ and, by the definition of $h_\vep$ and $H_\vep$, we have
    \begin{equation}\label{eq:LinftygOG}
    \|g_\vep\|_{L^\infty(Q)} \leq (2-\gamma) \left[ \max_{[0,1]}h  + \gamma + (1-\gamma)\beta^2 C_\circ \right] =: K_\circ,
    \end{equation}
    where $C_\circ$ is as in \Cref{thm:OptRegSpace}.
    This uniform bound will be useful in the next results in this section. 
    Note also that from this and the uniform $L^\infty$ bound of $u_\vep$ \eqref{eq:LinftyBoundEps}, the classical Schauder parabolic estimates (see \cite[Theorem 4.8]{Lieberman1996:book}) yield uniform Hölder bounds for $w_\vep$ in any compact set contained in $Q$ for every Hölder exponent in $(0,1)$. 
\end{rem}
With the gradient estimate from \Cref{thm:OptRegSpace} at hand, we can now establish further uniform bounds on weak solutions to \eqref{eq:ProbEps}. It is important to stress that, differently from \Cref{thm:OptRegSpace}, the proofs below have a \emph{local} nature, in the sense that the estimates do not depend on the initial data $u_\circ$, but only on the constant $C_\circ > 0$ in the $C^1$ estimate \eqref{eq:OptRegSpace}. In other words, we will consider classical nonnegative solutions $u_\vep$ to 
\begin{equation}\label{eq:ProbEpsloc}
\partial_t u_\vep - \Delta u_\vep = - f_\vep(u_\vep) \quad \text{ in }  Q_1
\end{equation}
satisfying \eqref{eq:OptRegSpace} in $Q_1$ and we will show the aforementioned uniform bounds. We begin with an \emph{optimal growth} estimate: the proof combines the parabolic Harnack's inequality and a comparison argument.
\begin{lem}\label{lem:OptGr} 
Let $\gamma \in [0,1]$ and $\vartheta > 0$. 
Then there exists a constant $C > 0$ such that for every $\vep > 0 $, every nonnegative classical solution $u_\vep$ to~\eqref{eq:ProbEpsloc} satisfying the estimate \eqref{eq:OptRegSpace} in $Q_1$ for some $C_\circ >0$, every $r \in (0,\tfrac{1}{4})$, and every $(x_\circ,t_\circ) \in \{u_\vep \leq \vartheta\vep^\beta\}\cap Q_{1/2}$, we have
\begin{equation}\label{eq:OptGr}
\sup_{Q_r(x_\circ,t_\circ)} u_\vep \leq C(\vep^2 + r^2)^{\frac{1}{2-\gamma}}.
\end{equation}
The constant $C$ depends only on $n$, $\gamma$, $\vartheta$, $\max_{[0,1]}h$, and $C_\circ$.
\end{lem}
\begin{proof} 
Fix $\vep > 0$, $\vartheta > 0$, and set $u := u_\vep$. Let $(x_\circ,t_\circ) \in \{u \leq \vartheta\vep^\beta\}\cap Q_{1/2}$ as in the statement and consider $w := u^{2/\beta}$. 
As in \Cref{rem:TransfLinftyRHS}, since $u$ satisfies the $C^1$ estimate \eqref{eq:OptRegSpace} for some constant $C_\circ>0$, we have that $\| w \|_{L^\infty(Q_1)} \leq C_\circ $ and $\partial_t w - \Delta w = - g_\vep $ in $Q_1$, where $g_\vep$ is given by \eqref{eq:DefRHSg} and is uniformly bounded in $L^\infty(Q_1)$ by the constant $K_\circ$ given in \eqref{eq:LinftygOG}.
Now, the function $w^{(x_\circ,t_\circ)}(x,t) := w(x+x_\circ,t+t_\circ)$ satisfies the same equation as $w$ with right-hand side $\tilde{g}_\vep(x,t) := g_\vep(x+x_\circ,t+t_\circ)$ which, in turn, satisfies $\|\tilde{g}_\vep\|_{L^\infty(Q_1)} \leq K_\circ$ in light of \eqref{eq:LinftygOG}. Since the next part of the proof just uses the equation of $w^{(x_\circ,t_\circ)}$, the $L^\infty$ bound for $\tilde{g}_\vep$ and that $\tilde{g}_\vep \geq 0$, we may assume $(x_\circ,t_\circ) = (0,0)$ and recover the general case by translation. 

\

\emph{Step 1:} By \cite[Lemma~5.1]{Weiss1999:art} and \eqref{eq:LinftygOG}, we have the following Harnack inequality:
\begin{equation}\label{eq:HarnackIneqParbola}
\sup_{P_r^\delta} w \leq C_\delta ( w(0,0) + K_\circ r^2 ) \leq C_\delta ( \vartheta^{2/\beta}\vep^2 + K_\circ r^2 ) \leq K_\delta(\vep^2 + r^2),
\end{equation}
for every $\delta\in (0,1)$, $r \in (0,\tfrac{1}{4})$, where $C_\delta > 0$ depends only on $n$ and $\delta$ (while $K_\delta > 0$ also depends on $\vartheta$ and $K_\circ$), and
\begin{equation}\label{eq:ParabolaDeltaDef}
P_r^\delta := \{(x,t) \in Q_r^-: t < -\delta|x|^2\}.
\end{equation}
\emph{Step 2:} Thanks to \eqref{eq:HarnackIneqParbola}, we are left to bound $w$ in the set $Q_r \setminus P_r^\delta$. 
This will be done using a comparison argument and choosing $\delta$ appropriately.

Let $a := \tfrac{2}{b} \cdot \max\{K_\delta, 8 C_\circ \}$, $b := \tfrac{1}{2n}$
and consider
\[
\phi(x,t) := a (t + b|x|^2) + \tfrac{ab}{2} \vep^2,
\]
which is a caloric function by the definition of $b$. 
We want to show that $w \leq \phi$ in $Q_{1/4} \setminus P_{1/4}^\delta$.
On the one hand, if we choose $\delta := \tfrac{b}{2} \in (0,1)$, in the set $\partial_p Q_{1/4} \setminus P_{1/4}^\delta$ (where $t\geq -\delta |x|^2$ and $|x|^2 = 1/16$) we get
\begin{equation}
    \phi(x,t) \geq  a \big (-\tfrac{\delta}{16}  + \tfrac{b}{16} \big) = \tfrac{ab}{16}    \quad \text{ in } \partial_p Q_{1/2} \setminus  P_{1/2}^\delta.
\end{equation}
Thanks to the bound $\|w\|_{L^\infty(Q_1)} \leq C_\circ$ and the definition of~$a$, which gives $ C_\circ \leq ab/16$, we obtain $w \leq \phi$ in $\partial_p Q_{1/4} \setminus  P_{1/4}^\delta$. 
On the other hand, by the definition of $\phi$ and the previous choice of $\delta$, we have 
\begin{equation}\label{eq:OptGrowBdBelowSuper}
\phi(x,t)|_{t = -\delta|x|^2} = \tfrac{ab}{2} (\vep^2 + \rho^2) \quad \text{ in } \partial B_\rho,
\end{equation}
for every $\rho > 0$.
Since by \eqref{eq:HarnackIneqParbola} we have 
\begin{equation}
\sup_{x \in \partial B_\rho,\,t=-\delta \rho^2} w(x,t) \leq K_\delta (\vep^2 + \rho^2),
\end{equation}
for every $\rho \in (0,\tfrac{1}{4})$, the definition of $a$ gives $ab/2 \geq K_\delta$ and thus $w \leq \phi$ in $\partial P_{1/4}^\delta \cap \{t > -\delta/4\}$.
Since $w$ is sub-caloric, we deduce $w \leq \phi$ in $Q_{1/4} \setminus P_{1/4}^\delta$ by comparison.

Since $\phi \leq \frac{a(b+1)}{2}(\varepsilon^2 + r^2)$ in $Q_r \setminus P^\delta_r$ for every $r\in (0,\frac{1}{4})$, the bound $w \leq \phi$ combined with \eqref{eq:HarnackIneqParbola} gives
\begin{equation}
    \sup_{Q_r} w \leq C(\vep^2 + r^2), \quad \text{ where } C = \max \{K_\delta, a(b+1)/2\},
\end{equation}
which is exactly \eqref{eq:OptGr} written in terms of $w$, up to a translation.
\end{proof}
Now, we establish an optimal regularity estimate in time. The proof is a sort of generalization of \cite[Lemma 5.2]{Weiss1999:art}: a more careful analysis is needed to obtain bounds which are uniform in $\vep$.
\begin{lem}\label{lem:OptRegTime} 
Let $\gamma \in [0,1]$ and $\UU \subset \subset Q_1$ an open set. 
Then there exists a constant $C > 0$ such that for every $\vep > 0 $, every nonnegative classical solution $u_\vep$ to \eqref{eq:ProbEpsloc} satisfying the  estimate \eqref{eq:OptRegSpace} in $Q_1$ for some $C_\circ>0$, we have
\begin{equation}\label{eq:OptRegTime}
\sup_\UU \, |\partial_t (u_\vep^{2/\beta})| \leq C.
\end{equation}
The constant $C$ depends only on $n$, $\gamma$, $C_\circ$, and $\UU$.
\end{lem}
\begin{proof} 
Set $w_\vep := u_\vep^{2/\beta} = u_\vep^{2-\gamma}$ and assume, by contradiction, that there exist an open bounded set $\UU \subset \subset Q_1$ and two sequences $\vep_j \downarrow 0$ and $\{(x_j,t_j)\}_{j\in \NN} \subset \UU$ such that 
\begin{equation}
    \label{eq:ContradOptRegTime}
    |\partial_t w_j(x_j,t_j)| \to \infty, \quad \text{ as } j \uparrow \infty,
\end{equation}
where $w_j := w_{\vep_j}$ (we also set $u_j := u_{\vep_j}$). 
After taking a subsequence, we may assume that $(x_j,t_j) \to (x_\star, t_\star)\subset \overline{\UU}$ and that $w_j(x_j, t_j) \to w_\star \in [0, C_\circ]$.

\

\emph{Step 1: The case $w_\star>0$.} If $w_\star>0$, then, by \Cref{rem:TransfLinftyRHS}, the functions $w_j$ converge uniformly in $Q_R(x_\star, t_\star)$ for some $R\in (0,1)$ and thus, by the uniform convergence, we have
\begin{equation}\label{eq:UnifLower}
    w_j \geq w_\star/2 >0 \quad \text{ in } Q_r(x_\star, t_\star)
\end{equation}
for some $r<R$ and for $j$ large enough.
As a consequence, we have $u_j \geq (w_\star/2)^{\beta/2} >0 $ in $Q_r(x_\star, t_\star)$ for $j$ large enough.
From this we show that, for $j$ large enough, $\partial_t u_j$ is bounded in $\overline{Q_{r/4}(x_\star, t_\star))}$ uniformly in $j$, and this will be a contradiction with \eqref{eq:ContradOptRegTime} since $\partial_t w_j = \frac{2}{\beta} u_j^{(2-\beta)/2}\partial_t u_j$.

To do this, it suffices to apply the classical parabolic Schauder to the equation of $u_j$ in $Q_r(x_\star, t_\star)$ and using that, thanks to the uniform lower bound for $u_j$, the right-hand side of the equation is Hölder continuous, with bounded Hölder norm uniformly in $j$ (for $j$ large enough).
Indeed, since
\begin{equation}
    0 \leq f_{\vep_j} (u_j) = \vep_j^{-\beta} h(u_j \vep_j^{-\beta}) u_j^\gamma + \gamma H(u_j \vep_j^{-\beta}) u_j^{\gamma - 1},
\end{equation}
we have $u_j \vep_j^{-\beta}\geq 1$ for $j$ large enough and thus $f_{\vep_j} (u_j) = \gamma u_j^{\gamma - 1} \in L^\infty(Q_r(x_\star, t_\star))$ uniformly in $j$. Then, by \cite[Theorem 4.8]{Lieberman1996:book}, $u_j \in \HH_{1+\alpha}(Q_{r/2}(x_\star, t_\star))$ uniformly in $j$ and thus, since the function $s \to s^{\gamma - 1}$ is $C^\infty(0,\infty)$, \cite[Theorem 4.9]{Lieberman1996:book} yields $\partial_t u_j \in C(\overline{Q_{r/4}(x_\star, t_\star)})$ uniformly in $j$, as claimed. 

\

\emph{Step 2: The case $w_\star=0$.} Assume now $w_\star = 0$.
In this case the argument is essentially the same as before, but since $w_j(x_j, t_j) \to 0$, we consider blow-up sequences to obtain suitable lower bounds as in~\eqref{eq:UnifLower}. Set $r_j^2 := w_j(x_j,t_j)$ ---which satisfies $r_j \to 0$ as $j \uparrow \infty$--- and consider the blow-up sequence 
\[
W_j(x,t) := \frac{1}{r_j^2} w_j(x_j + r_jx,t_j + r_j^2t),
\]
which is well defined in $Q_1$ for $j$ large enough (and from now on we restrict ourselves to this subsequence of $j$).
Note that each $W_j$ satisfies $\partial_t W_j - \Delta W_j = \tilde{g}_j$ in $Q_1$ with $\tilde{g}_j(x,t) := g_{\vep_j}(x_j + r_jx,t_j + r_j^2t)$, where the function $g_\vep$ is defined in \eqref{eq:DefRHSg}. 
In particular, $\| \tilde{g}_j \|_{L^\infty(Q_1) }\leq K_\circ$, where $K_\circ$ is given by \eqref{eq:LinftygOG} and is independent of $j$.

We first claim that there are $c_\circ,r \in (0,1)$ such that, up to passing to a suitable subsequence, 
\begin{equation}\label{eq:BdBelOptRegTime}
W_j \geq c_\circ > 0 \quad \text{ in }  Q_r^-,
\end{equation}
for every $j$ large enough. Since $W_j(0,0)=1$, it is enough to show that $W_j$ converge uniformly near the origin. Note first that, by the estimate \eqref{eq:OptRegSpace}, the family $\{|\nb ( W_j^{1/2})|\}_{j\in\NN}$ is uniformly bounded in $L^\infty(Q_1)$ and thus
\[
W_j^{1/2}(x,0) =  \frac{1}{r_j} \left[ w_j^{1/2}(x_j + r_jx,t_j) - w_j^{1/2}(x_j,t_j) + w_j^{1/2}(x_j,t_j)   \right] \leq C_\circ( |x| + 1),
\]
for every $j$ and every $x\in B_1$, where $C_\circ > 0$ is as in \eqref{eq:OptRegSpace}. 
Hence, as in \eqref{eq:HarnackIneqParbola}, the Harnack inequality implies that for every $x\in B_{1/2}$ 
\begin{equation}
    \label{eq:OptRegTimeProofHarnack}
    \sup_{P^1_{1/4}(x,0)} W_j \leq C_1 ( W_j(x,0) + K ) \leq C_1 \left[ C_\circ^2 (|x| + 1)^2 + K_\circ \right],
\end{equation}
where $C_1 > 0$ and $K_\circ > 0$ are as in \eqref{eq:HarnackIneqParbola} and \eqref{eq:LinftygOG}, respectively, and $P^1_{1/4}(x,0)= \{(y,t) \in Q_{1/4}^-(x,0) : t < -|x-y|^2\}$. 
Since \eqref{eq:OptRegTimeProofHarnack} holds for all $x\in B_{1/2}$, we deduce that, for every $j$,
\begin{equation}
    \|W_j\|_{L^\infty(Q_{1/4}^-)} \leq C_2
\end{equation}
for some $C_2 > 0$ (depending only on $C_\circ$, $C_1$, and $K_\circ$) and so, by \cite[Theorem 4.8]{Lieberman1996:book} (recall that each $W_j$ satisfies $\partial_t W_j - \Delta W_j = \tilde{g}_j$ in $Q_1$ with $\tilde{g}_j$ bounded independently of $j$), the family $\{W_j\}_j$ is uniformly bounded in $C^\nu (\overline{Q_{1/8}^-})$ for some $\nu \in (0,1)$. 
As a consequence, there exists $W \in C(\overline{Q_{1/8}^-})$ such that $W_j \to W$ uniformly in $\overline{Q_{1/8}^-}$ (up to subsequence) and with $W(0,0)= 1$ by construction. 
From this, claim \eqref{eq:BdBelOptRegTime} follows.

\

Now, recall that  $u_j = u_{\vep_j}$ and consider the sequence 
\[
U_j(x,t) := W_j^{\beta/2}(x,t) = \frac{1}{r_j^\beta} u_j(x_j + r_jx,t_j + r_j^2t).
\]
Thanks to \eqref{eq:BdBelOptRegTime}, $U_j \geq c_\circ^{\beta/2}$ in $Q_r^-$ for every $j$ large enough. In addition, each $U_j$ satisfies
\[
\partial_t U_j - \Delta U_j = - f_{\rho_j}(U_j) \quad \text{ in } Q_1^-,
\]
where $\rho_j := \vep_j/r_j$, recall \eqref{eq:Scalingfeps}. As in the step above, we next show that $\partial_t U_j(0,0)$ is uniformly bounded (using standard parabolic estimates \cite[Theorem 4.8 and Theorem 4.9]{Lieberman1996:book}): again, the key step is to obtain uniform H\"older estimates of the right-hand side of the equation. 
To do so, we use that
\begin{equation}
    0\leq f_{\rho_j}(U_j) = \rho_j^{-\beta} h(U_j \rho_j^{-\beta}) U_j^\gamma + \gamma H(U_j \rho_j^{-\beta}) U_j^{\gamma-1}
\end{equation}
and distinguish three cases:

\

(i) $\rho_j \to 0$ as $j \uparrow \infty$ up to passing to a subsequence: We have $U_j \rho_j^{-\beta} \geq 1$ for $j$ large enough and thus $f_{\rho_j} (U_j)= \gamma U_j^{\gamma - 1}$ in $Q_r^-$. Hence, proceeding as in the first part of the proof, the lower bound of $U_j$ provides a uniform bound in $L^\infty(Q_r^-)$ of the right-hand side.

\

(ii) $\rho_j \to \rho_\star \in (0,\infty)$ as $j \uparrow \infty$ up to passing to a subsequence: to obtain an $L^\infty$ bound we proceed as in the previous case, since $c_\circ^{\beta/2} \leq U_j \leq C_2^{\beta/2} $. 
    Note here in addition that $U_j \rho^{-\beta}$ is uniformly bounded by above and below by positive constants.

\

(iii) $\rho_j \to \infty$ as $j \uparrow \infty$: In this case, the uniform bound for $U_j$ gives that $U_j \rho_j^{-\beta} \to 0$ uniformly in $Q_r^-$ and thus, since $h'(0) > 0$, $f_{\rho_j}(U_j) \sim (1 + \frac{\gamma}{2})\rho_j^{-2\beta} U_j^{1+\gamma}$ in $Q_r^-$ for $j$ large enough. In particular, $f_{\rho_j}(U_j)$ is uniformly bounded in $Q_r^-$.

\

The uniform bound in $L^\infty(Q_r^-)$ of the right-hand side yields Hölder regularity of $U_j$ in $Q_{r/2}^-$ uniform in $j$. 
Consequently, in each of the three cases one sees that $f_{\rho_j} (U_j)$ is of class $\HH_\alpha$ uniformly in $j$ (for $j$ large enough) for some $\alpha >0$ and thus $\partial_t U_j(0,0)$ is uniformly bounded. 
Therefore, since $U_j(0,0) = 1$ by definition, we have
\[
|\partial_t w_j(x_j,t_j)| = |\partial_t (u_j^{2/\beta})(x_j,t_j)|
=\frac{2}{\beta} U_j^{\frac{2-\beta}{\beta}}(0,0) \, |\partial_t U_j(0,0)| = \frac{2}{\beta} |\partial_t U_j(0,0)|\leq C,
\]
for some  $C > 0$ independent of $j$. This contradicts our assumption $|\partial_t w_j(x_j,t_j)| \to \infty$ in~\eqref{eq:ContradOptRegTime} and concludes the proof in the case $w_\star = 0$.
\end{proof}

Finally, we establish a non-degeneracy result in the range $\gamma \in (0,1]$ which, differently from \Cref{lem:OptGr} and  \Cref{lem:OptRegTime}, has a completely local nature (the estimate does depend on $C_\circ$). 
Note that in this case the result cannot hold in general when $\gamma = 0$, even in the elliptic setting, see \cite{art:Weiss2003,KrivenWeiss25:art} ---indeed, in the stationary case one needs the solution to be a minimizer or, at least, to be stable; see \cite{KamburovWang}.

\begin{lem}\label{lem:NonDeg} 
Let $\gamma \in (0,1]$ and $\vartheta > 0$. Then there exists $c > 0$, depending only on $n$, $\gamma$, and $\vartheta$, such that for every $\vep > 0$, every nonnegative classical solution $u_\vep$ to \eqref{eq:ProbEpsloc}, every $r \in (0,\tfrac{1}{4})$ and $(x_\circ,t_\circ) \in \{u_\vep \geq \vartheta\vep^\beta\}\cap Q_{\frac{1}{2}}$, we have
\begin{equation}\label{eq:NonDeg}
\sup_{Q_r^-(x_\circ,t_\circ)} u_\vep \geq c(\vep^2 + r^2)^{\frac{1}{2-\gamma}}.
\end{equation}
\end{lem}
\begin{proof} 
Fix $\vep > 0$, $\vartheta > 0$ and set $u := u_\vep$. Let $(x_\circ,t_\circ) \in \{u \geq \vartheta\vep^\beta\} \cap Q_{1/2}$ as in the statement and consider $w := u^{2/\beta}$. 
As in \Cref{rem:TransfLinftyRHS}, $w$ satisfies \eqref{eq:EqwOG}, where $g_\vep$ is defined in \eqref{eq:DefRHSg}. 
In particular, by the definition of $g_\vep$ and since $h_\vep\geq 0$, we have
\begin{equation}
    \partial_t w - \Delta w \leq - \gamma(2-\gamma) H_\vep(u) \quad \text{ in } Q_1,
\end{equation}
and, further, since $H_\vep$ is non-decreasing, it follows that
\begin{equation}
    \partial_t w - \Delta w \leq - \gamma(2-\gamma) H(\vartheta) \quad \text{ in } \{u > \vartheta\vep^\beta\} \cap Q_1.
\end{equation}
As we did in \Cref{lem:OptGr}, we may assume $(x_\circ,t_\circ) = (0,0)$ and recover the general case by translation. 

Let $\{(x_k,t_k)\}_{k\in \NN} \subset \{u > \vartheta\vep^\beta\}\cap Q_{1/2}$ be such that $(x_k,t_k) \to (0,0)$, and define
\begin{equation}
    \phi_k(x,t) := w(x,t) - w(x_k,t_k) - c_0 \big( |x - x_k|^2 + t_k - t \big), \qquad \text{ with } \quad c_0 := \frac{\gamma(2-\gamma) H(\vartheta)}{2n+1}.
\end{equation}
Then, by the definition of $\phi_k$ and $c_0$, 
\begin{equation}
    \partial_t \phi_k - \Delta \phi_k = - \gamma(2-\gamma) H(\vartheta) + c_0(2n + 1) \leq 0 \quad \text{ in } \{u > \vartheta\vep^\beta\} \cap Q_1,
\end{equation}
and, furthermore, $\phi_k(x_k,t_k) = 0$.
Consequently, by the maximum principle, for $r\in (0,\frac{1}{4})$ we have
\begin{equation}
    0=  \phi_k(x_k, t_k) \leq \sup_{Q_r^-(x_k,t_k) \cap \{u>\vartheta \vep^\beta\}} \phi_k = \sup_{\partial_p (Q_r^-(x_k,t_k) \cap \{u>\vartheta \vep^\beta\})} \phi_k ,
\end{equation}
where $\partial_p \Omega$ denotes the parabolic boundary of a set $\Omega \subset \RR^{n+1}$. 
Since $\phi_k < 0$ in $\partial\{ u > \vartheta \vep^\beta\} \cap Q_r^-(x_k,t_k)$, we obtain
\begin{equation}
    0 \leq \sup_{\partial_p Q_r^-(x_k,t_k)\cap \{u>\vartheta \vep^\beta\} } \phi_k \leq \sup_{\partial_p Q_r^-(x_k,t_k) } \phi_k.
\end{equation}
Now, since $\phi_k \leq w - w(x_k,t_k) - c_0 r^2$ in $\partial_p Q_r^-(x_k,t_k)$, it follows that
\begin{equation}
    0 \leq \sup_{\partial_p Q_r^-(x_k,t_k)} \phi_k \leq \sup_{\partial_p Q_r^-(x_k,t_k)} w  - w(x_k,t_k) - c_0 r^2\leq \sup_{Q_r^-(x_k,t_k)} w  - w(x_k,t_k) - c_0 r^2,
\end{equation}
and, recalling that $\{(x_k,t_k)\}_{k\in \NN} \subset \{u > \vartheta\vep^\beta\} = \{ w^{\beta/2} > \vartheta\vep^\beta\}$, we obtain
\[
\sup_{Q_r^-(x_k,t_k)} w \geq \vartheta^{2/\beta}\vep^2  + c_0 r^2 \geq c(\vep^2 + r^2),
\]
where $c := \min\{\vartheta^{2/\beta},c_0\}$.
Passing to the limit as $k \uparrow \infty$, we deduce \eqref{eq:NonDeg} written in terms of $w$, up to a translation.
\end{proof}

Note that the previous proof fails for $\gamma = 0$ (as we mentioned) since then $c_0 = 0$.
Note also that the assumption $h'(0)>0$  it is crucial to guarantee $c_0>0$.
%
%
%
%
%
%
%
\section{Convergence to the free boundary problem}\label{sec:convergence}

In this section, we address the study of the problem one obtains in \eqref{eq:ProbEps} when taking $\varepsilon\downarrow0$.
First, we will pass to the limit using our uniform regularity estimates for solutions $u_\vep$ to obtain a limit function $u$ satisfying the expected regularity and energy bounds.
Then, we will show appropriate convergence of the sets $\{u\geq \vep^\beta\}$ and $\{u\leq \vep^\beta\}$ towards $\overline{\{u>0\}}$ and $\overline{\{u=0\}}$ respectively, as well as non-degeneracy and optimal growth estimates.
After this, we will need to obtain a finer control of the free boundary $\partial \{u>0\}$ and study the convergence of $H_\vep(u_\vep)$ towards $\chi_{\{u>0\}}$.
Even more crucial, we show the convergence of $f_\vep(u_\vep)$ towards $\gamma u_+^{\gamma - 1}$. This delicate step will require fine barriers in the set $\{ u \leq \vep^\beta\}$ (see~\Cref{rem:ExpDecay} below).
Finally, after all this analysis is done, we will show that $u$ is a solution of the FB problem, in the sense of \cref{def:WeakSolPIntro}.
In particular, this will allow us to obtain the FB condition.
\subsection{Convergence of the solutions} We begin with the following lemma.
\begin{lem}\label{lem:StrCmp}
Let $\gamma \in (0,1]$, $\alpha \in (0,1)$ and $\nu \in (0,\frac{1}{3n})$. Let $u_\circ \in C_c^{2+\alpha}(\RR^n)$ be nonnegative and let $\{u_\vep\}_{\vep > 0}$ be a family of nonnegative weak solutions to \eqref{eq:ProbEps} as in  \Cref{prop:LimDelta}. Then, there exist a nonnegative function $u \in \UU_\circ \cap C_{\loc}^{1/(3n)}(\overline{Q})$ and a sequence $\vep_j \downarrow 0$ such that $u_{\vep_j} \to u$ locally in $C^\nu(\overline{Q})$ and weakly in $\UU_\circ$ as $j \uparrow \infty$. Furthermore:

\

$\bullet$ \textbf{Energy estimates:} We have
\begin{equation}\label{eq:EnBound1Lim}
\int_Q |\partial_t u|^2  \leq \|u_\circ\|_{H^1(\RR^n)}^2 + 2\|u_\circ^\gamma\|_{L^1(\RR^n)},
\end{equation}
and
\begin{equation}\label{eq:EnBound2Lim}
\int_0^R \int_{\mathbb{R}^n} ( u^2 + |\nabla u|^2 ) \, \rd x \rd t + \max_{t \in [0,R]} \int_{\mathbb{R}^n} u^2(t) \, \rd x \leq 2C(u_\circ,R), 
\end{equation}
where $C(u_\circ,R) > 0$ is as in \Cref{prop:LimDelta}.

\ 

$\bullet$ \textbf{Regularity estimates:} We have
\begin{equation}\label{eq:OptRegSpaceLim}
\sup_Q \, u^{2/\beta} + \sup_Q \, |\nabla (u^{1/\beta})|^2 \leq C_\circ,
\end{equation}
where $C_\circ > 0$ is as in \Cref{thm:OptRegSpace}. Furthermore, for every open bounded set $\UU \subset\subset Q$, there exists $C > 0$ as in  \Cref{lem:OptRegTime} such that
\begin{equation}\label{eq:OptRegTimeLim}
\sup_\UU \, |\partial_t (u^{2/\beta})| \leq C.
\end{equation}
Finally, for every compact set $\KK \subset \overline{Q}$, there exists $C > 0$ as in \Cref{cor:UniformHolder} such that
\begin{equation}\label{eq:BdrHolderLim}
\| u \|_{C^\frac{1}{3n}(\KK)} \leq C.
\end{equation}
\end{lem} 
\begin{proof} Let $\{u_\vep\}_{\vep > 0}$ be a family of nonnegative weak solutions to \eqref{eq:ProbEps} as in the statement and let $R > 0$. By \Cref{prop:LimDelta}, we deduce the existence of $u \in L^2((0,R): H^1(\RR^n))\cap C([0,R]:L^2(\RR^n))$ with $\partial_tu \in L^2(Q)$ and a sequence $\vep_j \downarrow 0$ such that, setting $u_j := u_{\vep_j}$, we have $u_j \rightharpoonup u$ in $L^2((0,R): H^1(\RR^n))$ and $\partial_t u_j \rightharpoonup \partial_t u$ in $L^2(Q)$ as $j \uparrow \infty$. Thus, by the Aubin-Lions lemma (see, for example, \cite{Simon1987:art}), we have $u_j \to u$ in $C([0,R]:L^2(\RR^n))$ as $j \to \infty$ (through a subsequence), and thus a standard diagonal argument shows that $u \in \UU_\circ$. The energy estimates \eqref{eq:EnBound1Lim} and \eqref{eq:EnBound2Lim} then follow by \eqref{eq:EnBound1Eps} and \eqref{eq:EnBound2Eps}, and lower semicontinuity of the $L^2$ norm under weak convergence.

Now, combining the Hölder estimate \eqref{eq:BdrHolderEps} with another diagonal argument, we obtain $u_j \to u$ locally uniformly in $\overline{Q}$, up to passing to another subsequence. Furthermore, from the uniform $C^1$ estimates \eqref{eq:OptRegSpace} and \eqref{eq:OptRegTime}, we deduce $u_j \rightharpoonup^\star u$ in $L^\infty(Q)$, $\nb (u_j^{1/\beta}) \rightharpoonup^\star \nb (u^{1/\beta})$ in $L^\infty(Q)^n$, and $\partial_t (u_j^{2/\beta}) \rightharpoonup^\star \partial_t (u^{2/\beta})$ in $L^\infty(\UU)$ as $j \uparrow \infty$ (up to passing to another subsequence), where $\UU \subset\subset Q$ is a fixed open bounded set: hence \eqref{eq:OptRegSpaceLim} and \eqref{eq:OptRegTimeLim} then follow by lower semicontinuity of the $L^\infty$ norm under weak-$\star$ convergence. Finally, \eqref{eq:BdrHolderLim} follows by \eqref{eq:BdrHolderEps} and uniform convergence.
\end{proof}
\begin{rem}
We anticipate here that the regularity estimates \eqref{eq:OptRegSpaceLim} and \eqref{eq:OptRegTimeLim} are optimal, see \Cref{Subsec:SpecialSolutions}.  
\end{rem}

\

In the following result, we establish the appropriate convergence of the sets $\{u_\vep \geq \vep^\beta \}$ and $\{u_\vep \leq \vep^\beta\}$, as well as the optimal growth and non-degeneracy estimates of the limit function $u$. 
\begin{lem}\label{lem:HausConv} 
Let $\gamma \in (0,1]$ and $\alpha \in (0,1)$. 
Let $u_\circ \in C_c^{2+\alpha}(\RR^n)$ be nonnegative and nontrivial, and let $\vep_j$, $u_{\vep_j}$, and $u$ as in \Cref{lem:StrCmp}. 
Then, for every $\vartheta > 0$.
\begin{equation}\label{eq:HausConv}
\{u_{\vep_j} \geq \vartheta\vep_j^\beta\} \to \overline{\{u > 0\}} \quad \text{ and } \quad \{u_{\vep_j} \leq \vartheta\vep_j^\beta\} \to \{u = 0\}
\end{equation}
locally Hausdorff in $Q$ as $j \to +\infty$. Furthermore, there exist two  constants $C >c > 0$, where $C$ depends only on $n$, $\gamma$, $\max_{[0,1]}h$, and $C_\circ$ (where $C_\circ > 0$ is as in \Cref{thm:OptRegSpace}), and where $c$ depends only on $n$, $\gamma$, and $H$, such that: 

\

$\bullet$ \textbf{Optimal growth:}
    for every $r \in (0,\tfrac{1}{4})$ and every $(x_\circ,t_\circ) \in \{u = 0\}$ such that $Q_{4r}(x_\circ,t_\circ) \subset\subset Q$, we have
    \begin{equation}\label{eq:OptGrLimLem} 
    \sup_{Q_r(x_\circ,t_\circ)} u \leq C r^\beta.
    \end{equation}

\

$\bullet$ \textbf{Non-degeneracy:} for every $r\in (0,\frac{1}{4})$ and for every $(x_\circ,t_\circ) \in \overline{\{u > 0\}}$ such that $Q_{4r}(x_\circ,t_\circ) \subset\subset Q$, we have
    \begin{equation}\label{eq:NonDegLimLem} 
    \sup_{Q_r^-(x_\circ,t_\circ)} u \geq c r^\beta.
    \end{equation}
In particular, $u$ is nontrivial.
\end{lem} 
\begin{proof} Fix $\vartheta > 0$, a compact set $\KK \subset\subset Q$, and set  $u_j := u_{\vep_j}$. Define
\begin{equation}
    U_j := \{ u_j \geq \vartheta \varepsilon_j^\beta \} \cap \KK,
    \quad \text{ and } \quad 
    U := \overline{\{ u > 0 \}} \cap \KK,
\end{equation} 
and for a given $\sigma \in (0,1)$ we denote $\sigma$-neighborhoods of these sets by
\begin{equation}
    U_{j,\sigma} := \{(x,t): \text{dist}((x,t),U_j) \leq \sigma \},
    \quad \text{ and } \quad 
    U_\sigma := \{(x,t): \text{dist}((x,t),U) \leq \sigma \}.
\end{equation}

\ 

\emph{Step 1: Non-degeneracy.} Let us first prove the non-degeneracy estimate \eqref{eq:NonDegLimLem}. 
For this, consider a point $(x_\circ,t_\circ) \in \overline{\{u > 0\}}$ such that $Q_1(x_\circ,t_\circ) \subset\subset Q$, and let $\{(x_k, t_k)\}_k \subset \{u > 0\}$ be such that $(x_k,t_k) \to (x_\circ,t_\circ)$ as $k \to +\infty$. 
Fix $k \in \NN$ and notice that, since $u(x_k,t_k) > 0$, by uniform convergence we have $u_j(x_k,t_k) \geq \vartheta \vep_j^\beta$ for every $j$ large enough.
Consequently, by the non-degeneracy estimate of \Cref{lem:NonDeg}, for each $j$ large enough and every $r\in (0,\frac{1}{4})$, 
\begin{equation}
    u_j(y_j,\tau_j) \geq c(\vep_j^2 + r^2)^\frac{1}{2-\gamma}, 
\end{equation}
for some $(y_j,\tau_j) \in \overline{Q_r^-(x_k,t_k)}$, where $c>0$ depends only on $n$, $\gamma$, and $\vartheta$. Up to passing to a subsequence, $(y_j,\tau_j) \to (y_\star,\tau_\star) \in \overline{Q_r^-(x_k,t_k)}$ and thus, by uniform convergence, we have $u(y_\star,\tau_\star) \geq c r^\beta$ which, in turn, implies 
\begin{equation}
    \sup_{Q_r^-(x_k,t_k)} u \geq c r^\beta.
\end{equation}
Passing to the limit as $k \to +\infty$, \eqref{eq:NonDegLimLem} follows for $(x_\circ,t_\circ) \in \overline{\{u > 0\}}$ such that $Q_1(x_\circ,t_\circ) \subset\subset Q$.
To get the result also for points with $t_\circ \in (0,1]$, we proceed in the same way, simply noticing that the proof of \Cref{lem:NonDeg} works also if we require $r$ to be small enough so that $Q_{4r}(x_\circ,t_\circ) \subset\subset Q$.

\

\emph{Step 2: Hausdorff convergence.} With the previous estimate at hand, we can prove the local Hausdorff convergence stated in \eqref{eq:HausConv}.
For this, it suffices to show that the sets $U_j$ converge to $U$ in the Hausdorff sense. 
Indeed, since both $U_j$ and $U$ are subsets of a compact set $\KK$, it follows from the definition of Hausdorff convergence that $U_j\to U$ yields $\{u_j \leq \vartheta \vep_j^\beta\} \cap \KK  \to \{u = 0\} \cap \KK $ in the Hausdorff sense, and since $\KK$ is arbitrary we obtain~\eqref{eq:HausConv}. To establish the Hausdorff convergence of $U_j$ to $U$, we need to prove the next two statements:

\

(i) $U_j \subset U_\sigma$ for $j$ large enough (and $\sigma$ small enough). Assume by contradiction that there is a sequence $(x_j,t_j) \in U_j$, but $(x_j,t_j) \not \in U_\sigma$. 
    Then, by the non-degeneracy estimate~\eqref{eq:NonDeg} in \Cref{lem:NonDeg} applied to each $u_j$ (taking $\sigma$ small enough if needed),  there exists $(y_j,\tau_j) \in \overline{Q_{\sigma/2}(x_j,t_j)}$ such that
    \begin{equation}
        u_j(y_j,\tau_j) := \sup_{Q_{\sigma/2}(x_j,t_j)} u_j \geq c(\vep_j^2 + \sigma^2)^{\frac{1}{2-\gamma}} \geq c \sigma^\beta,
    \end{equation}
    for every $j$ and some $c > 0$ independent of $j$.
    Up to passing to a subsequence, we may assume $(x_j,t_j) \to (x_\star,t_\star) \in \KK$, $(y_j,\tau_j) \to (y_\star,\tau_\star)  \in \overline{Q_{\sigma/2}(x_\star,t_\star)}$, and $u_j(y_j,\tau_j) \to u(y_\star,\tau_\star)$ as $j \uparrow \infty$, by uniform convergence. 
    By construction, $\dist((x_\star, t_\star), U) \geq \sigma$ and, therefore, $(y_\star,\tau_\star)\not \in U = \overline{\{u>0\}}\cap \KK$.
    Thus, we have $u(y_\star,\tau_\star) = 0$, which contradicts the inequality above.

\

(ii) $U \subset U_{j,\sigma}$ for $j$ large enough (and $\sigma$ small enough). Assume by contradiction that there is a sequence $(x_j,t_j) \in U$, but $(x_j,t_j) \not\in U_{j,\sigma}$.
    Then, by the non-degeneracy estimate \eqref{eq:NonDegLimLem} proved above,  we have
    \begin{equation}
        u(y_j,\tau_j) \geq c \sigma^\beta,        
    \end{equation}
    for some $(y_j,\tau_j) \in \overline{Q_{\sigma/2}(x_j,t_j)}$ and for some $c > 0$ independent of $j$ (taking $\sigma$ small enough if needed).
    However, by construction, $u_j < \vartheta \varepsilon_j^\beta$ in $\overline{Q_{\sigma/2}(x_j,t_j)}$, which contradicts the inequality above for $j$ large enough.

\ 

\emph{Step 3: Optimal growth.} 
We are left to show the optimal growth estimate \eqref{eq:OptGrLimLem}. 
To do so, let us fix $(x_\circ,t_\circ) \in \{u = 0\}$ with $t_\circ > 1$. 
By local Hausdorff convergence, it is not difficult to check that there is a sequence $(x_j,t_j) \in \{u_j \leq \vartheta \vep_j^\beta\}$ such that $(x_j,t_j) \to (x_\circ ,t_\circ)$ as $j \to +\infty$. 
Consequently, the optimal growth estimate \eqref{eq:OptGr} in \Cref{lem:OptGr} for $u_j$ yields
\[
\sup_{Q_r(x_j,t_j)} u_j \leq C(\vep_j^2 + r^2)^{\frac{1}{2-\gamma}},
\]
for every $r \in (0, \frac{1}{4})$ and some $C > 0$ independent of $j$. Then, \eqref{eq:OptGrLimLem} follows (for points $(x_\circ,t_\circ) \in \{u = 0\}$ with $t_\circ > 1$) by taking the limit as $j \to \infty$ and uniform convergence.
As before, to get the result also for points with $t_\circ \in (0,1]$, we proceed in the same way, simply noticing that the proof of \Cref{lem:OptGr} works also if we require $r$ to be small enough so that $Q_{4r}(x_\circ,t_\circ) \subset\subset Q$.
\end{proof}

\subsection{Convergence of the nonlinearity}
Our next goal is passing to the limit in the weak formulations of the equation for $u_\vep$, to derive the equations of $u$ (that is, \eqref{eq:FirstOVLimIntro} and \eqref{eq:FirstIVLimIntro}).
For this, we first have to pay special attention to the convergence of $H_{\vep_j}(u_j)$ towards $ \chi_{\{u > 0\}}$, which is the content of the following lemma.
To prove the desired convergence, we first need to show that the set $\partial\{u > 0\}$ has measure zero in $Q$: this crucial property (missing in \cite{Phillips1987:art}) follows from the non-degeneracy and optimal growth estimates of \Cref{lem:HausConv}.
\begin{lem}\label{lem:L1ConvFBMeas0} 
Let $\gamma \in (0,1]$, $\alpha \in (0,1)$. Let $u_\circ \in C_c^{2+\alpha}(\RR^n)$ be nonnegative and nontrivial, and let $u_{\vep_j}$ and $u$ as in \Cref{lem:StrCmp}. 
Then $\{u>0\}$ has positive density in $Q$.
As a consequence,
\begin{equation}\label{eq:FB0Meas}
\mathcal{L}^{n+1}(\partial\{u > 0\}) = 0,
\end{equation}
where $\mathcal{L}^{n+1}$ denotes the $(n+1)$-dimensional Lebesgue measure.
Furthermore, 
\begin{equation}\label{eq:L1FvepCh}
H_{\vep_j}(u_j) \to \chi_{\{u > 0\}} \quad \text{in } L^1_{\loc}(Q),
\end{equation}
as $j \uparrow \infty$.
\end{lem} 
\begin{proof} 
We first show \eqref{eq:FB0Meas}, proceeding in the spirit of \cite[Theorem 1.3]{AudritoSanz2022:art} and \cite[Theorem 5.1]{Weiss1999:art}. Note that we require $u_\circ$ nontrivial to have $\{u>0\} \not=\varnothing$. 

\

\emph{Step 1: Density estimate.} First, we prove that, for every open bounded set $U \subset\subset Q$, there exists $c_0 > 0$ such that for every $(x,t) \in \partial\{u>0\}\cap U$ and every $r > 0$ small enough, we have 
\begin{equation}\label{eq:Density}
	\dfrac{\mathcal{L}^{n+1}(\{u>0\}\cap Q_r(x,t))}{\mathcal{L}^{n+1}(Q_r(x,t))} \geq c_0.
\end{equation}
By the non-degeneracy estimate \eqref{eq:NonDegLimLem}, there exists $(y,\tau)\in \{u>0\}\cap \overline{Q_{r/2}(x,t)}$ such that
\begin{equation}\label{eq:NonDegFinal}
u^{1/\beta}(y,\tau) \geq c r,
\end{equation}
for some $c > 0$ independent of~$r$ and $(x,t)$. 
Now, we claim that there exists $\eta \in (0,\tfrac{1}{2})$, independent of~$r$ and $(x,t)$, such that
\[
Q_{\eta r}(y,\tau) \subset \{u>0\} \cap  Q_r(x,t).
\]
Indeed, setting $w := u^{2/\beta}$ and using the optimal growth estimate \eqref{eq:OptGrLimLem} and the $C^1$ bounds \eqref{eq:OptRegSpaceLim} and \eqref{eq:OptRegTimeLim}, we easily see that
\begin{equation}
    |\nabla w| = |\nabla (u^{1/\beta}\cdot u^{1/\beta}) | = 2 u^{1/\beta} |\nabla (u^{1/\beta}) | \leq C r 
    \quad \text{ and } \quad 
    |\partial_t w | \leq C
    \quad \text{ in } Q_r(x,t)
\end{equation}
for some constant $C>0$ independent of $r$.
Thus, for every $(z,\theta) \in Q_{\eta r}(y,\tau)$, setting $\kappa(s) := s(y,\tau) + (1-s)(z,\theta)$ with $s \in [0,1]$, we easily see that
\[
\begin{aligned}
w(y,\tau) - w(z,\theta) &= \int_0^1 (\nabla w (\kappa(s)), \partial_t w (\kappa(s)) \cdot (y-z,\tau-\theta) \, \rd s \\
&\leq \sup_{s \in [0,1]} |\nabla w (\kappa(s))| \, |y-z| + \sup_{s \in [0,1]} |\partial_t w (\gamma(s))| \, |\tau-\theta| \leq C r \cdot \eta r + C \, (\eta r )^2  \leq 2C \eta r^2.
\end{aligned}
\]
Combining the bound above with \eqref{eq:NonDegFinal}, which rewrites as $w(y,\tau) \geq c^2 r^2$ we deduce
\[
w(z,\theta) \geq \big(c^2 - 2C\eta \big) r^2 > 0,
\]
if $\eta > 0$ is small enough. 
Our claim is proved and \eqref{eq:Density} immediately follows. 

\

\emph{Step 2: Measure of the free boundary.}
The  proof of \eqref{eq:FB0Meas} relies on a contradiction argument which combines the density estimate \eqref{eq:Density} with a fine covering procedure: it works exactly as in \cite[Theorem 1.3]{AudritoSanz2022:art} and \cite[Theorem 5.1]{Weiss1999:art}, and we skip it.

\

\emph{Step 3: Proof of \eqref{eq:L1FvepCh}.} We will show that
$H_{\vep_j}(u_j) \to \chi_{\{u > 0\}}$ a.e. in $Q$ as $j \uparrow\infty$, and thus the convergence in $L^1_\loc(Q)$  follows from the dominated convergence theorem. 
Furthermore, using that $\mathcal{L}^{n+1}(\partial\{u > 0\}) = 0$, it suffices to show that $H_{\vep_j}(u_j) \to 1$ a.e. in $\{u>0\}$ and that $H_{\vep_j}(u_j) \to 0$ a.e. in $\text{int}(\{u = 0\})$.

To prove that $H_{\vep_j}(u_j) \to 1$ a.e. in $\{u>0\}$, take $(x,t) \in \{u > 0\}$. 
Then, since $u_j \to u$ locally uniformly, we have $u_j(x,t)\geq u(x,t)/2$ for $j$ large enough and thus, since $H$ is nondecreasing, 
\begin{equation}
    H_{\vep_j}(u_j(x,t)) \geq H (u(x,t)/(2\vep_j^\beta)) = 1
\end{equation}
for every $j$ large enough. 
Thus, $H_{\vep_j}(u_j) \to 1$ pointwise in $\{u > 0\}$ as $j \uparrow\infty$.

To show that $H_{\vep_j}(u_j) \to 0$ in $\text{int}(\{u = 0\})$, assume by contradiction that there exists $\sigma > 0$ and $(x,t) \in \{u = 0\}$ with $\dist((x,t),\{u > 0\}) \geq \sigma$ such that $H_{\vep_j}(u_j(x,t)) \geq \sigma$ for some (not-relabeled) subsequence. 
Since $H$ is continuous with $H(0) = 0$, there exists $\vartheta \in (0,1)$ such that $H(\vartheta) < \sigma$ while, by the Hausdorff convergence \eqref{eq:HausConv}, it follows that $(x,t) \in \{u_j \leq \vartheta\vep_j^\beta\}$ for every $j$ large enough. Thus, by monotonicity of $H$, we obtain   
\[
\sigma \leq H_{\vep_j}(u_j(x,t)) = H(u_j(x,t)/\vep_j^\beta) \leq H(\vartheta) < \sigma,
\]
taking $j$ large enough, a contradiction. 
\end{proof}

Having the previous result at hand, we are getting closer to derive the limit equations. However, a crucial result is still needed, that is, the fact that $f_\vep (u_\vep) \to \gamma u_+^{\gamma - 1}$ a.e. in $Q$ as $\vep \downarrow 0$, along a suitable subsequence.

\begin{rem}\label{rem:PhilProb}
 The a.e. convergence of $f_\vep (u_\vep)$ towards $\gamma u_+^{\gamma - 1}$ is implicitly stated in \cite{Phillips1987:art} (see the proof of Theorem 1, pp. 259), without a detailed proof. We notice that this is not obvious at all, even in the framework in which Phillips works: in that case, the nonlinearity is given by
\[
f_\vep(u_\vep) = \gamma \, \frac{u_\vep}{\vep + u_\vep^{2-\gamma}},
\]
up to the multiplicative constant $\gamma$ (see \cite[Formula (0.3)]{Phillips1987:art}). Then, by uniform convergence, $f_\vep(u_\vep) \to \gamma u^{\gamma-1}$ in $\{u > 0\}$ (see \cite[Proof of Theorem 1]{Phillips1987:art} and/or \emph{Step 1} of  \Cref{lem:ConvergengeNonlinearityAE} below), but it is not clear that $f_\vep(u_\vep) \to 0$ in $\{u = 0\}$ (or in $\text{int}(\{u = 0\})$). Actually, without a stronger control over $u_\vep$, $f_\vep(u_\vep)$ may blow-up at points where $u=0$: for example, if $(x,t) \in \{u=0\}$ and $u_\vep(x,t) \to 0$ with $u_\vep^{2-\gamma}(x,t) \gg \vep$, then $f_\vep(u_\vep(x,t)) \sim \gamma \, u_\vep^{\gamma-1}(x,t) \to \infty$ as $\vep \downarrow 0$. 
\end{rem}
In the spirit of the remark above, we study the pointwise limit of the family $\{f_\vep(u_\vep)\}_{\vep > 0}$, along a suitable subsequence. We begin with two technical results (\Cref{Lemma:ODELowerBoundPowers} and \Cref{Lemma:ODEGrowth}) that we will use in the proof of \Cref{lem:ConvergengeNonlinearityAE}.

\begin{lem}\label{Lemma:ODELowerBoundPowers}
    Let $n \geq 1$, $\lambda,R > 0$ and $\alpha > 1$. Let $\phi:[0,R] \to \RR$ be a nonnegative solution to
    \begin{equation}
        \begin{cases}
            r^{1-n} (r^{n-1} \phi')' = \lambda \phi^\alpha & \text{ in } (0,R), \\
            \phi(0) = 1, \\
            \phi'(0) = 0.
        \end{cases}
    \end{equation}
    Then, $\phi'>0$ in $(0,R)$ and for every $k =1,2,\ldots$, there exists $c_{n,k}>0$, depending only on $n$ and $k$, such that
    \begin{equation}
    \label{Eq:ODELowerBoundPowers}
        \phi(r) \geq c_{n,k} \lambda^k r^{2k} \quad \text{ for every }r\in (0,R).
    \end{equation}
\end{lem}
\begin{proof}
    First, we notice that $\phi'>0$ in $(0,R)$. Indeed, integrating the equation of $\phi$, we have
    \begin{equation}\label{Eq:AppODEIntegrated}
        r^{n-1} \phi'(r) = \lambda \int_0^r \rho^{n-1} \phi(\rho)^\alpha \, \rd \rho
    \end{equation}
    and thus, since $\phi$ is not identically zero in $(0,R)$ (this is a consequence of the assumption $\phi(0) = 1$), \eqref{Eq:AppODEIntegrated} yields our claim. 

    Now, since $\phi$ is increasing, we have $\phi\geq 1$ in $(0,R)$. Thus, from \eqref{Eq:AppODEIntegrated} using that $\alpha>1$, we obtain
    \begin{equation}
        \label{Eq:AppODEIntegratedIneq}
        r^{n-1} \phi'(r) \geq  \lambda \int_0^r \rho^{n-1} \phi(\rho) \, \rd \rho .
    \end{equation}
    From this, using that $\phi\geq 1$, we get $r^{n-1} \phi'(r) \geq  (\lambda/n) r^n$, which yields
    \begin{equation}
    \label{Eq:AppODEQuadraticLB}
        \phi(r) \geq 1 + \dfrac{\lambda}{2n} r^2 \geq  \dfrac{\lambda}{2n} r^2 \quad r\in (0,R).
    \end{equation}
    Using this last inequality in the right-hand side of \eqref{Eq:AppODEIntegratedIneq}, we obtain $r^{n-1} \phi'(r) \geq  \lambda^2 /(2n(n+2)) r^{n+2}$, which gives
    \begin{equation}
    \label{Eq:AppODEQuadraticLB}
        \phi(r) \geq \dfrac{\lambda^2}{8n(n + 2)} r^4 \quad r\in (0,R).
    \end{equation}
    Iterating this procedure, that is, inserting the last inequality in \eqref{Eq:AppODEIntegratedIneq} and integrating, \eqref{Eq:ODELowerBoundPowers} easily follows.
\end{proof}
\begin{lem}\label{Lemma:ODEGrowth}
    Let $n \geq 1$, $\gamma,\delta \in (0,1]$, and $c >0$. Set $j:= \lceil 1/\gamma \rceil$, $j-1/\gamma =: \theta \in [0,1)$ and define
    \begin{equation}\label{eq:OmegaRstar}
        \omega_\star := \delta^{\frac{1 + \theta}{2 (\theta + 2/\gamma)}} \quad \text{ and } \quad R_\star:= \bigg (\dfrac{\delta^{\frac{1 + \theta}{2}}}{c^j c_{n,j}} \bigg)^{\frac{1}{2j}},
    \end{equation}
    where $c_{n,j}$ is the constant given by \Cref{Lemma:ODELowerBoundPowers} for $k=j$.
    Assume that $\varphi:[0,R_\star] \to \RR$ satisfies
    \begin{equation}
        \begin{cases}
            r^{1-n} (r^{n-1} \varphi')' = \dfrac{c\,\omega_\star}{\delta} \varphi^{1+\gamma} & \text{ in } (0,R_\star),\\
            \varphi(0) = 1, \\
            \varphi'(0) = 0.
        \end{cases}
    \end{equation}
    Then,
    \begin{equation}
        \label{Eq:ODEGrowth}
        (\omega_\star \delta)^{1/\gamma} \varphi(R_\star)  \geq \delta.
    \end{equation}
\end{lem}

\begin{proof}
    Applying \Cref{Lemma:ODELowerBoundPowers} with $\phi = \varphi$, $\alpha=1 + \gamma$, $\lambda = c\, \omega_\star / \delta$, $R = R_\star$, and $k=j$, we get
    \begin{equation}
        \varphi(R_\star) \geq c_{n,j} \left ( c \frac{\omega_\star}{\delta} \right)^j R_\star^{2j} = \left ( \frac{\omega_\star}{\delta} \right)^j \delta^{\frac{1 + \theta}{2}} = \dfrac{\omega_\star^j}{\delta^j} \dfrac{\delta^{1 + \theta}}{\delta^{\frac{1 + \theta}{2}}} = \dfrac{\omega_\star^j}{\omega_\star^{\theta + 2/\gamma}} \dfrac{\delta^{1 + \theta}}{\delta^j},
    \end{equation}
    where in the first equality we have used the definition of $R_\star$ and, in the last one, the definition of $\omega_\star$.
    Using that $j = \theta + 1/\gamma$, \eqref{Eq:ODEGrowth} follows.
\end{proof}
With the previous results at hand, we can now prove that $f_{\vep_j}(u_j) \to \gamma u_+^{\gamma-1} $ a.e. in $Q$.
As mentioned, the delicate step will be to have a suitable control of $u_{\vep}$ in $\textrm{int}(\{u=0\})$ in terms of $\vep$, which will follow from a suitable barrier argument using the above technical lemmas.
\begin{lem}\label{lem:ConvergengeNonlinearityAE}
    Let $\gamma \in (0,1]$, $\alpha \in (0,1)$. Let $u_\circ \in C_c^{2+\alpha}(\RR^n)$ be nonnegative and nontrivial, and let $u_{\vep_j}$ and $u$ as in \Cref{lem:StrCmp}. 
    Then
    \begin{equation}\label{eq:AEConvfvep}
        f_{\vep_j}(u_j) \to \gamma u_+^{\gamma-1} \quad \text{ a.e. in } Q,
    \end{equation}
    as $j \uparrow \infty$.
\end{lem}
\begin{proof}
Recall that
\begin{equation}
    f_{\vep_j}(u_j) = \vep_j^{-\beta} h(u_j/\vep_j^\beta) u_j^\gamma + \gamma H (u_j/\vep_j^\beta) u_j^{\gamma - 1},
\end{equation}
for every $j$. We will separately study the convergence in $\{u > 0\}$ and in $\text{int}(\{u = 0\})$, showing pointwise convergence to $u^{\gamma - 1}$ in $\{u > 0\}$ and to $0$ in $\text{int}(\{u = 0\})$.
Then, since $\mathcal{L}^{n+1}(\partial\{u > 0\}) = 0$ by \Cref{lem:L1ConvFBMeas0}, we will obtain the desired a.e. convergence in $Q$.

\

\emph{Step 1:} We start with the easiest case, in which we can actually show that $f_{\vep_j}(u_j) \to \gamma u^{\gamma-1} $ locally uniformly in $\{u > 0\}\cap Q$.
For this, given a compact set $\KK \subset\subset \{u > 0\}\cap Q$, we can find $\delta_\KK > 0$ such that $u_j \geq \delta_\KK$ in $\KK$ for $j$ large enough (recall that $u_j \to u$ locally uniformly).
Consequently, recalling that $\supp \, h \subset [0,1]$ and $H \equiv 1$ in $[1,\infty)$, 
\[
f_{\vep_j}(u_j) = \vep^{-\beta} h(u_j/ \vep^{\beta})u_j^\gamma + \gamma  H(u_j/\vep_j^\beta) u_j^{\gamma-1} = \gamma u_j^{\gamma-1} \quad \text{ in } \KK,
\]
for $j$ large enough, which implies $f_{\vep_j}(u_j) \to \gamma u^{\gamma-1}$ uniformly in $\KK$ as $j \to +\infty$.

\

\emph{Step 2:} Next we consider points $(x,t) \in \text{int}(\{u = 0\})$ and we fix $\vartheta\in (0,1)$ small enough such that
\begin{equation}
\label{Eq:BoundshH}
    c_\circ s \leq h(s) \leq C_\circ s \quad
    \text{ and } \quad 
    \dfrac{c_\circ}{2} s^2 \leq H(s) \leq \dfrac{C_\circ}{2} s^2 \quad
    \text{ for } s\in (0,\vartheta),
\end{equation}
for some constants $c_\circ, C_ \circ>0$ depending only on $h$ ---recall that $h(0)=0$ and $h'(0)>0$, see \eqref{Eq:Propertiesh}. 
Then, by the Hausdorff convergence \eqref{eq:HausConv}, we have $(x,t) \in \{u_j \leq \vartheta \vep_j^\beta\}$ for every $j$ large enough, and therefore
\begin{equation}
0 \leq  \vep_j^{-\beta} h(u_j(x,t)/\vep_j^\beta) u_j^\gamma(x,t) \leq C_\circ \dfrac{u_j(x,t)}{\vep_j^\beta} \dfrac{u_j(x,t)^\gamma}{\vep_j^\beta}  \leq C_\circ \vartheta \dfrac{u_j(x,t)^\gamma}{\vep_j^\beta}
\end{equation}
and
\begin{equation}
0 \leq H(u_j(x,t)/\vep_j^\beta)u_j^{\gamma-1}(x,t) 
\leq \dfrac{C_\circ}{2} \dfrac{u_j(x,t)}{\vep_j^\beta} \dfrac{u_j(x,t)^\gamma}{\vep_j^\beta} 
\leq \dfrac{C_\circ}{2} \vartheta \dfrac{u_j(x,t)^\gamma}{\vep_j^\beta}.  
\end{equation}
Thus, to conclude the proof, it is enough to show that 
\begin{equation}
\label{Eq:FineConvToZero}
    \dfrac{u_j(x,t)^\gamma}{\vep_j^\beta} \to 0,
\end{equation}
as $j \uparrow \infty$, which is what we do in the next paragraphs.
From now on, since the point $(x,t)$ does not play any role and we do not use the initial condition, we may assume that $(x,t) = (0,0)$.

\

\emph{Step 2.1:} The proof of \eqref{Eq:FineConvToZero} is a delicate barrier argument as follows. First, notice that, thanks to the bounds \eqref{Eq:BoundshH}, we have that
\begin{equation}
\label{Eq:FineConvSubsolution}
    \partial_t u_j - \Delta u_j \leq - c_\circ \left( 1 + \dfrac{\gamma}{2} \right) \dfrac{u_j^{\gamma + 1}}{\vep_j^{2\beta}} =: - c_\star \dfrac{u_j^{\gamma + 1}}{\vep_j^{2\beta}} \quad \text{ in } \{u_j \leq \vartheta \vep_j^\beta\}.
\end{equation}
In particular, for $j$ large enough, this equation holds in $B_R\times (-T, 0)$ for any $R>0$ and $T>0$ small enough (recall the Hausdorff convergence \eqref{eq:HausConv} and that $(0,0) \in \text{int}(\{u = 0\})$).

Then our goal is to find $R=R(\vep_j)$ and $T= T(\vep_j)$ such that $R(\vep_j),T(\vep_j) \to 0$ as $j \uparrow \infty$, and a function $\overline{u}$ satisfying
\begin{equation}
\label{Eq:FineConvSupersolution}
    \begin{cases}
        \partial_t \overline{u} - \Delta \overline{u} \geq - c_\star \dfrac{\overline{u}^{\gamma + 1}}{\vep_j^{2\beta}} & \text{ in } B_R\times (-T, 0),\\
        \overline{u} \geq \vep_j^\beta & \text{ in } (\partial B_R\times [-T, 0]) \cup (B_R \times \{-T\}),
    \end{cases}
\end{equation}
and
\begin{equation}
\label{Eq:FineConvBoundZERO}
    \overline{u}(0,0) \leq \omega (\vep_j) \vep_j^{\beta/\gamma},
\end{equation}
for some $\omega = \omega(\vep_j)$ satisfying $\omega(\vep_j) \to 0$ as $j \uparrow \infty$. If so, for $j$ large enough, both \eqref{Eq:FineConvSubsolution} and \eqref{Eq:FineConvSupersolution} hold in $B_R\times (-T, 0)$ and thus $u_j \leq \overline{u}$ in $B_R\times (-T, 0)$, by the comparison principle. Then, \eqref{Eq:FineConvToZero} immediately follows from \eqref{Eq:FineConvBoundZERO}.

\

The rest of the proof consists of building such supersolution $\overline{u}$.
For the sake of clarity, and since $j$ will not change, we will denote $\vep = \vep_j$.
We will construct $\overline{u}$ of the form
\begin{equation}
    \overline{u}(x,t) = \phi(t) + X(x),
\end{equation}
with $\phi$ and $X$ nonnegative functions satisfying
\begin{equation}\label{eq:SeparateSuperSol}
\begin{cases}
    \phi' = - \dfrac{c_\star}{2 \vep^{2\beta}} \phi^{\gamma + 1}, \\
    \phi(0) = \omega_{\mathrm{T}} \, \vep^{\beta/\gamma},
\end{cases}
\quad \text{ and } \qquad 
\begin{cases}
    \Delta X = \dfrac{c_\star}{2 \vep^{2\beta}} X^{\gamma + 1}, \\
    X(0) = \omega_{\mathrm{X}} \, \vep^{\beta/\gamma},
\end{cases}
\end{equation}
for some $\omega_{\mathrm{T}}$ and $\omega_{\mathrm{X}}$. Indeed, if such $\phi$ and $X$ exist, the first inequality of \eqref{Eq:FineConvSupersolution} is satisfied, since
\begin{equation}
    \partial_t \overline{u} - \Delta \overline{u} + c_\star \dfrac{\overline{u}^{\gamma + 1}}{\vep^{2\beta}} =
    \partial_t \phi - \Delta X + c_\star \dfrac{(\phi + X)^{\gamma + 1}}{\vep^{2\beta}} 
    \geq 
    \partial_t \phi + c_\star \dfrac{\phi^{\gamma + 1}}{2\vep^{2\beta}} 
    - \Delta X + c_\star \dfrac{X^{\gamma + 1}}{2\vep^{2\beta}} = 0.
\end{equation}
Moreover, if $\omega_{\mathrm{T}} = \omega_{\mathrm{T}}(\vep) $ and $\omega_{\mathrm{X}}=\omega_{\mathrm{X}}(\vep)$ are such that $\omega_{\mathrm{T}}(\vep)\to 0$ and $\omega_{\mathrm{X}}(\vep)\to 0$ as $\vep \downarrow 0$, then \eqref{Eq:FineConvBoundZERO} is satisfied with $\omega := \omega_{\mathrm{T}}+ \omega_{\mathrm{X}}$.

Therefore, it remains to appropriately choose $\omega_{\mathrm{T}}$ and  $\omega_{\mathrm{X}}$ in such a way $\phi:[-T, 0]\to \RR$ and $X:\overline{B_R} \to \RR$ satisfy
\begin{equation}
    \phi(-T) \geq \vep^\beta \quad \text{ and } \quad X \geq \vep^\beta \quad \text{ in } \partial B_R,
\end{equation}
so that the second inequality of \eqref{Eq:FineConvSupersolution} is satisfied as well. We study each problem separately.

\

\emph{Step 2.2:} For the time-dependent function $\phi$, we just use some elementary ODE analysis.
Indeed, it is not difficult to check that if $\phi' = - \lambda \phi^{1 + \gamma}$ with $\lambda>0$, and $\phi(0) = \phi_0$, then $\phi$ is defined as $\phi:(-t_\star, \infty) \to \RR$ by
\begin{equation}
    \phi(t) = \dfrac{\phi_0}{(1 + t/ t_\star )^{1/\gamma}}, \quad \text{ where} \quad t_\star :=  \dfrac{1}{\lambda \gamma \phi_0^\gamma}.
\end{equation}
In addition, for any $L>0$ we have that
\begin{equation}
    \phi(t) \geq L \quad \Leftrightarrow \quad t \leq t_\star\left[ (\phi_0/L)^\gamma - 1\right].
\end{equation}
Using these facts with $\phi_0 = \omega_{\mathrm{T}} \, \vep^{\beta/\gamma}$, $\lambda =c_\star/(2 \vep^{2\beta})$, and $L = \vep^\beta$, we obtain 
\begin{equation}
    \phi(t) \geq \vep^\beta \quad \Leftrightarrow \quad t \leq - \dfrac{2}{\gamma c_\star} \frac{\vep^\beta}{\omega_{\mathrm{T}}^\gamma}\bigg( 1 - \dfrac{\omega_{\mathrm{T}}^\gamma \vep^\beta}{\vep^{\beta\gamma}} \bigg).
\end{equation}
Thus, taking $\omega_{\mathrm{T}} := \vep^{\frac{\beta - 1}{\gamma}}$, we have that $\phi: (-2\vep/(\gamma c_\star)  , +\infty)\to \RR$ satisfies
\begin{equation}
    \phi(-T) = \vep^\beta, \qquad \text{where } T:= \dfrac{2}{\gamma c_\star} \vep (1 - \vep).
\end{equation}

\

\emph{Step 2.3:} We turn our attention to the space-dependent problem.
For this, we look for a radially symmetric solution to the second problem in \eqref{eq:SeparateSuperSol}, that is, a solution $X:[0,R] \to \RR$ to
\begin{equation}
    \begin{cases}
    r^{1 - n} (r^{n-1}  X')' = \dfrac{c_\star}{2 \vep^{2\beta}} X^{\gamma + 1}  & \text{ in } (0,R),\\
    X(0) = \omega_{\mathrm{X}} \, \vep^{\beta/\gamma} , \\ 
    X'(0) = 0,
\end{cases}
\end{equation}
and choose $\omega_{\mathrm{X}}$ and $R$ such that $X(R) \geq \vep^\beta$. The existence of solutions follows by standard methods. Further, by \Cref{Lemma:ODELowerBoundPowers}, we have $X' > 0$. 
In the arguments below, we assume that $X$ is defined in the whole interval $(0,R)$ for $R = R_\star$, where $R_\star$ is given by \Cref{Lemma:ODEGrowth}. Otherwise, $X$ blows-up at some point $R < R_\star$ and therefore the bound $X(R) \geq \vep^\beta$ is satisfied choosing $R$ smaller. 

Notice that the above initial value problem is equivalent to
\begin{equation}
    \begin{cases}
    r^{1 - n} (r^{n-1}  \varphi')' = \dfrac{c_\star}{2 \vep^{\beta}} \omega_{\mathrm{X}}^\gamma \varphi^{\gamma + 1}  & \text{ in } (0,R), \\
    \vspace{1mm}
    \varphi(0) = 1 ,\\ 
    \varphi'(0) = 0 ,
\end{cases}
\end{equation}
where $X(r) := \omega_{\mathrm{X}} \, \vep^{\beta/\gamma} \varphi(r)$. Therefore, we want to show that for some appropriate choice of $\omega_{\mathrm{X}}$ and $R$ (satisfying the desired smallness in $\vep$), it holds
\begin{equation}
\label{Eq:OmegaBoundsX}
    (\omega_{\mathrm{X}}^\gamma \vep^\beta)^{1/\gamma}\varphi(R) \geq \vep ^\beta.
\end{equation}
But this follows by \Cref{Lemma:ODEGrowth} with the choices $\delta = \vep^\beta$, $c = c_\star/2$, $\omega_\star = c_1 \vep^{\sigma_1}$, and $R_\star = c_2 \vep^{\sigma_2}$, where $c_1, c_2, \sigma_1, \sigma_2 >0$ depend only on $n$ and $\gamma$ ---they can be explicitly computed using \eqref{eq:OmegaRstar}---, and taking $\omega_{\mathrm{X}}^\gamma = \omega_\star$ and $R = R_\star$. This concludes the construction of the function $X$ and the proof of \eqref{eq:AEConvfvep}.
\end{proof}
\subsection{Limit equations}
In the next two results of this section, we derive the weak formulations of the equation of $u$. 
As a byproduct, we complete the proof of the first part of \Cref{thm:MAIN1}. 
First, we obtain the weak formulation with respect to domain variations.
\begin{lem}\label{lem:L2H1StgConv} 
Let $\gamma \in (0,1]$ and $\alpha \in (0,1)$. Let $u_\circ \in C_c^{2+\alpha}(\RR^n)$ be nonnegative and nontrivial, and let $u_{\vep_j}$ and $u$ as in \Cref{lem:StrCmp}. Then
\begin{equation}\label{eq:L2H1StgConv}
u_{\vep_j} \to u \quad \text{ locally in } L^2(0,\infty;H^1(\RR^n)),
\end{equation}
as $j \uparrow \infty$. Furthermore, $u$ satisfies
\begin{equation}\label{eq:FirstIVLim}
\int_Q \big( |\nabla u|^2 + 2u_+^\gamma \big) \dv_x \Phi  -  2 \nb u \cdot D_x\Phi \cdot \nabla u - 2\partial_t u \,(\nb u \cdot \Phi) = 0,
\end{equation}
for every $\Phi \in C_c^\infty(Q;\RR^{n+1})$ and
\begin{equation}\label{eq:LocalEstPartialTu}
\int_s^\tau \int_{\RR^n} (\partial_tu)^2\psi^2 + \frac{1}{2} \int_{\RR^n} \big( |\nb u|^2 + 2u_+^\gamma)\psi^2 \rd x \, \bigg|_{t=s}^{t=\tau}  + 2 \int_s^\tau \int_{\RR^n} \partial_tu  \psi \,(\nb u\cdot \nb \psi) \leq 0,
\end{equation}
for a.e. $0 < s < \tau$ and every $\psi \in C_c^\infty(\RR^n)$.
\end{lem} 
\begin{proof} Let us set $u_j := u_{\vep_j}$. We proceed in three steps as follows.

\

\emph{Step 1: Proof of \eqref{eq:L2H1StgConv}.} It is enough to check that, for every bounded open set $\UU \subset\subset Q$ and every $\eta \in C_c^\infty(\UU)$, we have
\begin{equation}\label{eq:StrongConvLoc}
\int_\UU |\nb u_j|^2 \eta  \to \int_\UU |\nb u|^2 \eta,
\end{equation}
as $j\uparrow\infty$. First, since $u_j \to u$ locally uniformly, we have that $\{u > 0\}$ is open and $u$ is a classical solution to
\begin{equation}\label{eq:EqPosSet}
\partial_t u - \Delta u = - \gamma u^{\gamma-1} \quad \text{ in } \{u > 0\}.
\end{equation}
Indeed, by \Cref{lem:ConvergengeNonlinearityAE} (\emph{Step1}) we have that, for every compact $\KK \subset\subset \{u > 0\}$, $f_{\vep_j}(u_j) \to \gamma u^{\gamma-1}$ uniformly in $\KK$ as $j \uparrow\infty$. So \eqref{eq:EqPosSet} follows by testing the equation $\partial_t u_j- \Delta u_j = -f_{\vep_j}(u_j) $ with an arbitrary $\varphi \in C_c^\infty(\KK)$, passing to the limit as $j \uparrow\infty$, and using the classical parabolic Schauder theory (\cite[Theorem 4.8 and Theorem 4.9]{Lieberman1996:book}), noticing that the function $\tau \to \tau^{\gamma - 1}$ is smooth in $(0,\infty)$.

Now, fix $\sigma \in (0,1)$ and set $u_\sigma := (u-\sigma)_+$. 
By \eqref{eq:EqPosSet}, $u_\sigma$ is a classical solution to 
\[
\partial_t u_\sigma - \Delta u_\sigma = - \gamma (u_\sigma + \sigma)^{\gamma-1} \quad \text{ in } \{u > \sigma\}:
\]
Now, we test this equation with $\varphi = u_\sigma \eta$, where $\UU \subset\subset Q$ is open and bounded and $\eta \in C_c^\infty(\UU)$ is nonnegative.
Note that $\supp\, \varphi \subset \overline{\{u > \sigma\}}$ by the definition of $u_\sigma$.
Integrating by parts in space (this is possible since, thanks to Sard's lemma, $\{u=\sigma\}$ is a smooth hypersurface for a.e. $\sigma >0$ ---recall that $u$ is a classical solution, and thus smooth, in $\{u>0\}$), we obtain 
\begin{equation}\label{eq:StrongConUsigmaRel}
\int_Q |\nb u_\sigma|^2 \eta \dx\dt = -\int_Q \big[u_\sigma \partial_t u_\sigma \eta + u_\sigma \nb u_\sigma \cdot \nb \eta \big]\dx\dt - \gamma \int_Q (u_\sigma + \sigma)^{\gamma-1}u_\sigma \eta \dx\dt,
\end{equation}
for a.e. $\sigma > 0$. Noticing that $0 \leq (u_\sigma + \sigma)^{\gamma-1}u_\sigma \leq u_\sigma^\gamma \leq u^\gamma \in L^\infty(Q)$, and $u_\sigma \to u$ locally uniformly in $Q$, $\partial_tu_\sigma \to \partial_tu$ and $\nabla u_\sigma \to \nb u$ locally in $L^2(Q)$ as $\sigma \downarrow 0$, we may pass to the limit into the equation above to deduce
\begin{equation}\label{eq:IntGrad2Lim}
\int_Q |\nb u|^2 \eta = -\int_Q \big[u \partial_t u \eta + u\nb u \cdot \nb \eta \big] - \gamma \int_Q u^\gamma \eta .
\end{equation}
Now, testing the equation for $u_j$, \eqref{eq:FirstOVvep}, with $\varphi = u_j \eta$, we obtain
\begin{equation}
    \int_Q |\nb u_j|^2\eta = - \int_Q (u_j \partial_t u_j \eta + u_j \nb u_j \cdot \nb \eta ) - \int_Q  u_j f_{\vep_j}(u_j) \eta .
\end{equation}
Therefore, \eqref{eq:StrongConvLoc} will follow if we show that the right-hand side of the previous equation converges to the right-hand side of \eqref{eq:IntGrad2Lim}.

We do this as follows. 
For the first integral, we simply use that for every open bounded $\tilde{\UU} \subset\subset Q$ it holds that $\nb u_j \rightharpoonup \nb u$ and $\partial_t u_j \rightharpoonup \partial_tu$ weakly in $L^2(\tilde{\UU})$, and $u_j \to u$ uniformly in $\tilde{\UU}$.
For the second integral, we have 
\begin{equation}
    \begin{split}
    \int_Q  u_j f_{\vep_j}(u_j) \eta &= \int_{\{u_j \leq \vep_j^\beta\}}  u_j f_{\vep_j}(u_j) \eta  + \int_{\{u_j > \vep_j^\beta\}}  u_j f_{\vep_j}(u_j) \eta \\
    &= O(\vep_j^\beta) +\gamma \int_Q  u_j^\gamma \chi_{\{u_j > \vep_j^\beta\}} \eta \to \gamma \int_Q u_+^\gamma \eta,
    \end{split}
\end{equation}
as $j \uparrow \infty$.
Here we have used that $\{f_{\vep_j}(u_j)\}_{j \in\NN}$ is uniformly bounded in $L_\loc^1(Q)$ (see \emph{Step 1} in \Cref{lem:FirstOVLim}) and that $\chi_{\{u_j > \vep_j^\beta\}} \to \chi_{\{u > 0\}}$ in $L^1_\loc(Q)$.
The proof  of this last fact follows the same lines of \eqref{eq:L1FvepCh}, using that if $x\in \{u>0\}$ and $y\in \textrm{int}(\{u=0\})$, then for $j$ large enough we have $u_j(x) > u(x)/2 > \vep_j^\beta$ and $y\in \{u\leq \vep_j\}$ and thus $\chi_{\{u_j > \vep_j^\beta\}} (x) = 1$ and $\chi_{\{u_j > \vep_j^\beta\}} (y) = 0$.
This completes the proof of \eqref{eq:StrongConvLoc}.

\

\emph{Step 2: Proof of \eqref{eq:FirstIVLim}.} For every $j \in \NN$, $u_j$ satisfies \eqref{eq:FirstIVvep}, that is, 
\[
\int_Q \big( |\nabla u_j|^2 + 2H_{\vep_j}(u_j)u_j^\gamma \big) \dv_x \Phi  -  2\nb u_j \cdot D_x \Phi \cdot \nabla u_j - 2\partial_tu_j \,(\nb u_j \cdot \Phi) = 0,
\]
for every $\Phi \in C_c^\infty(Q;\RR^{n+1})$. Then, \eqref{eq:FirstIVLim} follows by passing to the limit as $j \uparrow \infty$ in the above equation, using that $\partial_t u_j \rightharpoonup \partial_tu$ weakly in $L^2(Q)$, $\nb u_j \to \nb u$ locally in $L^2(Q)$ ---see \eqref{eq:L2H1StgConv}---, $H_{\vep_j}(u_j) \to \chi_{\{u>0\}}$ in $L^1_\loc(Q)$ ---by \eqref{eq:L1FvepCh} in \Cref{lem:L1ConvFBMeas0}---, and recalling that $u_j^\gamma \to u^\gamma$ locally uniformly in $Q$ (by  \Cref{lem:StrCmp}).

\

\emph{Step 3: Proof of \eqref{eq:LocalEstPartialTu}.} 
The limit in \eqref{eq:L2H1StgConv} implies that, up to passing to a suitable subsequence, the set $A := \{t \in (0,\infty): \nb u_j(t) \to \nb u(t) \text{ in } L^2_\loc(\RR^n)\}$ satisfies $\LL^1((0,\infty)\setminus A) = 0$. So, let $s,\tau \in A$ with $s < \tau$ and let $\psi \in C_c^\infty(\RR^n)$ be a spatial cut-off function. Testing the weak formulation of the equation satisfied by $u_j$, \eqref{eq:FirstOVvep}, with $\varphi := \chi_{(s,\tau)}(t) \psi^2(x) \partial_tu_j$, we easily obtain
\[
\int_s^\tau \int_{\RR^n} (\partial_tu_j)^2\psi^2 + \frac{1}{2} \int_{\RR^n} \big[ |\nb u_j|^2 + 2 H_{\vep_j} (u_j) u_j^\gamma \big] \psi^2 \rd x \, \bigg|_{t=s}^{t=\tau}  + 2 \int_s^\tau \int_{\RR^n} \partial_tu_j  \psi \,(\nb u_j\cdot \nb \psi) = 0.
\]
Then \eqref{eq:LocalEstPartialTu} follows by passing to the limit as $j \uparrow \infty$ in the above equation (exactly as in \emph{Step 2}) and using the lower semicontinuity of the $L^2(Q)$-norm.
\end{proof}
\begin{rem}
    \label{rem:FBcondition}
    Note that the variational formulation \eqref{eq:FirstIVLim} encodes the free boundary condition 
    \begin{equation}
    \label{eq:RemFBcondition}
        |\nabla(u^{1/\beta})| = \dfrac{\sqrt{2}}{\beta} \quad \text{ in } \partial \{u>0\}\cap Q,
    \end{equation}
    as shown by the following formal computation.
    Indeed, considering $\Phi = (\Phi^1, \ldots, \Phi^n, \Phi^{n+1})$, denoting derivatives as $u_j := \partial_j u$, and and using the summation convention over repeated indices (from $1$ to $n$), we can integrate by parts and compute formally
    \begin{equation}
    \label{eq:ComputationWeakSolution}
        \begin{split}
            0 &= \int_{\{u>0\}}\big( u_i^2 + 2u^\gamma \big)  \Phi_j^j  -  2 u_i u_j \Phi_i^j - 2\partial_t u \,(u_j \Phi^j) \\
            &= -\int_{\{u>0\}}\big( 2 u_i u_{ij} + 2\gamma u^{\gamma-1}u_j \big)  \Phi^j  -  (2 u_{ii} u_j + 2 u_i u_{ji} ) \Phi^j + 2\partial_t u \,(u_j \Phi^j) \\
            & \quad \quad + \int_{\partial \{u>0\}}\big( u_i^2 + 2u^\gamma \big)  \Phi^j \nu^j  -  2 u_i u_j \Phi^j \nu^i \\
            &= -\int_{\{u>0\}} 2 \big( \gamma u^{\gamma-1} - \Delta u + \partial_t u \big)(u_j \Phi^j) 
            + \int_{\partial \{u>0\}}\big( |\nabla u|^2 + 2u^\gamma \big)  \Phi^j \nu^j  -  2 u_i u_j \Phi^j \nu^i, \\
        \end{split}
    \end{equation}
    where $\nu = (\nu^1, \ldots, \nu^n, \nu^{n+1}) $ is the unit normal vector to $\partial \{u>0\}$.
    Assuming that $u$ satisfies \eqref{eq:EqPosSet} and that $\nu = -(\nabla u, \partial_t u)/(|\nabla u|^2 + (\partial_t u)^2)^{1/2}$, we obtain
    \begin{equation}
    \begin{split}
        0 &= \int_{\partial \{u>0\}}\big( |\nabla u|^2 + 2u^\gamma \big)  \Phi^j \nu^j  -  2 u_i u_j \Phi^j \nu^i
        = \int_{\partial \{u>0\}}\big(  2 u^\gamma - |\nabla u |^2 \big) \dfrac{\Phi^j u_j}{\sqrt{|\nabla u|^2 + (\partial_t u)^2}} \\
        & = \int_{\partial \{u>0\}}\left( 2  - \frac{|\nabla u |^2}{u^\gamma} \right) u^\gamma  \dfrac{\Phi \cdot \nabla u}{\sqrt{|\nabla u|^2 + (\partial_t u)^2}}
         = \int_{\partial \{u>0\}}\left( 2  - \beta^2 |\nabla (u^{1/\beta})|^2 \right) u^\gamma  \dfrac{\Phi \cdot \nabla u}{\sqrt{|\nabla u|^2 + (\partial_t u)^2}},
    \end{split}
    \end{equation}
    which unveils the FB condition \eqref{eq:RemFBcondition}.
    Of course, to show that a solution satisfies \eqref{eq:RemFBcondition}, one should make rigorous these formal computations, under appropriate regularity assumptions on $u$ and the free boundary, as shown in the next result.  
\end{rem}

\begin{lem}\label{lem:FBCondition} 
    Let $\gamma \in (0,1]$, let $u$ be a weak solution to \eqref{eq:EQSing} in the sense of \eqref{eq:FirstIVLim} and let $(x_\circ, t_\circ)\in \partial  \{u>0\}$. Assume that there exists $r > 0$ such that:
\

$\bullet$ The exists $\tau \in C^1(B_r(x_\circ))$ such that $|\nb \tau| \not=0$ in $B_r(x_\circ)$ and
        \[
        \{ u > 0\} \cap Q_r(x_\circ,t_\circ) = \{ (x,t): t > \tau(x) \}\cap Q_r(x_\circ,t_\circ). 
        \]
\

$\bullet$ $u^{1/\beta} \in C^1(\overline{\{u>0\}}\cap Q_r(x_\circ,t_\circ))$ and $|\nabla (u^{1/\beta})| \not= 0$ in $\overline{\{u>0\}} \cap Q_r(x_\circ,t_\circ)$.

\

    Then,
    \begin{equation}
        |\nabla(u^{1/\beta})|(x_\circ, t_\circ) = \dfrac{\sqrt{2}}{\beta}.
    \end{equation}
\end{lem}

\begin{proof} 
By invariance under translations, we may assume $(x_\circ,t_\circ) = (0,0)$. Testing \eqref{eq:FirstIVLim} with the vector field $(x,t) \mapsto (\Phi(x)\,\eta_\sigma(t), 0)$, where $\Phi = (\Phi^1, \ldots, \Phi^n)\in C_c^\infty(B_r;\mathbb{R}^n)$ and $\eta_\sigma\in C_c^\infty(\mathbb{R})$ is a nonnegative temporal cut-off with $\operatorname{supp}\eta_\sigma = [-\sigma/2,\sigma/2]$ and $\eta_\sigma \rightharpoonup^{\star} \delta_0$ as $\sigma \downarrow 0$.  
    Then, thanks to our assumptions, a standard argument using the Mean Value Theorem leads to 
    \begin{equation}
        0 = \int_{\Omega}\big( u_i^2 + 2u^\gamma \big)  \Phi_j^j  -  2 u_i u_j \Phi_i^j - 2\partial_t u \,(u_j \Phi^j) \dx, \qquad \Omega := \{x: u(x,0) > 0\} \cap B_r.
    \end{equation}
    As before, we are using the summation convention over repeated indices (from $1$ to $n$), and the subindex denotes differentiation, that is, $w_i := \partial_{x_i} w$.
    Now, we set
    \begin{equation}
        v:=u^{1/\beta},\qquad\text{i.e.}\qquad u=v^\beta.
    \end{equation}
    We obtain
    \begin{equation}
    0 = \int_{\Omega}\big(\beta^2 v^{2\beta-2}|\nabla v|^2 + 2 v^{\beta\gamma}\big)\Phi_j^j
    -2\beta^2 v^{2\beta-2} v_i v_j \Phi_i^j  -2\beta^2 v^{2\beta-2} (\partial_t v)\, v_j\Phi^j \dx.
    \end{equation}
    Noticing that $2(\beta - 1) = \beta \gamma$, this rewrites as
    \begin{equation}
    0 = \int_{\Omega} v^{\beta\gamma} \left [ \big(|\nabla v|^2 + 2/\beta^2 \big)\Phi_j^j
    -2 v_i v_j \Phi_i^j  -2 (\partial_t v)\, v_j\Phi^j \right] \dx .
    \end{equation}
    Since $\Phi$ can be taken Lipschitz by a standard approximation, we fix $\varepsilon>0$ and take 
    \begin{equation}
    \Phi :=
    \begin{cases}
    v^{-\beta\gamma}\Psi & \text{ if } v\ge\varepsilon\\[4pt]
    \varepsilon^{-\beta\gamma}\,\Psi & \text{ if } v < \varepsilon,
    \end{cases}
    \end{equation}
    where $\Psi \in C_c^\infty(B_r;\mathbb{R}^n)$. Then, we obtain
    \begin{equation}
        \begin{split}
        0 &= 
        \int_{\Omega \cap \{v\ge\varepsilon\}}\Big[ \big(|\nabla v|^2+2/\beta^2\big)\Psi_j^j -2v_i v_j \Psi_i^j -2(\partial_t v)\, v_j\Psi^j \Big]\, \dx \\
        &\quad
        +\beta\gamma \int_{\Omega \cap\{v\ge\varepsilon\}} v^{-1}  \big(|\nabla v|^2-2/\beta^2\big)v_j\Psi^j
        \,\dx \\
        &\quad
        +\int_{\Omega \cap \{0<v<\varepsilon\}} (v/\varepsilon)^{\beta\gamma}\Big[ \big(|\nabla v|^2+2/\beta^2\big)\Phi_j^j -2v_i v_j\Phi_i^j - 2(\partial_t v)\, v_j\Phi^j\Big]\,\dx.
        \end{split}
    \end{equation}
    Now, by our regularity assumptions, we can let $\vep \downarrow0$ in the previous expression and, by dominated convergence, the last integral converges to zero, while the first integral  in $\Omega \cap \{v\ge\varepsilon\}$ converges to the same integral in $\Omega$.
    Consequently, we deduce that $v^{-1}  \big(|\nabla v|^2-2/\beta^2\big)v_j\Psi^j$ must be integrable in $\Omega$. This, combined with the regularity of $|\nabla v|$ and that $|\nb v| \not=0$ in $\overline{\Omega}$, shows that $|\nabla v|^2=2/\beta^2$ on $\partial \Omega$ (this is because $v$ behaves like the distance to the free boundary $\partial \Omega$ and thus the weight $v^{-1}$ is not integrable in $\Omega$; hence the continuous factor multiplying it must vanish on $\partial \Omega$).
\end{proof}
Before proceeding further, some remarks are in order.
\begin{rem}
    \label{rem:hypothesesFBCondition}
    Our regularity assumptions on $u^{1/\beta}$ and the FB are necessary since otherwise the result may be false, as shown by the example $u(x,t) = (-\frac{2\gamma}{\beta}t)_+^{\beta/2}$ described later in \Cref{Subsec:SpecialSolutions}.
    
    The first assumption in \Cref{lem:FBCondition} rules out ``horizontal'' FB points (points at which, following the notation in the statement, $\nabla \tau = 0$): in such case, the set $\Omega$ appearing in the proof could be empty or the whole ball $B_r$, and thus the above argument would not work. 
    
    Regarding our assumptions on $u^{1/\beta}$, there are two main comments. On the one hand, the assumption $|\nabla (u^{1/\beta})| \not= 0$ in $\overline{\{u>0\}} \cap Q_r(x_\circ,t_\circ)$ is quite natural in view of the examples we provide later on in Subsection \ref{Subsec:SpecialSolutions}, \Cref{sec:SelfSimilar}, and \Cref{sec:TravelingWaves}. 
    On the other hand, even though we cannot expect $u^{1/\beta} \in C^1(\overline{\{u>0\}})$ in time (recall again the example $u(x,t) = (-\frac{2\gamma}{\beta}t)_+^{\beta/2}$), we notice that the previous result holds under weaker assumptions on $\partial_t u$; for instance, the last limit requires $\partial_t(u^{1/\beta})$ being locally integrable only (still assuming $\nabla (u^{1/\beta})$ continuous). 
    Furthermore, having some control over $\partial_t (u^{1/\beta})$ in terms of $|\nabla (u^{1/\beta})|$ is natural (cf. \cite[Section 8]{CafVaz95} in the case $\gamma=0$), at least at ``regular vertical'' FB points, that is, FB points where the blow-up is of the form $(\tfrac{\sqrt{2}}{\beta}x_1)_+^\beta$, up to a rotation and a translation (see Subsection \ref{Subsec:SpecialSolutions} again).    
\end{rem}
\begin{rem}
\label{remark:solutionsAreWeakSol}
We also stress that the computation \eqref{eq:ComputationWeakSolution} shows that, whenever $\partial \{u>0\}$ is locally the graph of a $C^1$ function $t = \tau(x)$ as in Lemma \ref{lem:FBCondition}, and $u$ satisfies the integrability assumptions of \Cref{def:WeakSolEIntro} and
    \[
    \begin{cases}
     \partial_t u - \Delta u = -\gamma u^{\gamma-1} \quad &\text{in } \{ u > 0\} \\
     |\nb u| = 0                                    \quad &\text{in } \partial\{ u > 0\}  
    \end{cases}
    \]
in the classical sense, then $u$ is a weak solution in the sense of \Cref{def:WeakSolEIntro}. Notice that this is true even when $\partial\{u > 0\} = \{u=0\}$, that is, $u > 0$ in ``both sides'' of $\partial\{u > 0\}$ and the FB is \emph{non-regular}, since the outer normal vector at FB points is undefined.

Indeed, to see that $u$ is a solution in the sense of domain variations (that is, \eqref{eq:FirstIVLimIntro}), one just needs to split $\{u > 0\}$ into its two connected components $\mathcal{C}_1$ and $\mathcal{C}_2$ and start from the last two terms in \eqref{eq:ComputationWeakSolution} replacing the integration on $\{u>0\}$ with the integration on $\mathcal{C}_1$ and $\mathcal{C}_2$, respectively. Notice the all the integrals are well-defined (with some abuse of notation, $\nu$ denotes the exterior unit normal to both $\partial\mathcal{C}_1$ and $\partial\mathcal{C}_2$) and are both equal to zero thanks to our assumption $|\nb u| = 0$ in $\partial\{ u > 0\}$. Then one goes backward in our computations, undoing the integration by parts in each connected component and arriving to the first line in \eqref{eq:ComputationWeakSolution}. Using the integrability assumptions on $u$ and recalling that $\mathcal{L}^{n+1}(\partial\{ u > 0\}) = 0$, our claim follows. A very similar argument shows that $u$ is a weak solution in the sense of \eqref{eq:FirstOVLimIntro} as well.

Note that this is a substantial difference between the cases $\gamma >0$ and $\gamma = 0$. When $\gamma=0$, the function $\sqrt{2}|x_1|$ is still a weak solution in the sense of domain variations,\footnote{Now each integrand in the boundary term appearing in \eqref{eq:ComputationWeakSolution} does not vanish individually, but the whole boundary integral is zero.} that is, it satisfies
\begin{equation}\label{eq:InnerVarGamma0}
\int_{\RR^{n+1}} \big( |\nabla u|^2 + 2\chi_{\{u>0\}} \big) \dv_x \Phi  -  2 \nb u \cdot D_x \Phi \cdot \nabla u - 2\partial_t u \,(\nb u \cdot \Phi) = 0,
\end{equation}
for every $\Phi \in C_c^\infty(\RR^{n+1};\RR^{n+1})$. However, it is not a weak solution in the sense of \eqref{eq:FirstOVLimIntro}. One can see this in the elliptic setting, whose solutions are stationary solutions of the parabolic equation. In dimension $n=1$ (then the argument is extended easily to higher dimensions), one can easily realize that the elliptic energy $\JJ_0$ ---see \eqref{eq:EllipticEnergy}--- of the function $\sqrt{2}|x|$ can be decreased lifting the function near $x=0$, removing the FB (this is only possible with an variation of the type $\sqrt{2}|x| + \vep \varphi$, $\varphi \geq 0$), and thus it cannot be a critical point of the energy.
\end{rem}

To conclude this section, we obtain the weak formulation of the equation of $u$, now in the sense of~\eqref{eq:FirstOVLimIntro}.

\begin{lem}\label{lem:FirstOVLim} 
Let $\gamma \in (0,1]$ and $\alpha \in (0,1)$. Let $u_\circ \in C_c^{2+\alpha}(\RR^n)$ be nonnegative and nontrivial, and let $u_{\vep_j}$ and $u$ as in \Cref{lem:StrCmp}. Then $u_+^{\gamma-1} \in L_\loc^1(Q)$ and $u$ satisfies
\begin{equation}\label{eq:FirstOVLim}
\int_Q \partial_t u \varphi + \nabla u \cdot\nabla\varphi + \gamma u_+^{\gamma-1} \varphi = 0,
\end{equation}
for every $\varphi \in C_c^\infty(Q)$.
\end{lem} 
\begin{proof} Let $u_j := u_{\vep_j}$. The proof is divided in two steps as follows. 

\

\emph{Step 1: $u_+^{\gamma-1} \in L^1_\loc(Q)$}. First we notice that the family $\{f_{\vep_j}(u_j)\}_{j \in\NN}$ is uniformly bounded in $L_{\loc}^1(Q)$: this easily follows from the energy bounds \eqref{eq:EnBound1Eps} and \eqref{eq:EnBound2Eps}, and the equation of $u_j$ \eqref{eq:FirstOVvep}. Combining this with the a.e. limit in \eqref{eq:AEConvfvep} and Fatou's lemma, we deduce $u_+^{\gamma-1} \in L^1_\loc(Q)$. 

\

\emph{Step 2: Proof of \eqref{eq:FirstOVLim}.} The final part of the proof generalizes the argument in \cite[Theorem 1]{Phillips1987:art}. Let us consider a function $\psi \in C^\infty(\RR)$ satisfying $\psi',\psi'' \geq 0$, $\psi(v) = 0$ for $v \leq 1/2$ and $\psi(v) = v-1$ for $v \geq 2$. For every $\sigma > 0$, set $\psi_\sigma(v) := \sigma \psi (v/\sigma)$: in this way, $\psi_\sigma(v) \to v_+$ as $\sigma \downarrow 0$.  

\

Now, let us fix $j \in \NN$, $\sigma > 0$, and $\varphi \in C_c^\infty(Q)$. Testing the weak formulation of the equation satisfied by $u_j$, \eqref{eq:FirstOVvep},  with $ \psi_\sigma'(u_j)\varphi$, we deduce that 
\begin{equation}\label{eq:Eqpsisigma}
\int_Q  \psi_\sigma'(u_j)(\partial_t u_j \varphi + \nb u_j \cdot \nb\varphi)  = -\int_Q \psi_\sigma''(u_j) |\nb u_j|^2 \varphi - \int_Q \psi_\sigma'(u_j) f_{\vep_j} (u_j) \varphi.
\end{equation}
We next show how \eqref{eq:FirstOVLim} follows from \eqref{eq:Eqpsisigma} passing to the double limit $j\uparrow\infty$ and $\sigma \downarrow 0$.

\

\emph{Limit as $j \uparrow \infty$:} First, since $\supp \, \psi_\sigma' \subset [\sigma/2,\infty)$, we have $\psi_\sigma'(u_j) f_{\vep_j} (u_j) = \gamma \psi_\sigma'(u_j) u_j^{\gamma-1} \leq \gamma \|\psi'\|_\infty (\tfrac{\sigma}{2})^{\gamma-1}$ for every $j$ large enough and, in light of \eqref{eq:AEConvfvep}, $\psi_\sigma'(u_j) f_{\vep_j} (u_j) \to \gamma \psi_\sigma'(u) u^{\gamma-1}$ a.e. in $Q$. 
Consequently, recalling that $u_j \to u$ locally uniformly in $Q$ and $\nabla u_j \to \nb u$ in $L^2_\loc(Q)$, by \Cref{lem:L2H1StgConv}, we obtain
\[
\int_Q \psi_\sigma''(u_j) |\nb u_j|^2 \varphi +\int_Q \psi_\sigma'(u_j) f_{\vep_j} (u_j) \varphi \to \int_Q \psi_\sigma''(u) |\nb u|^2 \varphi + \gamma \int_Q \psi_\sigma'(u) u^{\gamma-1} \varphi,
\]
as $j \uparrow \infty$, by the dominated convergence theorem. 
Therefore, since $\partial_t u_j \rightharpoonup \partial_t u$ weakly in $L^2(Q)$ and, as above, $\nabla u_j \to \nb u$ locally in $L^2(Q)$, we may pass to the limit as $j \uparrow \infty$ in \eqref{eq:Eqpsisigma} to deduce
\begin{equation}\label{eq:LimEqPsi}
\int_Q  \psi_\sigma'(u)(\partial_t u \varphi + \nb u \cdot \nb\varphi) = \int_Q \psi_\sigma''(u) |\nb u|^2 \varphi + \gamma \int_Q \psi_\sigma'(u) u^{\gamma-1} \varphi.
\end{equation}

\emph{Limit as $\sigma \downarrow 0$:} 
We study the convergence of the right-hand side of \eqref{eq:LimEqPsi}.
On the one hand, since $|\nb u|^2 \leq C_\circ \beta^2 u^\gamma$ in $Q$ by \eqref{eq:OptRegSpaceLim} and $u_+^{\gamma-1} \in L^1_\loc(Q)$ (see \emph{Step 1}), we have
\begin{equation}\label{eq:LimEqEst1}
\begin{aligned}
\int_Q \psi_\sigma''(u) |\nb u|^2 \varphi &= \int_{\{0 < u < 2\sigma\}} \psi_\sigma''(u) |\nb u|^2 \varphi\leq \frac{C_\circ\beta^2\|\psi''\|_\infty}{\sigma} \int_{\{0 < u < 2\sigma\}} u^\gamma \varphi  \\
&\leq 2C_\circ \beta^2\|\psi''\|_\infty \int_{\{0 < u < 2\sigma\}} u_+^{\gamma-1} \varphi \to 0,
\end{aligned}
\end{equation}
as $\sigma \downarrow 0$.  On the other hand, since $\mathcal{L}^{n+1}(\partial\{u > 0\}) = 0$ by \eqref{eq:FB0Meas}, it is not difficult to check that $\psi_\sigma'(u) \to \chi_{\{u > 0\}}$ a.e. in $Q$ as $\sigma \downarrow 0$, and thus locally in $L^2(Q)$ by the dominated convergence theorem. Then, since $\supp \, \psi_\sigma' \subset [\sigma/2,\infty)$ and $\psi_\sigma'(u) u^{\gamma-1}_+ \leq u_+^{\gamma-1}$,
\begin{equation}\label{eq:LimEqEst2}
\int_Q \psi_\sigma'(u) u^{\gamma-1} \varphi = \int_Q \psi_\sigma'(u) u_+^{\gamma-1} \varphi \to \int_Q u_+^{\gamma-1} \varphi
\end{equation}
as $\sigma \downarrow 0$, by dominated convergence again. 

Now, we focus on the left-hand side of \eqref{eq:LimEqPsi}.
Using $\psi_\sigma'(u) \to \chi_{\{u > 0\}}$ in $L^2_\loc(Q)$ as $\sigma \downarrow 0$ and the energy estimates \eqref{eq:EnBound1Lim} and \eqref{eq:EnBound2Lim}, it follows that
\begin{equation}\label{eq:LimEqEst3}
\int_Q  \psi_\sigma'(u)(\partial_t u \varphi + \nb u \cdot \nb\varphi) \to \int_Q  \chi_{\{u>0\}}(\partial_t u \varphi + \nb u \cdot \nb\varphi) = \int_Q (\partial_t u \varphi + \nb u \cdot \nb\varphi),
\end{equation}
as $\sigma \downarrow 0$, where we have also used $|\nb u| = 0$ and $\partial_tu = 0$ a.e. in $\{u=0\}$ (again, the information $\mathcal{L}^{n+1}(\partial\{u > 0\}) = 0$ is crucial). As a consequence, the weak formulation \eqref{eq:FirstOVLim} follows by \eqref{eq:LimEqEst1}, \eqref{eq:LimEqEst2} and \eqref{eq:LimEqEst3}, passing to the limit as $\sigma \downarrow 0$ in \eqref{eq:LimEqPsi} and using the arbitrariness of $\varphi \in C_c^\infty(Q)$.
\end{proof}
As mentioned in the introduction, we show that, under suitable smoothness assumptions on $u$ and $\partial\{u > 0\}$, the weak formulation \eqref{eq:FirstOVLim} encodes the FB condition \eqref{eq:RemFBcondition} as well.
\begin{lem}\label{lem:FBConditionBis} 
Let $\gamma \in (0,1]$, let $u$ be a weak solution to \eqref{eq:EQSing} in the sense of \eqref{eq:FirstOVLim} and let $(x_\circ, t_\circ)\in \partial  \{u>0\}$. Assume that there exists $r > 0$ such that:
\

$\bullet$ The exists $\tau \in C^1(B_r(x_\circ))$ such that $|\nb \tau| \not=0$ in $B_r(x_\circ)$ and
        \[
        \{ u > 0\} \cap Q_r(x_\circ,t_\circ) = \{ (x,t): t > \tau(x) \}\cap Q_r(x_\circ,t_\circ). 
        \]
\

$\bullet$ $u^{1/\beta} \in C^1(\overline{\{u>0\}}\cap Q_r(x_\circ,t_\circ))$ and $|\nabla (u^{1/\beta})| \not= 0$ in $\overline{\{u>0\}} \cap Q_r(x_\circ,t_\circ)$.

\

$\bullet$ $t_\circ$ is a Lebesgue point for $t \to \int_{B_r(x_\circ)} u_+^{\gamma-1}(x,t) \rd x$.
    
\    

Then, the same conclusion of Lemma \ref{lem:FBCondition} holds true.
\end{lem}
\begin{proof}
By invariance under translations, we may assume $(x_\circ,t_\circ) = (0,0)$. Since $t_\circ = 0$ is a Lebesgue point for $t \to \int_{B_r} u_+^{\gamma-1}(x,t) \rd x$, a standard argument gives us
    \begin{equation}
        0 = \int_{\Omega} \partial_t u \varphi + \nb u \cdot \nb \varphi + \gamma u^{\gamma-1} \varphi \dx, \qquad \Omega := \{x: u(x,0) > 0\} \cap B_r.
    \end{equation}
Now, let us consider $v:=u^{1/\beta}$. Using that $2(\beta-1) = \beta\gamma$, it is not difficult to check that it satisfies
\[
0 = \int_{\Omega} v^{\frac{\beta\gamma}{2}} \big( \partial_t v \varphi + \nb v \cdot \nb \varphi + \tfrac{\gamma}{\beta} v^{-1} \varphi \big) \dx.
\]
Taking 
    \begin{equation}
    \varphi :=
    \begin{cases}
    v^{-\frac{\beta\gamma}{2}} \psi & \text{ if } v\ge\varepsilon\\[4pt]
    \varepsilon^{-\frac{\beta\gamma}{2}}\,\psi & \text{ if } v < \varepsilon,
    \end{cases}
    \end{equation}
where $\psi \in C_c^\infty(B_r)$, we find
\[
\begin{aligned}
0 &= \int_{\Omega\cap\{v\geq\vep\}} \big( \partial_t v \psi + \nb v \cdot \nb \psi \big) \dx + \tfrac{\beta\gamma}{2}\int_{\Omega\cap\{v\geq\vep\}} v^{-1} \big( 2/\beta^2 - |\nb v|^2 \big)\psi \dx \\
&\quad+ \int_{\Omega\cap\{0<v<\vep\}} (v/\vep)^{\frac{\beta\gamma}{2}} \big( \partial_t v \psi + \nb v \cdot \nb \psi \big) \dx + \tfrac{\gamma}{\beta} \vep^{-\frac{\beta\gamma}{2}} \int_{\Omega\cap\{0<v<\vep\}} v^{\frac{\beta\gamma}{2}-1} \psi  \dx
\end{aligned}
\]
Now, we let $\vep \downarrow 0$. By our regularity assumptions, the first integral converge to the same integral in $\Omega$, while the third one converges to zero. Further, since $v$ behaves like the distance close to the FB and $\partial\Omega$ is $C^1$ in $B_r$, the fourth integral behaves like $\vep^{\frac{\beta\gamma}{2}}$ as $\vep \downarrow 0$ (this can be checked using a Bi-Lipschitz transformation that sends $\Omega\cap B_r$ into $\{x_1 > 0\} \cap B_r$ ), and thus the last term converges to a finite value. Consequently, $v^{-1}  \big(|\nabla v|^2-2/\beta^2\big)\psi$ must be integrable in $\Omega$ and our claim follows as in the last part of the proof of Lemma \ref{lem:FBCondition}. 
\end{proof}
%
%
%
%
%
%
%
%
%
%
%
%
\section{Weiss monotonicity formula and blow-ups}\label{sec:WeissMonotonicity}

In this section, we obtain a monotonicity formula for weak solutions $u$ to \eqref{eq:ParReacDiff} built as the limit of solutions $u_\vep$ to \eqref{eq:ProbEps}, and we exploit it to show that the blow-up limits of $u$ at FB points are parabolically $\beta$-homogeneous ``backward in time'' (see \Cref{rem:parabolicBetaHomog} below).

We stress that such monotonicity formula was previously obtained by Weiss in \cite{Weiss1999:art} for a class of solutions called ``variational solutions'', defined in a broader setting, but under some regularity and integrability assumptions on the solutions themselves. In particular, such formula can be applied to our weak solutions only if $\gamma\in (2/3,1)$ ---this is the range in which $\partial_t u $ has the right integrability to fulfill the definition of ``variational solutions'' used in the article of Weiss.

Here we follow the approach of \cite{art:Weiss2003}: our interest lies in obtaining a monotonicity formula with the natural regularity and integrability assumptions given by the structure of problem \eqref{eq:ParReacDiff} and valid for every $\gamma \in (0,1)$ and every weak solution to \eqref{eq:ParReacDiff}. 
Again, we derive a monotonicity formula for solutions $u_\vep$ to the semilinear problem \eqref{eq:ProbEps} and then show how it passes to the limit $\vep\downarrow 0$.

\subsection{Weiss-type monotonicity formula} We begin with some notation. We consider the backward heat kernel
\[
\varrho(x,t) := G(x,|t|) = \frac{1}{|4\pi t|^\frac{n}{2}} e^{-\frac{|x|^2}{4|t|}},
\]
defined for every $x \in \RR^n$ and $t < 0$, where $G$ is as in \eqref{eq:Gaussian}. For $(x_\circ,t_\circ) \in \RR^{n+1}$ and $r > 0$, we define the strip
\begin{equation}
    S_r^-(t_\circ) := \RR^n\times(t_\circ - 4r^2, t_\circ - r^2),
\end{equation}
with the convention $S_r^- := S_r^-(0)$, and the translations
\begin{equation}\label{eq:Transp0Weiss}
\begin{aligned}
v^{(x_\circ,t_\circ)}(x,t) &:= v(x + x_\circ,t + t_\circ) \\
v_{(x_\circ,t_\circ)}(x,t) &:= v(x - x_\circ,t - t_\circ),    
\end{aligned}
\end{equation}
where $v$ is a given function. If $v$ is regular enough ---for example, $v \in H^1_\loc(Q)$---, we also set
\[
Z_{(x_\circ,t_\circ)}v := (x-x_\circ)\cdot\nb v - 2(t_\circ-t)\partial_t v - \beta v,
\]
with the convention $Zv := Z_{(0,0)}v$.

Now, let $\vep > 0$. We consider the Weiss-type energies
\[
\begin{aligned}
\WW_{(x_\circ,t_\circ)}^{\,\vep}(v,r) :=  \frac{1}{r^{2 +\beta\gamma}} \int_{S_r^-(t_\circ)} \left[ |\nb v|^2 + 2F_\vep(v) \right] \varrho_{(x_\circ,t_\circ)} - \frac{\beta}{2r^{2+\beta\gamma}} \int_{S_r^-(t_\circ)} \frac{v^2}{t_\circ-t} \, \varrho_{(x_\circ,t_\circ)} 
\end{aligned}
\]
and
\[
\begin{aligned}
\WW_{(x_\circ,t_\circ)}(v,r) &:= \frac{1}{r^{2 +\beta\gamma}} \int_{S_r^-(t_\circ)} \left[ |\nb v|^2 + 2v_+^\gamma \right] \varrho_{(x_\circ,t_\circ)}  - \frac{\beta}{2r^{2+\beta\gamma}} \int_{S_r^-(t_\circ)} \frac{v^2}{t_\circ-t} \, \varrho_{(x_\circ,t_\circ)},
\end{aligned}
\]
with the conventions $\WW^{\,\vep} := \WW_{(0,0)}^{\,\vep}$ and $\WW := \WW_{(0,0)}$.

In the next proposition, we show that both $\WW_{(x_\circ,t_\circ)}^{\,\vep}$ and $\WW_{(x_\circ,t_\circ)}$ are monotone in $r$ along solutions $u_\vep$ to the approximating problem \eqref{eq:ProbEps} and limit solutions $u$ (as in \Cref{lem:StrCmp}), respectively.
\begin{prop}[Weiss Monotonicity Formula]\label{prop:WeissForSemLim}
Let $\gamma \in [0,1]$, $\alpha \in (0,1)$ and $(x_\circ,t_\circ) \in Q$. Let $u_\circ \in C_c^{2+\alpha}(\RR^n)$ be nonnegative and let $\{u_\vep\}_{\vep > 0}$ be a family of nonnegative weak solutions to \eqref{eq:ProbEps}. Then, for every $0 < R_1 < R_2 < \sqrt{t_\circ}/2$, we have
\begin{equation}\label{eq:WeissForSemilinear}
\begin{aligned}
\WW_{(x_\circ,t_\circ)}^{\,\vep}(u_\vep,R_2) - \WW_{(x_\circ,t_\circ)}^{\,\vep}(u_\vep,R_1) 
&= \int_{R_1}^{R_2} \frac{1}{r^{3 + \beta\gamma}} \left( \int_{S_r^-(t_\circ)} \frac{1}{t_\circ-t}\left[ Z_{(x_\circ,t_\circ)}u_\vep \right]^2 \varrho_{(x_\circ,t_\circ)} \right) \,  \rd r\\
&\quad + 2\beta \int_{R_1}^{R_2} \frac{1}{r^{3 + \beta\gamma}} \left(    \int_{S_r^-(t_\circ)} h_\vep(u_\vep) u_\vep^{\gamma +1} \varrho_{(x_\circ,t_\circ)} \right) \,  \rd r.
\end{aligned}
\end{equation}
Furthermore, if $\gamma \in (0,1]$ and $u$ is as in \Cref{lem:StrCmp}, then, for every $0 < R_1 < R_2 < \sqrt{t_\circ}/2$, we have
\begin{equation}\label{eq:WeissForLim}
\WW_{(x_\circ,t_\circ)}(u,R_2) - \WW_{(x_\circ,t_\circ)}(u,R_1) \geq \int_{R_1}^{R_2} \frac{1}{r^{3 + \beta\gamma}} \left(  \int_{S_r^-(t_\circ)} \frac{1}{t_\circ-t}\left[ Z_{(x_\circ,t_\circ)}u \right]^2 \varrho_{(x_\circ,t_\circ)} \right) \,  \rd r.
\end{equation}
\end{prop}
The proof of  \Cref{prop:WeissForSemLim} goes as follows. We first consider a spacial cut-off $\eta \in C_c^\infty(\RR^n)$ and, in  \Cref{lem:WeissCutOffTech}, we compute the derivative w.r.t. $r$ of the ``truncated'' Weiss energy
\[
\begin{aligned}
\WW_{(x_\circ,t_\circ)}^{\,\vep}(v,\eta_{x_\circ},r) := \frac{1}{r^{2 +\beta\gamma}} \int_{S_r^-(t_\circ)} \left[ |\nb v|^2 + 2F_\vep(v) \right] \varrho_{(x_\circ,t_\circ)} \eta_{x_\circ} - \frac{\beta}{2r^{2+\beta\gamma}} \int_{S_r^-(t_\circ)} \frac{v^2}{t_\circ-t} \, \varrho_{(x_\circ,t_\circ)} \eta_{x_\circ},
\end{aligned}
\]
along solutions to \eqref{eq:ProbEps}. This is technically involved, but all computations and integrations by parts are easily justifiable, since $\eta$ has compact support and the solutions are classical in $Q$, see \Cref{rem:InteriorReg+Inner}. Then, we show \eqref{eq:WeissForSemilinear} by keeping $\vep > 0$ fixed and letting $\eta \to 1$ locally uniformly in $\RR^n$. As a final step, we let $\vep \downarrow 0$ in \eqref{eq:WeissForSemilinear} and obtain \eqref{eq:WeissForLim} which can thus be interpreted as the limit of \eqref{eq:WeissForSemilinear}, as $\vep \downarrow 0$. Before proceeding with the proof, we discuss a couple of important issues in the following remarks.
\begin{rem}
First of all, we notice that, regarding  \eqref{eq:WeissForSemilinear}, we cannot compute the weak derivative $\frac{\rd}{\rd r}\WW_{(x_\circ,t_\circ)}(u,r)$: this is because we do not know if the weak convergence $\partial_t u_\vep \rightharpoonup \partial_tu$ in $L^2(Q)$ (given by the bound \eqref{eq:EnBound1Eps} in \Cref{prop:LimDelta}) is locally strong in $L^2(Q)$, along a suitable sequence.

Another important comment is that, contrary to \eqref{eq:WeissForSemilinear}, \eqref{eq:WeissForLim} does not hold for $\gamma = 0$: this is due to the lack of a non-degeneracy property (see \eqref{eq:NonDeg} and \eqref{eq:NonDegLimLem}) which, in turn, yields locally $L^1(Q)$-convergence of $F_\vep(u_\vep)$ to $u_+^\gamma$ along a suitable sequence (see \eqref{eq:L1FvepCh}). In other words, when $\gamma=0$, we cannot prove that $\WW_{(x_\circ,t_\circ)}^{\,\vep}(u_\vep,r) \to \WW_{(x_\circ,t_\circ)}(u,r)$ as $\vep \downarrow 0$, along a suitable sequence: as already mentioned, the case $\gamma = 0$ behaves differently and was treated in \cite{art:Weiss2003} (see also \cite{KrivenWeiss25:art} for some recent advances in this direction in the elliptic setting).    
\end{rem}
\begin{rem}\label{rem:parabolicBetaHomog}
As a final remark, we notice that the variation of $\WW_{(x_\circ,t_\circ)}(u,r)$ measures how far $u$ is of being a  parabolically $\beta$-homogeneous function w.r.t. the point $(x_\circ,t_\circ)$ ``backward in time'', that is,
\begin{equation}\label{eq:ParabolicHomDegBetaDef}
u_r^{(x_\circ,t_\circ)}(x,t) = u^{(x_\circ,t_\circ)}(x,t),   
\end{equation}
for every $(x,t) \in \RR^n \times(-\infty,0)$ and every $r > 0$, where
\[
u_r^{(x_\circ,t_\circ)}(x,t) := \frac{u^{(x_\circ,t_\circ)}(rx,r^2t)}{r^\beta} = \frac{u(x_\circ + rx,t_\circ + r^2t)}{r^\beta}.
\]
Indeed, if $r \to \WW_{(x_\circ,t_\circ)}(u,r)$ is constant in $(\rho,R)$ for every $0 < \rho < R$ fixed, we have $Z_{(x_\circ,t_\circ)}u = 0$ a.e. in $\RR^n \times (t_\circ-4R^2,t_\circ-\rho^2)$. By the arbitrariness of $0 < \rho < R$ and the definition of $Z_{(x_\circ,t_\circ)}u$, it is not difficult to check that this means 
\[
Z u^{(x_\circ,t_\circ)} = 0,
\]
a.e. in $\RR^n \times(-\infty,0)$ which, in turn, is equivalent to \eqref{eq:ParabolicHomDegBetaDef}. Notice that, taking $r := |t|^{-1/2}$, \eqref{eq:ParabolicHomDegBetaDef} yields
\[
u^{(x_\circ,t_\circ)}(x,t) = |t|^{\frac{\beta}{2}} u^{(x_\circ,t_\circ)}\big(|t|^{-\frac{1}{2}}x,-1\big) =: |t|^{\frac{\beta}{2}} U\big(|t|^{-\frac{1}{2}}x\big),
\]
a.e. in $\RR^n \times(-\infty,0)$.
We thus say that $u^{(x_\circ,t_\circ)}$ is \emph{self-similar} and $U$ is its \emph{self-similar profile}.
\end{rem}
As anticipated above, we begin with a technical lemma.
\begin{lem}\label{lem:WeissCutOffTech} Let $\gamma \in [0,1]$, $\alpha \in (0,1)$ and $(x_\circ,t_\circ) \in Q$. Let $u_\circ \in C_c^{2+\alpha}(\RR^n)$ be nonnegative and let $\{u_\vep\}_{\vep > 0}$ be a family of nonnegative weak solutions to \eqref{eq:ProbEps}. Then, for every $r \in (0,\sqrt{t_\circ}/2)$, we have
\begin{equation}\label{eq:DerWeissCutOff}
\begin{aligned}
\frac{\rd}{\rd r} \WW_{(x_\circ,t_\circ)}^{\,\vep}(u_\vep,\eta_{x_\circ},r) &= \frac{1}{r^{3 + \beta\gamma}} \int_{S_r^-(t_\circ)} \frac{1}{t_\circ-t}\left[Z_{(x_\circ,t_\circ)}u_\vep\right]^2 \varrho_{(x_\circ,t_\circ)} \eta_{x_\circ}  \\
&\quad+ \frac{2\beta}{r^{3 + \beta\gamma}}   \int_{S_r^-(t_\circ)} h_\vep(u_\vep) u_\vep^{\gamma +1} \varrho_{(x_\circ,t_\circ)} \eta_{x_\circ}  \\
&\quad+ \frac{1}{r} \,\WW_{(x_\circ,t_\circ)}^{\,\vep}(u_\vep,(x-x_\circ)\cdot\nb \eta_{x_\circ},r)  \\
&\quad - \frac{2}{r^{3 + \beta\gamma}} \int_{S_r^-(t_\circ)} Z_{(x_\circ,t_\circ)} u_\vep \, (\nb u_\vep \cdot \nb \eta_{x_\circ}) \, \varrho_{(x_\circ,t_\circ)}.
\end{aligned}
\end{equation}
\end{lem}
\begin{proof} 
Let us fix $\vep > 0$ and set $u = u_\vep$. 
It may be also useful while reading the proof to recall that $2 + \beta \gamma = 2 \beta$.

\

\emph{Step 1: Translation invariance and scaling.} The proof uses translation invariance and scaling as follows. 
First, since
\begin{equation}\label{eq:WeissInvTr}
\WW_{(x_\circ,t_\circ)}^{\,\vep}(u,\eta_{x_\circ},r) = \WW^{\,\vep}(u^{(x_\circ,t_\circ)},\eta,r),
\end{equation}
it is enough to compute the derivative of $\WW^{\,\vep}(\tilde{u},\eta,r)$ for a function $\tilde{u}$ satisfying $\partial_t \tilde{u}- \Delta \tilde{u} = - f_\vep(\tilde{u})$ in $\RR^n \times (-t_\circ, +\infty)$. 
Second, setting 
\begin{equation}
    v(x,t) := \frac{\tilde{u}(\vep x,\vep^2t)}{\vep^\beta}, \qquad 
\end{equation}
we easily see ---recall \eqref{eq:Scalingfeps}--- that $v$ satisfies $\partial_t v - \Delta v = - f_1(v)$ in $\RR^n \times (-t_\circ/\vep^2, +\infty) $ and that 
\begin{equation}\label{eq:WeissScalingVep}
\WW^{\,\vep}(\tilde{u},\eta, r) = \WW^{\,1}(v,\eta(\vep \cdot),\tfrac{r}{\vep})
\end{equation}
holds for all $r$ with $4r^2 < t_\circ/\vep^2$.
As a consequence, since 
\begin{equation}\label{eq:WeissScalingVepDerivative}
\frac{\rd }{\rd r} \WW^{\,\vep}(\tilde{u},\eta, r) = \dfrac{1}{\vep}\frac{\rd }{\rd \rho}\Bigg|_{\rho = r/\vep} \WW^{\,1}(v,\eta(\vep \cdot),\rho) ,
\end{equation}
for our purposes it will be enough to compute $\frac{\rd }{\rd r} \WW^{\,1}(v,\xi, r)$ for $v$ satisfying $\partial_t v - \Delta v = - f_1(v)$ in $\RR^n \times (-t_\circ/\vep^2, +\infty) $ and $\xi\in C^\infty_c(\RR^n)$ and then set
\begin{equation}
    \xi(x) := \eta(\vep x).
\end{equation}
Last, defining
\begin{equation}
    v_r(x,t) := \frac{v(rx,r^2t)}{r^{\beta}}, \qquad \xi_r(x) := \xi(rx),
\end{equation}
then $v_r$ satisfies the equation $\partial_t v_r - \Delta v_r = - f_{1/r}(v_r)$ in $\RR^n \times (-4, +\infty)$ provided that $r < \sqrt{t_\circ}/ 2$. 
Therefore, since
\[
\begin{aligned}
\WW^{\,1}(v,\xi, r) &= \frac{1}{r^{2 +\beta\gamma}} \int_{S_r^-} \left[ |\nb v|^2 + 2F(v) \right] \varrho \xi  - \frac{\beta}{2r^{2+\beta\gamma}} \int_{S_r^-} \frac{v^2}{|t|}  \,  \varrho \xi , \\
&= \int_{S_1^-} \left[ |\nb v_r|^2 + 2F_{1/r}(v_r) \right] \varrho \xi_r  - \frac{\beta}{2} \int_{S_1^-} \frac{v_r^2}{|t|}  \,  \varrho \xi_r = \WW^{\,1/r}(v_r,\xi_r,1).
\end{aligned}
\]
we are left to compute the derivative of $\WW^{\,1/r}(v_r,\xi_r,1)$.

\

\emph{Step 2: Computation of $\frac{\rd}{\rd r}\WW^{\,1/r}(v_r,\xi_r,1)$.} 
Since $v_r$ is a classical solution to $\partial_t v_r - \Delta v_r = - f_{1/r}(v_r)$ in $\RR^n \times (-4, +\infty)$ and $\xi_r$ is smooth and has compact support, we have enough regularity to differentiate under the sign of integral the quantities defining $\WW^{\,1/r}(v_r,\xi_r,1)$.
First, we have
\[
\frac{\rd}{\rd r} \int_{S_1^-} |\nb v_r|^2 \varrho \xi_r = 2 \int_{S_1^-} \nb v_r \cdot \nb \left( \tfrac{\rd}{\rd r} v_r \right) \varrho \xi_r  + \int_{S_1^-} |\nb v_r|^2 \varrho \tfrac{\rd}{\rd r}\xi_r
\]
Now, integrating by parts, recalling that $\nb \varrho = -\tfrac{x}{2|t|} \varrho$, and using the equation of $v_r$, we obtain
\begin{equation}\label{eq:WeissDerGrad2}
\begin{aligned}
\frac{\rd}{\rd r} \int_{S_1^-} |\nb v_r|^2 \varrho \xi_r &= - 2 \int_{S_1^-} \left( \Delta v_r + \nb v_r \cdot \tfrac{x}{2t} \right) \tfrac{\rd}{\rd r} v_r \, \varrho \xi_r \\
&\quad\, - 2 \int_{S_1^-} (\nb v_r \cdot \nb \xi_r) \, \varrho \tfrac{\rd}{\rd r} v_r + \int_{S_1^-} |\nb v_r|^2 \varrho \tfrac{\rd}{\rd r}\xi_r  \\
&=  - 2 \int_{S_1^-} \left( \partial_t v_r + f_{1/r}(v_r) + \nb v_r \cdot \tfrac{x}{2t} \right) \tfrac{\rd}{\rd r} v_r \, \varrho \xi_r \\
&\quad\, - 2 \int_{S_1^-} (\nb v_r \cdot \nb \xi_r) \, \varrho \tfrac{\rd}{\rd r} v_r + \int_{S_1^-} |\nb v_r|^2 \varrho \tfrac{\rd}{\rd r}\xi_r.
\end{aligned}
\end{equation}
Furthermore, noticing that
\[
\frac{\rd}{\rd r} F_{1/r}(v_r) = f_{1/r}(v_r) \frac{\rd}{\rd r} v_r + \beta r^{\beta-1} h(v_r r^\beta) v_r^{\gamma +1},
\]
it follows readily that
\begin{equation}\label{eq:WeissDerFAdj}
\frac{\rd}{\rd r} \int_{S_1^-} 2 F_{1/r}(v_r) \varrho \xi_r = \int_{S_1^-} \left[ 2 f_{1/r}(v_r) \tfrac{\rd}{\rd r} v_r + 2 \beta r^{\beta-1} h(v_r r^\beta) v_r^{\gamma +1} \right] \varrho\xi_r + \int_{S_1^-} 2 F_{1/r}(v_r) \varrho \tfrac{\rd}{\rd r}\xi_r.
\end{equation}
Last, we also have
\begin{equation}\label{eq:WeissDerLastTerm}
\frac{\rd}{\rd r} \left(- \frac{\beta}{2} \int_{S_1^-} \frac{v_r^2}{|t|}  \,  \varrho \xi_r \right) =  -\beta \int_{S_1^-} \tfrac{v_r}{|t|} \tfrac{\rd}{\rd r} v_r  \,  \varrho \xi_r - \frac{\beta}{2}  \int_{S_1^-} \frac{v_r^2}{|t|}  \,  \varrho \tfrac{\rd}{\rd r} \xi_r.
\end{equation}
Consequently, combining \eqref{eq:WeissDerGrad2} with \eqref{eq:WeissDerFAdj} and \eqref{eq:WeissDerLastTerm}, and recalling that $r \frac{\rd}{\rd r} v_r = Z v_r = x\cdot\nb v_r + 2t\partial_t v_r - \beta v_r$, we deduce
\[
\begin{aligned}
\frac{\rd}{\rd r} \WW^{\,1/r}(v_r,\xi_r,1)&=  \int_{S_1^-} \tfrac{1}{|t|}\left( 2t\partial_t v_r + x\cdot\nb v_r - \beta v_r \right) \tfrac{\rd}{\rd r} v_r \,\varrho \xi_r  + 2 \beta r^{\beta-1} \int_{S_1^-} h(v_r r^\beta) v_r^{\gamma +1} \varrho \xi_r  \\
&\quad\, +  \int_{S_1^-} \left[|\nb v_r|^2 + 2F_{1/r}(v_r) \right]\varrho \tfrac{\rd}{\rd r}\xi_r - \frac{\beta}{2}\int_{S_1^-} \frac{v_r^2}{|t|}  \,  \varrho \tfrac{\rd}{\rd r} \xi_r - 2 \int_{S_1^-} (\nb v_r \cdot \nb \xi_r) \, \varrho \, \tfrac{\rd}{\rd r} v_r  \\                                                                   
&= \frac{1}{r} \int_{S_1^-} \tfrac{1}{|t|}\left( Z v_r \right)^2 \varrho \xi_r + 2 \beta r^{\beta-1} \int_{S_1^-} h(v_r r^\beta) v_r^{\gamma +1} \varrho \xi_r \\
&\quad\, + \int_{S_1^-} \left[|\nb v_r|^2 + 2F_{1/r}(v_r) \right]\varrho \, [x\cdot \nb\xi(rx)] - \frac{\beta}{2}\int_{S_1^-} \frac{v_r^2}{|t|}  \,  \varrho \, [x\cdot \nb\xi(rx)] \\
&\quad\, - \frac{2}{r} \int_{S_1^-} Zv_r (\nb v_r \cdot \nb \xi_r) \, \varrho.
\end{aligned}
\]

\

\emph{Step 3: Scaling and conclusion.} Scaling back to $v$ and using the definition of $\WW$, we obtain
\[
\begin{aligned}
\frac{\rd}{\rd r} \WW^{\,1/r}(v_r,\xi_r,1)&= \frac{1}{r^{3 + \beta\gamma}} \int_{S_r^-} \tfrac{1}{|t|} \left( Zv \right)^2 \varrho\xi + \frac{2\beta}{r^{3 + \beta\gamma}}   \int_{S_r^-} h(v) v^{\gamma +1} \varrho\xi \\
&\quad\, + \frac{1}{r^{3 + \beta\gamma}} \int_{S_r^-} \left[|\nb v|^2 + 2F(v) \right] \varrho (x\cdot\nb \xi) - \frac{\beta}{2r^{3 + \beta\gamma}} \int_{S_r^-} \frac{v^2}{|t|} \varrho (x\cdot\nb \xi) \\
&\quad\, - \frac{2}{r^{3 + \beta\gamma}} \int_{S_r^-} Zv \, (\nb v \cdot \nb \xi) \varrho  \\
&= \frac{1}{r^{3 + \beta\gamma}} \int_{S_r^-} \tfrac{1}{|t|} \left( Zv \right)^2 \varrho\xi + \frac{2\beta}{r^{3 + \beta\gamma}} \int_{S_r^-} h(v) v^{\gamma +1} \varrho\xi \\
&\quad\, + \frac{1}{r} \,\WW^1(v,x\cdot\nb \xi,r)  - \frac{2}{r^{3 + \beta\gamma}} \int_{S_r^-} Zv \, (\nb v \cdot \nb \xi) \varrho,
\end{aligned}
\]
which, in turn, yields \eqref{eq:DerWeissCutOff} setting $\xi(x) = \eta(\vep x)$, by translation ---recall \eqref{eq:WeissInvTr}--- and taking into account~\eqref{eq:WeissScalingVepDerivative}.
\end{proof}
Now, we proceed with the proof of \eqref{eq:WeissForSemilinear}, the monotonicity formula for solutions $u_\vep$.
\begin{proof}[Proof of \eqref{eq:WeissForSemilinear}] Let us fix $\vep > 0$ and set $u = u_\vep$. As above, it is enough to fix $(x_\circ,t_\circ) = (0,0)$ and recover the general case using \eqref{eq:WeissInvTr}. 

\

For $\sigma > 0$, let $\eta_\sigma(x) := \eta(\sigma x)$, where $\eta(x) := \min\{1,(2-|x|)_+\}$ as in  \Cref{lem:BdryEstimates}. By definition of $\eta_\sigma$, We have
\begin{equation}\label{eq:ConvetasigmaWeiss}
\eta_\sigma \uparrow 1 \qquad \text{and} \qquad |\nb \eta_\sigma| \to 0 \quad \text{ locally uniformly in } \RR^n,
\end{equation}
as $\sigma \downarrow 0$. 

\

Plugging $\eta_\sigma$ into \eqref{eq:DerWeissCutOff} and integrating between $R_1$ and $R_2$, we obtain
\begin{equation}\label{eq:WeissForSemCutOff}
\begin{aligned}
\WW^{\,\vep}(u,\eta_\sigma,R_2) - \WW^{\,\vep}(u,\eta_\sigma,R_1) &= \int_{R_1}^{R_2} \frac{1}{r^{3 + \beta\gamma}} \int_{S_r^-} \frac{1}{-t}\left[Zu\right]^2 \varrho \eta_\sigma  + 2\beta \int_{R_1}^{R_2} \frac{1}{r^{3 + \beta\gamma}}   \int_{S_r^-} h_\vep(u) u^{\gamma +1} \varrho \eta_\sigma  \\
&\quad+ \int_{R_1}^{R_2} \frac{1}{r} \,\WW^{\,\vep}(u,x\cdot\nb \eta_\sigma,r) - 2\int_{R_1}^{R_2} \frac{1}{r^{3 + \beta\gamma}} \int_{S_r^-} Z u \, (\nb u \cdot \nb \eta_\sigma) \, \varrho.
\end{aligned}
\end{equation}
On the one hand, the energy estimates \eqref{eq:EnBound2Eps}, the definition of $\varrho$, and \eqref{eq:ConvetasigmaWeiss}, we immediately see that 
\begin{equation}\label{eq:WeissEpsLim0}
\WW(u,\eta_\sigma,R_1) \to \WW(u,R_1), \qquad \WW(u,\eta_\sigma,R_2) \to \WW(u,R_2),    
\end{equation}
as $\sigma \downarrow 0$, by the dominated convergence theorem. 
On the other hand, first by the monotone convergence theorem, we find
\begin{equation}\label{eq:WeissEpsLim1}
\begin{aligned}
&\int_{R_1}^{R_2} \frac{1}{r^{3 + \beta\gamma}}\int_{S_r^-} \frac{1}{-t}\left[ Z u \right]^2 \varrho \eta_\sigma  \,\uparrow \, \int_{R_1}^{R_2} \frac{1}{r^{3 + \beta\gamma}}\int_{S_r^-} \frac{1}{-t}\left[ Z u \right]^2 \varrho, \\
&\int_{R_1}^{R_2} \frac{1}{r^{3 + \beta\gamma}}   \int_{S_r^-} h_\vep(u) u^{\gamma +1} \varrho \eta_\sigma \, \uparrow \, \int_{R_1}^{R_2} \frac{1}{r^{3 + \beta\gamma}}   \int_{S_r^-} h_\vep(u) u^{\gamma +1} \varrho,
\end{aligned}
\end{equation}
as $\sigma \downarrow 0$. 
Next, by the definition of $\eta_\sigma$, we have $|\nb \eta_\sigma| \leq \sigma$ and $|(x\cdot\nb \eta_\sigma)| \leq 2$ a.e. in $\RR^n$. 
This, combined with the energy estimates \eqref{eq:EnBound2Eps}, the definition of $\varrho$, and \eqref{eq:ConvetasigmaWeiss}, allows us to apply the dominated convergence theorem again to deduce
\begin{equation}\label{eq:WeissEpsLim2}
\WW^{\,\vep}(u,x\cdot\nb \eta_\sigma,r) \to 0, \qquad    \int_{S_r^-} Z u \, (\nb u \cdot \nb \eta_\sigma) \, \varrho \to 0,
\end{equation}
as $\sigma \downarrow 0$. Then, \eqref{eq:WeissForSemilinear} follows by passing to the limit as $\sigma \downarrow 0$ into \eqref{eq:WeissForSemCutOff} and using \eqref{eq:WeissEpsLim0}, \eqref{eq:WeissEpsLim1}, and~\eqref{eq:WeissEpsLim2}.
\end{proof}
Lastly, we show the monotonicity formula \eqref{eq:WeissForLim} obtained by taking the limit $\vep \downarrow 0$.
\begin{proof}[Proof of \eqref{eq:WeissForLim}] 
As always, we fix $(x_\circ,t_\circ) = (0,0)$ and recover the general case by translation. Let $\vep_j$, $u_j := u_{\vep_j}$, and $u$ as in \Cref{lem:StrCmp}. 
For each $j \in \NN$, $u_j$ satisfies the monotonicity formula \eqref{eq:WeissForSemilinear} with $\vep = \vep_j$, which gives 
\begin{equation}\label{eq:WeissForSemilinearIneq}
\WW^{\,\vep_j}(u_j,R_2) - \WW^{\,\vep_j}(u_j,R_1) \geq \int_{R_1}^{R_2}\frac{1}{r^{3 + \beta\gamma}} \left (\int_{S_r^-} \frac{1}{-t}\left[ Z u_j \right]^2 \varrho \right) \rd r, 
\end{equation}
for every $j \in \NN$ and every $0 < R_1 < R_2$.
Note that we used that the second term in the right-hand side of \eqref{eq:WeissForSemilinear} is nonnegative, but actually one can show that it converges to zero as $j \to +\infty$ (the inequality in the result will come from Fatou's lemma, as will be seen below, since there is no strong convergence of $\partial_t u_j$).  

\

By  \Cref{lem:StrCmp} and \eqref{eq:L2H1StgConv}, we know that $u_j \to u$ locally uniformly in $Q$ and locally in $L^2(0,\infty:H^1(\RR^n))$ as $j \uparrow \infty$, while $F_{\vep_j}(u_j) \to u_+^{\gamma-1}$ locally in $L^1(Q)$, by virtue of \eqref{eq:L1FvepCh}. Therefore,
\begin{equation}\label{eq:WeissForLim1}
\WW^{\,\vep_j}(u_j,R_1) \to \WW(u,R_1), \qquad \WW^{\,\vep_j}(u_j,R_2) \to \WW(u,R_2),    
\end{equation}
as $j\uparrow \infty$, by the dominated convergence theorem.

\

By the energy estimates in \Cref{prop:LimDelta}, $\partial_t u_j \rightharpoonup \partial_t u$ weakly in $L^2(Q)$ and thus, for every $\UU \subset\subset Q$, $Z u_j \rightharpoonup Z u$ weakly in $L^2(\UU)$. By the same energy estimates, we easily deduce that the family $\{[\varrho/(-t)]^{1/2} Zu_j\}_{j\in\NN}$ is uniformly bounded in $L^2(S_r^-)$ and thus there exists $w \in L^2(S_r^-)$ such that $[\varrho/(-t)]^{1/2} Zu_j \rightharpoonup w$ weakly in $L^2(S_r^-)$, up to passing to a suitable subsequence.
Then a straightforward argument shows that $w = [\varrho/(-t)]^{1/2} Zu$ a.e. in $S_r^-$. Consequently, \eqref{eq:WeissForLim} follows by passing to the limit as $j \uparrow \infty$ into \eqref{eq:WeissForSemilinearIneq}, using \eqref{eq:WeissForLim1}, Fatou's lemma, and the lower semicontinuity of the $L^2(S_r^-)$ norm.
\end{proof}
\subsection{Blow-ups and the proof of \Cref{thm:MAIN2Blowups}}
\label{Subsec:Blowups}
Let $u$ be a weak solution to \eqref{eq:ParReacDiff} given by \Cref{thm:MAIN1}. Let $(x_\circ,t_\circ) \in \partial\{u > 0\}\cap Q$ and $r_\circ > 0$ such that $Q_{r_\circ}(x_\circ,t_\circ) \subset\subset Q$. 

\

We consider the blow-up family $\{u_r\}_{r>0}$ defined by
\begin{equation}\label{eq:BUFam}
u_r^{(x_\circ,t_\circ)}(x,t) := \frac{u(x_\circ +rx,t_\circ + r^2t)}{r^\beta}, \qquad \forall \,(x,t) \in Q_{r_\circ/r},
\end{equation}
and we study the limit as $r \downarrow 0$. Notice that, since $(x_\circ,t_\circ) \in \partial\{u > 0\}\cap Q$ and $r \downarrow 0$, it is enough to take $r_\circ = 1$ (the only change is that $u_r^{(x_\circ,t_\circ)}$ is defined in $Q_{1/r}$ instead of $Q_{r_\circ/r}$). Further, since $(x_\circ,t_\circ)$ is fixed once for all, we set $u_r := u_r^{(x_\circ,t_\circ)}$, dropping the dependence on $(x_\circ,t_\circ)$.    

\

Using the properties collected in \Cref{thm:MAIN1} together with \Cref{lem:L2H1StgConv} and \Cref{prop:WeissForSemLim} and scaling, we directly deduce that $u_r$ satisfies:

\

$\bullet$ \textbf{Weak formulation:} for every $R > 0$, every $\varphi \in C_c^\infty(Q_R)$ and every $r \in (0,\frac{1}{R})$, we have
\begin{equation}\label{eq:OVBUFam}
\int_{Q_R} \partial_t u_r \varphi + \nabla u_r \cdot\nabla\varphi + \gamma (u_r)_+^{\gamma-1} \varphi = 0,
\end{equation}
see \Cref{lem:FirstOVLim}.

\

$\bullet$ \textbf{Weak formulation (domain variations):} for every $R > 0$, every $\Phi \in C_c^\infty(Q_R;\RR^n)$ and every $r \in (0,\frac{1}{R})$, we have
\begin{equation}\label{eq:IVBUFam}
\int_{Q_R} \big( |\nabla u_r|^2 + 2(u_r)_+^\gamma \big) \dv_x \Phi  -  2 \nb u_r \cdot D_x \Phi \nabla u_r - 2\partial_t u_r \,(\nb u_r \cdot \Phi) = 0,
\end{equation}
see \Cref{lem:L2H1StgConv}.

\

$\bullet$ \textbf{The free boundary has measure zero:} for every $r > 0$, we have
\begin{equation}\label{eq:FB0MeasBUFam}
\mathcal{L}^{n+1}(\partial\{u_r > 0\}) = 0,
\end{equation}
see \Cref{lem:L1ConvFBMeas0}.

\

$\bullet$ \textbf{Optimal regularity and non-degeneracy estimates:} for every $R > 0$ and every $r \in (0,\tfrac{1}{4R})$, we have
\begin{equation}\label{eq:OptGrBUFam}
c R^\beta \le \sup_{Q_R^-} u_r  \le \sup_{Q_R} u_r \le C R^\beta,
\end{equation}
and 
\begin{equation}\label{eq:OptRegEStBUFam}
\sup_{Q_R} \, |\nabla (u_r^{1/\beta})|^2 + \sup_{Q_R} \, |\partial_t (u_r^{2/\beta})| \leq C,
\end{equation}
where $C,c > 0$ are independent of $r$, $R$, and $u$, see  \Cref{lem:StrCmp} and  \Cref{lem:HausConv}.

\

$\bullet$ \textbf{Local energy inequality:} for a.e. $R > 0$, every $r \in (0,\tfrac{1}{4R})$ and every $\psi \in C_c^\infty(B_R)$, we have
\begin{equation}\label{eq:LocalEstPartialTuBUFam}
\int_{Q_R} (\partial_tu_r)^2\psi^2 + \frac{1}{2} \int_{B_R} \big[ |\nb u_r|^2 + 2(u_r)_+^\gamma \big] \psi^2 \rd x \, \bigg|_{t= -R^2}^{t= R^2}  + 2 \int_{Q_R} \partial_tu_r  \psi \,(\nb u_r\cdot \nb \psi) \leq 0,
\end{equation}
see \Cref{lem:L2H1StgConv}.

\

$\bullet$ \textbf{Weiss monotonicity formula:}
for every $0<R_1<R_2$ and every $r \in (0,\frac{1}{2R})$, 
\begin{equation}\label{eq:WeissMonBUFam}
\WW(u_r,R_2) - \WW(u_r,R_1) \geq \int_{R_1}^{R_2} \frac{1}{\rho^{3 + \beta\gamma}} \left( \int_{S_\rho^-} \frac{1}{-t}\left[ Zu_r \right]^2 \varrho\right) \rd \rho,    
\end{equation}
see \Cref{prop:WeissForSemLim}. 

\

With the previous properties at hand we can now proceed with the proof of \Cref{thm:MAIN2Blowups}. We essentially follow the proofs of Section \ref{sec:convergence}.

\

\begin{proof}[Proof of \Cref{thm:MAIN2Blowups}] Let $n \geq 1$, $\gamma \in (0,1]$, $\alpha \in (0,1)$, and let $u_\circ \in C_c^{2+\alpha}(\RR^n)$ be nonnegative. Let $u$ be a nonnegative weak solution to \eqref{eq:ParReacDiff} given by \Cref{thm:MAIN1}, $(x_\circ,t_\circ) \in \partial\{ u > 0\}\cap Q$ and let $\{u_r\}_{r>0}$ be the blow-up family defined in \eqref{eq:BUFam}. The proof is divided in several steps as follows.
    
\

\emph{Step 1: H\"older compactness, optimal growth, non-degeneracy and regularity estimates.} First we show that for every $\nu \in (0,\frac{\beta}{2})$, there exist $r_j \downarrow 0$ and a nonnegative nontrivial $u_0 \in C^{\beta/2}(\RR^{n+1})$, such that
\begin{equation}\label{eq:BULocUnifCon}
u_{r_j} \to u_0 \quad \text{locally in } C^\nu(\RR^{n+1}),
\end{equation}
as $j \uparrow \infty$. Furthermore, $u_0(0,0) = 0$ and
\begin{equation}\label{eq:OptGrNonDegBU}
c R^\beta \le \sup_{Q_R^-} u_0 \le \sup_{Q_R} u_0 \le C R^\beta,
\end{equation}
for every $R > 0$, and 
\begin{equation}\label{eq:OptRegEStBU}
\sup_{\RR^{n+1}} \, |\nabla (u_0^{1/\beta})|^2 + \sup_{\RR^{n+1}} \, |\partial_t (u_0^{2/\beta})| \leq C,
\end{equation}
where $C,c > 0$ are independent of $R$. 

\

To see this, we combine \eqref{eq:OptGrBUFam} and \eqref{eq:OptRegEStBUFam} to deduce that for every $R > 0$ and every $r \in (0,\frac{1}{4R})$, there holds 
\[
\sup_{Q_R} \, |\partial_t (u_r^{2/\beta})| + |\nabla (u_r^{2/\beta})| + u_r^{2/\beta} \leq C (1 + R + R^2),
\]
for some $C > 0$ independent of $r$ (and $R$), that is, $\{u_r^{2/\beta}\}_{r\in(0,1/(4R))}$ is uniformly bounded in $C^{0,1}(Q_R)$. Combing this with the fact that the function $s \mapsto s^{\beta/2}$ is $C^{\beta/2}([0,\infty))$, we easily see that $\{u_r\}_{r\in(0,1/(4R))}$ is uniformly bounded in $C^{\beta/2}(Q_R)$.

Thus, the existence of $u_0$ and $r_j$ as above and \eqref{eq:BULocUnifCon} directly follow by the Arzel\`a-Ascoli theorem and a standard diagonal argument. In turn, \eqref{eq:BULocUnifCon} yields \eqref{eq:OptGrNonDegBU} by passing to the limit as $j \uparrow \infty$ in \eqref{eq:OptGrBUFam} (computed at $r = r_j$) and \eqref{eq:OptRegEStBU} follows by \eqref{eq:OptRegEStBUFam} by lower semicontinuity as in the proof of Lemma \ref{lem:StrCmp}.

\

\emph{Step 2: Energy compactness and energy estimates.} Second we show that for every $R > 0$,
\begin{equation}\label{eq:WeakConvBUSeq}
u_{r_j} \rightharpoonup u_0 \quad \text{ weakly in } H^1(Q_R)
\end{equation}
as $j \uparrow \infty$, up to passing to a suitable subsequence, and
\begin{equation}\label{eq:UnifEnBdBUFam1+2}
\int_{Q_R} |\nb u_0|^2 \leq CR^{n+2\beta}, \qquad \int_{Q_R} (\partial_tu_0)^2 \leq C R^{n + 2(\beta-1)},  
\end{equation}
for some new $C > 0$ independent of $R$.

\

Let us fix $R >0$ and let $j$ large enough, such that $r_j \in (0,\frac{1}{4R})$. On the one hand, combining the growth bounds \eqref{eq:OptGrBUFam} with the regularity estimate \eqref{eq:OptRegEStBUFam}, we easily obtain $|\nb u_{r_j}|^2 \leq C\beta^2 R^{\beta \gamma}$ in $Q_R$, where $C > 0$ is independent of $j$ and $R$. 
Thus, recalling that $\beta\gamma = 2(\beta-1)$, we get
\begin{equation}\label{eq:UnifEnBdBUFam1}
\int_{Q_R} |\nb u_{r_j}|^2 \leq CR^{n+2\beta}.    
\end{equation}
On the other hand, let us consider the local energy estimate \eqref{eq:LocalEstPartialTuBUFam}, with $\psi \in C_c^\infty(B_R)$ satisfying $\psi = 1$ in $B_{R/2}$ and $|\nb \psi| \leq C_n/R$, for some $C_n > 0$ depending only on $n$. 
Using again that $\beta\gamma = 2(\beta-1)$, and \eqref{eq:OptGrBUFam} and \eqref{eq:OptRegEStBUFam} as before, we easily see that
\[
\int_{B_R} \big[ |\nb u_{r_j}|^2 + 2(u_{r_j})_+^\gamma \big] \psi^2 \rd x \, \bigg|_{t=-R^2}^{t= R^2} \leq C R^{n + 2(\beta-1)},
\]
for some new $C > 0$ independent of $j$ and $R$. 
Consequently, applying the Young's inequality in \eqref{eq:LocalEstPartialTuBUFam} and using \eqref{eq:UnifEnBdBUFam1}, we obtain $\frac{1}{2}\int_{Q_R} (\partial_tu_{r_j})^2\psi^2  \leq C R^{n + 2(\beta-1)}$,
and thus
\begin{equation}\label{eq:UnifEnBdBUFam2}
\int_{Q_{R/2}} (\partial_tu_{r_j})^2 \leq C R^{n + 2(\beta-1)},
\end{equation}
for some new constant $C$ as above. Combining \eqref{eq:UnifEnBdBUFam1}, \eqref{eq:UnifEnBdBUFam2} and \eqref{eq:OptGrNonDegBU}, we deduce that $\{u_{r_j}\}_j$ is uniformly bounded in $H^1(Q_R)$, and thus \eqref{eq:WeakConvBUSeq} and \eqref{eq:UnifEnBdBUFam1+2} by the reflexivity of $H^1(Q_R)$ and lower semicontinuity of the $L^2(Q_R)$-norm.

\

\emph{Step 3: Hausdorff convergence.} We show that
\begin{equation}\label{eq:HausConvBU}
\overline{\{u_{r_j} > 0\}} \to \overline{\{u_0 > 0\}} \quad \text{ and } \quad \{u_{r_j} = 0\} \to \{u_0 = 0\}
\end{equation}
locally Hausdorff in $\RR^{n+1}$ and 
\begin{equation}\label{eq:L1ChiBU}
\chi_{\{u_{r_j} > 0\}} \to \chi_{\{u_0 > 0\}} \quad \text{in } L^1_{\loc}(\RR^{n+1}),
\end{equation}
as $j \uparrow \infty$. Furthermore,
\begin{equation}\label{eq:FB0MeasBU}
\mathcal{L}^{n+1}(\partial\{u_0 > 0\}) = 0.
\end{equation}
The proof is an adaptation of \Cref{lem:HausConv} and  \Cref{lem:L1ConvFBMeas0} and works as follows.

\

$\bullet$ The proof of \eqref{eq:HausConvBU} uses the non-degeneracy properties for $u_{r_j}$ and $u_0$ ---see \eqref{eq:OptGrBUFam} and \eqref{eq:OptGrNonDegBU}--- and the locally uniform convergence \eqref{eq:BULocUnifCon}, as in \Cref{lem:HausConv}. In this case, it is enough to fix a compact set $\KK \subset \RR^{n+1}$ and define
\begin{equation}
    U_j := \overline{\{ u_{r_j} > 0 \}} \cap \KK
    \quad \text{ and } \quad 
    U := \overline{\{ u > 0 \}} \cap \KK,
\end{equation} 
and, for a given $\sigma \in (0,1)$, 
\begin{equation}
    U_{j,\sigma} := \{(x,t): \text{dist}((x,t),U_j) \leq \sigma \}
    \quad \text{ and } \quad 
    U_\sigma := \{(x,t): \text{dist}((x,t),U) \leq \sigma \}.
\end{equation}
Give these definitions, the proof closely follows the lines of the proof of \Cref{lem:HausConv} and we skip it.

\

$\bullet$ The proof of \eqref{eq:FB0MeasBU} uses the optimal growth and non-degeneracy properties for $u_0$ ---see \eqref{eq:OptGrNonDegBU}--- and works exactly as in  \Cref{lem:L1ConvFBMeas0}; see the proof of \eqref{eq:FB0Meas}.

\

$\bullet$ Thanks to \eqref{eq:FB0MeasBU}, \eqref{eq:L1ChiBU} follows if $\chi_{\{u_{r_j} > 0\}} \to \chi_{\{u_0 > 0\}}$ a.e. in $\RR^{n+1}$. By uniform convergence ---see \eqref{eq:BULocUnifCon}---, we easily deduce that $\chi_{\{u_{r_j} > 0\}}(x) = 1$ for every fixed $x \in \{u_0>0\}$ and $j$ large enough. 
Hence, it is enough to check that $\chi_{\{u_{r_j} > 0\}}(x) = 0$ for every fixed $x \in \text{int}(\{u_0 = 0\})$ and $j$ large enough, but this readily follows by the local Hausdorff convergence \eqref{eq:HausConvBU}, proceeding as in the proof of \eqref{eq:L1FvepCh}.

\

\emph{Step 4: Locally strong convergence in $L^2(\RR;H^1(\RR^n))$ and weak formulation (domain variations).} In this step, we prove that
\begin{equation}\label{eq:L2H1StgConvBU}
u_{r_j} \to u_0 \quad \text{ locally in } L^2(\RR;H^1(\RR^n)),
\end{equation}
as $j \uparrow \infty$, up to passing to a suitable subsequence, and $u_0$ satisfies
\begin{equation}\label{eq:FirstIVLimBU}
\int_{\RR^{n+1}} \big( |\nabla u_0|^2 + 2(u_0)_+^\gamma \big) \dv_x \Phi  -  2 \nb u_0 \cdot D_x \Phi \cdot \nabla u_0 - 2\partial_t u_0 \,(\nb u_0 \cdot \Phi) = 0,
\end{equation}
for every $\Phi \in C_c^\infty(\RR^{n+1};\RR^{n+1})$.

\

$\bullet$ To show \eqref{eq:L2H1StgConvBU}, it is enough to combine \eqref{eq:WeakConvBUSeq} with the arguments of the proof of  \Cref{lem:L2H1StgConv} (\emph{Step 2}) as follows. Let us fix an open bounded set $\UU \subset \RR^{n+1}$ and $\eta \in C_c^\infty(\UU)$. It is enough to show
\begin{equation}\label{eq:StrongConvLocBU}
\int_\UU |\nb u_j|^2 \eta  \to \int_\UU |\nb u_0|^2 \eta,
\end{equation}
as $j\uparrow\infty$. By local uniform convergence and the classical Schauder theory, we know that $u_0$ is a classical solution to $\partial_tu_0 - \Delta u_0 =-\gamma u_0^{\gamma-1}$ in $\{u_0 > 0\}$, and thus $u_\sigma := (u_0-\sigma)_+$   
is a classical solution to $\partial_t u_\sigma - \Delta u_\sigma = - \gamma (u_\sigma + \sigma)^{\gamma-1}$ in $\{u > \sigma\}$, for every $\sigma > 0$ fixed. Multiplying the equation by $\varphi = u_\sigma \eta$ and integrating by parts in space, it follows that $u_\sigma$ satisfies \eqref{eq:StrongConUsigmaRel} in $\RR^{n+1}$. Then, letting $\sigma \downarrow 0$ along a suitable sequence and using that $\nb_{x,t} u_0 \in L^2(\UU)^{n+1}$ (see \eqref{eq:UnifEnBdBUFam1+2}), we deduce that $u_0$ satisfies \eqref{eq:IntGrad2Lim} in $\RR^{n+1}$, that is,
\begin{equation}\label{eq:IntGrad2LimBU}
\int_{\RR^{n+1}} |\nb u_0|^2 \eta = -\int_{\RR^{n+1}} \big[u_0 \partial_t u_0 \eta + u_0\nb u_0 \cdot \nb \eta \big] - \gamma \int_{\RR^{n+1}} u_0^\gamma \eta.
\end{equation}
On the other hand, testing the equation of $u_{r_j}$ with $\varphi = u_{r_j}\eta$ and letting $j \uparrow \infty$, we obtain
\[
\int_{\RR^{n+1}} |\nb u_{r_j}|^2 \eta \to -\int_{\RR^{n+1}} \big[u_0 \partial_t u_0 \eta + u_0\nb u_0 \cdot \nb \eta \big] - \gamma \int_{\RR^{n+1}} u_0^\gamma \eta,
\]
by uniform convergence and \eqref{eq:WeakConvBUSeq}. Combining this with \eqref{eq:IntGrad2LimBU}, \eqref{eq:StrongConvLocBU} follows. 

\

$\bullet$ The proof of \eqref{eq:FirstIVLimBU} works exactly as in  \Cref{lem:L2H1StgConv}: it is enough to pass to the limit in \eqref{eq:IVBUFam} (with $r = r_j$) as $j \uparrow \infty$, using \eqref{eq:L2H1StgConvBU} and \eqref{eq:WeakConvBUSeq}.

\

\emph{Step 5: Weak formulation.} Now we show that $(u_0)_+^{\gamma-1} \in L_\loc^1(\RR^{n+1})$ and $u_0$ satisfies
\begin{equation}\label{eq:FirstOVLimBU}
\int_{\RR^{n+1}} \partial_t u_0 \varphi + \nabla u_0 \cdot\nabla\varphi + \gamma (u_0)_+^{\gamma-1} \varphi = 0,
\end{equation}
for every $\varphi \in C_c^\infty(\RR^{n+1})$.

\

The proof cloesely follows \Cref{lem:FirstOVLim}. Since $\{(u_{r_j})_+^{\gamma-1}\}_{j\in\NN}$ is uniformly bounded in $L^1_{\loc}(\RR^{n+1})$ ---this easily follows by the equation of $u_{r_j}$ \eqref{eq:OVBUFam} and the energy bounds \eqref{eq:UnifEnBdBUFam1} and \eqref{eq:UnifEnBdBUFam2}---, by Fatou's lemma, it is enough check that $(u_{r_j})_+^{\gamma-1} \to \gamma (u_0)_+^{\gamma-1}$ a.e. in $\RR^{n+1}$, as $j \uparrow \infty$. 
But this is an immediate consequence of the uniform convergence \eqref{eq:BULocUnifCon}, the Hausdorff convergence $\{u_{r_j} = 0\} \to \{ u_0= 0 \}$, and the fact that $\LL^{n+1}(\partial\{u_0 > 0\}) = 0$; see \eqref{eq:FB0MeasBU}.

\

Now, let us fix $R > 0$, $j \in \NN$, $\sigma > 0$, and $\varphi \in C_c^\infty(Q_R)$. Testing the equation of $u_{r_j}$ with $\eta = \psi_\sigma'(u_{r_j})\varphi$, where $\psi$ is defined exactly as in the proof of \Cref{lem:FirstOVLim} (\emph{Step 2}), we obtain
\[
\int_{Q_R}  \psi_\sigma'(u_{r_j})(\partial_t u_{r_j} \varphi + \nb u_{r_j} \cdot \nb\varphi)  = -\int_{Q_R} \psi_\sigma''(u_{r_j}) |\nb u_{r_j}|^2 \varphi - \gamma \int_{Q_R} \psi_\sigma'(u_{r_j}) (u_{r_j})_+^{\gamma-1} \varphi.
\]
Then, by the uniform convergence \eqref{eq:BULocUnifCon}, the weak convergence $\partial_tu_{r_j} \rightharpoonup \partial_tu_0$ in \eqref{eq:WeakConvBUSeq}, and the strong convergence $\nabla u_{r_j} \rightharpoonup \nabla u_0$ in \eqref{eq:L2H1StgConvBU}, we may pass to the limit as $j \uparrow \infty$ to find
\begin{equation}\label{eq:EqpsisigmaBU}
\int_{Q_R}  \psi_\sigma'(u_0)(\partial_t u_0 \varphi + \nb u_0 \cdot \nb\varphi) = \int_{Q_R} \psi_\sigma''(u_0) |\nb u_0|^2 \varphi + \gamma \int_{Q_R} \psi_\sigma'(u_0) u_0^{\gamma-1} \varphi,
\end{equation}
The final part of the argument ---that is, deriving \eqref{eq:FirstOVLimBU} by passing to the limit as $\sigma \downarrow 0$ in \eqref{eq:EqpsisigmaBU}--- uses the optimal regularity \eqref{eq:OptRegEStBU}, the fact that $(u_0)_+^{\gamma-1} \in L^1_\loc(\RR^{n+1})$, and that $\LL^{n+1}(\partial\{u_0 > 0\}) = 0$, and works exactly as the last paragraphs of the proof of \Cref{lem:FirstOVLim}.

\

\emph{Step 6: Homogeneity of $u_0$.} Finally, we show that $u_0$ is parabolically $\beta$-homogeneous w.r.t. $(0,0)$ ``backward in time'', that is,
\begin{equation}\label{eq:HomogBU}
u_0(rx,r^2t) = r^\beta u_0(x,t),
\end{equation}
for every $(x,t) \in \RR^n \times(-\infty,0)$ and every $r > 0$.

\

Let us fix $0 < R_1 < R_2$. Then, by \eqref{eq:WeissMonBUFam}, we have
\[
\WW(u_{r_j},R_r) - \WW(u_{r_j},R_1) \geq \int_{R_1}^{R_2} \frac{1}{\rho^{3 + \beta\gamma}}  \left( \int_{S_\rho^-} \frac{1}{-t}\left[ Zu_{r_j} \right]^2 \varrho \right) \rd \rho,    
\]
for every $j$. Using the strong convergence of $\nabla u_{r_j}$ in $L^2_\loc(\RR^{n+1})^n$ given by \eqref{eq:L2H1StgConvBU}, we have $\WW(u_{r_j},R_1) \to \WW(u_0,R_1)$, $\WW(u_{r_j},R_2) \to \WW(u_0,R_2)$ as $j \uparrow \infty$ while, since $\WW(u_{r_j},R) = \WW(u,r_jR)$ and $\WW$ is monotone in $r$, we have $\WW(u_{r_j},R_2) - \WW(u_{r_j},R_1) \to 0$ as $j \uparrow \infty$. 
Consequently, arguing exactly as in the proof of \eqref{eq:WeissForLim} (that is, using uniform convergence, \eqref{eq:WeakConvBUSeq}, and Fatou's lemma), we deduce $Z u_0 = 0$ a.e. in $\RR^n\times (-4 R_2^2,-R_1^2)$. By the arbitrariness of $0 < R_1 < R_2$, it follows $Z u_0$ a.e. in $\RR^n\times (-\infty,0)$ which, in turn, yield \eqref{eq:HomogBU} since $u_0$ is continuous. 
\end{proof}

As a corollary we obtain a sharp bound on the parabolic Hausdorff dimension of the free boundary. The proof is nowadays standard dimension reduction procedure and we skip it: it works exactly as in \cite[Theorem 5.2]{Weiss1999:art}.
\begin{cor}[\cite{Weiss1999:art}, Theorem 5.2]
Let $\gamma \in (0,1]$ and let $u$ be a weak solution to \eqref{eq:ParReacDiff} given by \Cref{thm:MAIN1}. Then
\[
\text{dim}_{\mathcal{P}} (\partial\{u>0\}) \leq n+1,
\]
where $\text{dim}_{\mathcal{P}}$ denotes the parabolic Hausdorff dimension. 
\end{cor}
\subsection{Examples of homogeneous solutions}
\label{Subsec:SpecialSolutions}
Below we give some examples of $\beta$-homogeneous weak solutions. 

\

(i) \emph{Space-independent solutions.} Let $t_\circ \in \RR$. The function
\[
T(t) = \big[-\tfrac{2\gamma}{\beta}(t - t_\circ)\big]_+^{\beta/2}, \quad  t \in \RR,
\]
is a weak solution to \eqref{eq:EQSing} in $\RR^{n+1}$ in the sense of Definition \ref{def:WeakSolEIntro}; the proof is a direct computation. Notice that it can be easily obtained as the limit of the singular perturbation problem \eqref{eq:PerProbEps}. Indeed, let $\vep \in (0,1)$ and consider the initial-value problem 
\begin{equation}
\begin{cases}
T_\vep' = - f_\vep(T_\vep) \quad \text{ in } \RR \\
T_\vep(0) = \vep^\beta,
\end{cases}    
\end{equation}
where $f_\vep$ is defined as in Section \ref{subsub:RegApproach}. By the classical ODE's theory the above problem has a unique positive solution $T_\vep = T_\vep(t)$ defined in the whole $\RR$ and satisfying $T_\vep' < 0$. Consequently, by definition of $f_\vep$ and the monotonicity of $T_\vep$, we have $T_\vep \leq \vep^\beta$ in $[0,\infty)$, $T_\vep' = -\gamma T_\vep^{\gamma-1}$ in $(-\infty,0)$ and a direct integration shows
\[
T_\vep(t) = \big(\vep^2 -\tfrac{2\gamma}{\beta}t  \big)^{\beta/2}, \qquad t < 0.
\]
Passing to the limit as $\vep \downarrow 0$, $T_\vep \to \big(-\tfrac{2\gamma}{\beta}t  \big)_+^{\beta/2}$ locally uniformly in $\RR$, which is $T$ with $t_\circ = 0$.

\

(ii) \emph{One-dimensional time-independent solutions and extensions.} Let $-\infty \leq a \leq b \le \infty$ and $e \in \RR^n$ be a unit vector. The functions
\[
\psi_0(x) = (\tfrac{\sqrt{2}}{\beta})^\beta\big[ (a - (e\cdot x))_+^\beta + ((e\cdot x) - b)_+^\beta \big]
\]
are weak solutions to \eqref{eq:EQSing} in $\RR^{n+1}$ in the sense of Definition \ref{def:WeakSolEIntro}, with the convention $\psi_0(x) = (2/\beta)^\beta [(e\cdot x) - b]_+^\beta$ if $a = -\infty$ and $\psi_0(x) = (2/\beta)^\beta[a - (e\cdot x)]_+^\beta$ if $b = \infty$. 
The proof of this fact is a direct consequence of \Cref{remark:solutionsAreWeakSol}.

Even in this case, the function $\psi_0$ can be obtained as a limit of a singular semilinear problem. To see this, fix $\vep \in (0,1)$ and consider the initial-value problem
\begin{equation}
\begin{cases}
\psi_\vep' = \sqrt{2F_\vep(\psi_\vep)}  \quad \text{ in } \RR \\
\psi_\vep(0) = \vep^\beta,
\end{cases}    
\end{equation}
where $F_\vep$ is defined as in Section \ref{subsub:RegApproach}. Again, the classical ODE's theory shows the existence and uniqueness of a positive solution $\psi_\vep = \psi_\vep(x)$ defined in the whole $\RR$, satisfying $\psi_\vep' >  0$ and $\psi_\vep'' = f_\vep(\psi_\vep)$ in $\RR$. Using again the definition of $f_\vep$ and the monotonicity of $\psi_\vep$, we deduce $\psi_\vep \leq \vep^\beta$ in $(-\infty,0]$, $\psi_\vep' = \sqrt{2\psi_\vep^\gamma}$ in $(0,\infty)$. By direct integration again, we find
\[
\psi_\vep(x) = \big(\vep +\tfrac{\sqrt{2}}{\beta}x  \big)^{\beta}, \qquad x > 0,
\]
and taking the limit as $\vep \downarrow 0$, it follows $\psi_\vep \to (\tfrac{\sqrt{2}}{\beta})^\beta x_+^\beta$ locally uniformly in $\RR$, which is $\psi_0$ with $a = -\infty$ and $b=0$, up to a rotation.

Notice that one can easily extend the above families of solutions by adding fictitious variables as follows. Let $k \in \NN$ such that $k \le n-1$, $e \in \RR^{n-k}$ be a unit vector and let $x = (y,z) \in \RR^{n-k}\times\RR^k$. Then, the functions $\psi(x) := \psi_0(y)$ are weak solutions to \eqref{eq:EQSing} in $\RR^{n+1}$ in the sense of Definition \ref{def:WeakSolEIntro}.

\

(iii) \emph{Radial time-independent solutions and extensions.} Let $n \ge 2$ and $x_\circ \in\RR^n$. The function
\[
u(x) = c |x-x_\circ|^\beta, \qquad c = \bigg[\frac{\gamma}{\beta(n + \beta -2)}\bigg]^{\beta/2}.
\]
is a weak solution to \eqref{eq:EQSing} in $\RR^{n+1}$ in the sense of Definition \ref{def:WeakSolEIntro}. As above, one can build other solutions by adding fictitious variables: if $k \in \NN$ with $0 \le k \le n-2$, $x = (y,z) \in \RR^{n-k}\times\RR^k$ and $x_\circ = (y_\circ,z_\circ)$, the function
\[
u(x) = c |y-y_\circ|^\beta, \qquad c = \bigg[\frac{\gamma}{\beta(n - k + \beta -2)}\bigg]^{\beta/2}.
\]
is a weak solution to \eqref{eq:EQSing} in $\RR^{n+1}$ in the sense of Definition \ref{def:WeakSolEIntro}. 
The proof is a quite standard ``cut-off argument'' near $\{y=y_\circ\}$ and we omit the details.

\

As an application of the solutions constructed in the examples in (i) and (ii) above, we state the following corollary, which completes the proof of Theorem \ref{thm:MAIN1}. The proof is essentially given in \cite[Corollary 1]{Phillips1987:art}.
\begin{cor}[\cite{Phillips1987:art}, Corollary 1]\label{cor:CmptSupp}
Let $\gamma \in (0,1]$ and $\alpha \in (0,1)$. Let $u_\circ \in C_c^{2+\alpha}(\RR^n)$ be nonnegative, and let $u$ as in \Cref{lem:StrCmp}. Then $u$ has compact support in $\overline{Q}$.  
\end{cor}
\begin{proof}
Let $\vep_j$, $u_j := u_{\vep_j}$, and $u$ be as in \Cref{lem:StrCmp}. Let $T_j := T_{\vep_j}$ and $\psi_j := \psi_{\vep_j}$ be as in examples (i) and (ii), respectively. Fix an arbitrary unit vector $e \in \RR^n$,  and take $t_\circ > 0$ and $a_\circ > 0$ such that
\[
u_\circ(x,0) \leq \min\{ T_j(-t_\circ),\psi_j(a_\circ - (e\cdot x)) \}.
\]
Then, by the comparison principle, we deduce
\[
0 \le u_j(x,t) \leq \min\{ T_j(t-t_\circ),\psi_j(a_\circ - (e\cdot x))\},  
\]
for every $j$ and every $(x,t) \in Q$. The thesis follows by passing to the limit as $j \uparrow\infty$ and using the arbitrariness of $e$, together with the explicit expressions of $T$ and $\psi_0$.
\end{proof}

\begin{rem}
    The above corollary holds true in the case $\gamma = 0$ as well (taking into account the second comment after \Cref{thm:MAIN1}), but with a different proof, see \cite[Theorem 1.6]{CafVaz95}. 
    Indeed, our argument cannot work since the solution built in (i) is specific of the case $\gamma >0$. Note also that the solutions described in (iii) also require $\gamma >0$. 
\end{rem}
%
%
%
%
%
\section{Self-similar solutions}\label{sec:SelfSimilar}
In this section, we consider the problem of constructing parabolically $\beta$-homogeneous weak solutions to \eqref{eq:EQSing}, also called \emph{self-similar} solutions (see Remark \ref{rem:parabolicBetaHomog}), which play an important role in the study of the blow-up limits, as stated in \Cref{thm:MAIN2Blowups}.

First, in \Cref{Subsec:SSForward} we build radial self-similar solutions with \emph{unbounded} support; our construction works for every $\gamma \in [0,1]$ and, in addition, when $\gamma = 0,1$, the solutions can be written explicitly in terms of special functions (Kummer's and Tricomi's functions). 
However, it has to be stressed that such solutions are parabolically $\beta$-homogeneous ``forward in time'', so it is not clear if they can show up as blow-up limits of general weak solutions.
As will be seen, some arguments use crucially the sign of the coefficients appearing in the ODE \eqref{eq:SSProfile0} below, and thus they are not directly applicable to do an analogous argument to build $\beta$-homogeneous ``backward in time'' solutions, where one has to analyze a different ODE ---see \eqref{eq:SSProfileShriniking} in \Cref{Subsec:SSBackward}.

Second, in \Cref{Subsec:SSBackward} we consider radial solutions with \emph{bounded} support. In this case the analysis seems much harder: when $\gamma = 0,1$, we complement the theory in \cite{CafVaz95} ($\gamma = 0$), giving an explicit formula for the self-similar profile in terms of special functions while, when $\gamma = 1$, we show non-existence of self-similar profiles. What happens in the range $\gamma \in (0,1)$ is left as an open problem.

\subsection{Self-similar solutions with unbounded support.}
\label{Subsec:SSForward}
In this subsection we study the existence of self-similar solutions to \eqref{eq:EQSing} ``forward in time'', that is, solutions with form
\begin{equation}\label{eq:SSAnsatz}
u(x,t) = t^\frac{\beta}{2} U \big(t^{-\frac{1}{2}} x \big), \qquad t > 0
\end{equation}
for some $U:\RR^n \to [0,\infty)$, called self-similar profile. Plugging the ansatz \eqref{eq:SSAnsatz} into the equation of $u$, we easily find
\[
\Delta U - \frac{\xi}{2}\cdot\nb U + \frac{\beta}{2}U = \gamma U^{\gamma-1} \quad \text{ in } \{U > 0\},    
\]
where the gradient and the Laplacian are taken w.r.t. $\xi := |t|^{-\frac{1}{2}} x$.  Taking into account the free boundary condition \eqref{eq:FBCondIntro}, it is natural to look for nontrivial profiles $U: \RR^n \to [0,\infty)$ satisfying 
\begin{equation}\label{eq:DerFBSelfS}
\begin{cases}
\Delta U + \frac{\xi}{2}\cdot\nb U - \frac{\beta}{2}U = \gamma U^{\gamma-1} \quad &\text{ in } \{U > 0\} \\
|\nb(U^{1/\beta})| =  \frac{\sqrt{2}}{\beta} \quad &\text{ in } \partial\{ U > 0 \}, 
\end{cases}
\end{equation}
Specifically, we study the existence of \emph{radial} profiles $U = U(r)$ ($r := |\xi|$) constructed as follows: we look for $R > 0$ and $U \in C^2((R,\infty);\RR_+)$ such that
\begin{equation}\label{eq:SSProfile0}
\begin{cases}
U'' + \big( \frac{n-1}{r} + \frac{r}{2} \big) U' - \frac{\beta}{2}U = \gamma U^{\gamma-1} \quad \text{ in } (R,\infty) \\
U(R) = 0, (U^{1/\beta})'(R) = \frac{\sqrt{2}}{\beta},
\end{cases}
\end{equation}
with $U$ extended to zero in $[0,R)$. Then, $\{ U = 0\} = \{r \leq R\}$  which, in terms of $u$, means
\begin{equation}\label{eq:PosSetSSGammaPosShriniking}
\{u = 0\} = \{(x,t): t > 0, \,|x|^2 \leq R^2|t|\},
\end{equation}
and thus each time-slice of $\{u = 0\}$ is a ball of radius $R\sqrt{|t|}$: the positivity set $\{u(t) > 0\}$ is unbounded for every $t > 0$ and the contact set $\{u(t) = 0\}$ expands in time. 

Further, as a consequence of the limit in \eqref{eq:AsymptBehavSS} below, our construction will allow to extend $u$ up to $t=0$, by setting
\begin{equation}\label{eq:ExtOfuSSt0}
u(x,0) := \lim_{t\downarrow 0} \, t^\frac{\beta}{2} U \big(t^{-\frac{1}{2}} |x| \big) = c|x|^\beta,   
\end{equation}
for some suitable $c > 0$.

\

The plan has three key steps. First of all, we show that the auxiliary problem
\begin{equation}\label{eq:SSProfile}
\begin{cases}
U'' + \big( \frac{n-1}{r} + \frac{r}{2} \big) U' - \frac{\beta}{2}U = \gamma U^{\gamma-1} \quad \text{ in } (R,\infty) \\
U(R) = U'(R) = 0,
\end{cases}
\end{equation}
has at least one solution: this is not a trivial fact, since the right-hand side of the equation becomes singular when $r \sim R^+$ and the classical ODE's theory does not apply. In the second step, we prove that the solution $U$ satisfies the FB condition in \eqref{eq:SSProfile0}, which is also non-trivial, since one must characterize the behavior of $U$ close to $R$. Lastly, we study the asymptotic behavior of $U$ as $r \uparrow \infty$: as anticipated, this is crucial to extend the self-similar solution up to $t=0$ as in \eqref{eq:ExtOfuSSt0}. Given this, we can state the main result of this section.
\begin{thm}\label{thm:ExSSProfiles}
Let $\gamma \in [0,1]$. Then, for every $R > 0$, there exists a solution $U$ to \eqref{eq:SSProfile0} which is positive, increasing, and satisfies
\begin{equation}\label{eq:AsymptBehavSS}
\lim_{r\to\infty} \frac{U(r)}{r^\beta} = c    
\end{equation}
for some $c > 0$. 
\end{thm}
\begin{rem}
We notice that, in view of \Cref{remark:solutionsAreWeakSol}, if $U$ is a solution to \eqref{eq:SSProfile0} as in the statement above and $\gamma \in (0,1]$, then $u$ defined as in \eqref{eq:SSAnsatz} is a weak solution to \eqref{eq:EQSing} in $Q$ in the sense of \Cref{def:WeakSolEIntro}. If $\gamma = 0$, $u$ is a solutions in the sense of domain variations in $Q$, see \eqref{eq:InnerVarGamma0}.
\end{rem}
We first consider the limit cases $\gamma = 0,1$, where the analysis is easier and the solutions are essentially explicit in terms of special functions.
\begin{lem}\label{lem:SSProfileUnBdd0}
Let $\gamma = 0$. Then, for every $R > 0$, \eqref{eq:SSProfile0} has exactly one solution $U$ given by
\begin{equation}\label{eq:SelfSimProfGamma0}
U(r) = \frac{2\sqrt{2}}{R \mathbf{W}(\frac{R^2}{4})} e^{\frac{R^2-r^2}{4}} \big[ \mathbf{M}(\tfrac{n+1}{2},\tfrac{n}{2},\tfrac{R^2}{4}) \mathbf{U}(\tfrac{n+1}{2},\tfrac{n}{2},\tfrac{r^2}{4}) - \mathbf{U}(\tfrac{n+1}{2},\tfrac{n}{2},\tfrac{R^2}{4}) \mathbf{M}(\tfrac{n+1}{2},\tfrac{n}{2},\tfrac{r^2}{4}) \big ]\end{equation}
where $s \mapsto \mathbf{M}(\tfrac{n+1}{2},\tfrac{n}{2},s)$ and $s \mapsto \mathbf{U}(\tfrac{n+1}{2},\tfrac{n}{2},s)$ are the Kummer's and Tricomi's functions, respectively, and $\mathbf{W}$ is the Wronskian of $\mathbf{M}$ and $\mathbf{U}$. 

Furthermore, $U$ is positive and increasing in $(R,\infty)$, and
\begin{equation}\label{eq:AsymptBehavSS0}
\lim_{r\to\infty} \frac{U(r)}{r} = \tfrac{R^{n-1}}{2^{n-1/2}}\,\mathbf{U}(\tfrac{n+1}{2},\tfrac{n}{2},\tfrac{R^2}{4}).
\end{equation}
\end{lem}
\begin{proof} 
For every $R > 0$, it is not difficult to check that \eqref{eq:SSProfile0} for $\gamma = 0$ (thus $\beta = 1$) has a unique solution $U$ which is positive and increasing in $(R,\infty)$. Now, setting $s := \tfrac{r^2}{4}$ and $e^{-s}V(s) := U(r)$, we get
\[
sV'' + (\tfrac{n}{2} - s) V' - \tfrac{n+1}{2} V = 0 \quad \text{ in } (S,\infty), \quad \text{ where } S := R^2/4,
\]
that is, $V$ is a solution to the Confluent Hypergeometric Equation with parameters $a = \tfrac{n+1}{2}$ and $b = \tfrac{n}{2}$ (see \cite[Section 13]{AS1972} and \cite[Chapter 1]{Slater1960:book}). By \cite[Table 1]{Mathews2022:art} (see Case 5.A. and Case 1.C.), we have
\[
V(s) = p \mathbf{M}(\tfrac{n+1}{2},\tfrac{n}{2},s) + q \mathbf{U}(\tfrac{n+1}{2},\tfrac{n}{2},s),
\]
where $\mathbf{M}$ and $\mathbf{U}$ denote the Kummer's and Tricomi's functions, respectively, and $p,q \in \RR$ are free parameters. Therefore
\[
U(r) = e^{-\frac{r^2}{4}} \big[ p \mathbf{M}(\tfrac{n+1}{2},\tfrac{n}{2},\tfrac{r^2}{4}) + q \mathbf{U}(\tfrac{n+1}{2},\tfrac{n}{2},\tfrac{r^2}{4}) \big].
\]
Fixing $p$ and $q$ according to the initial conditions, \eqref{eq:SelfSimProfGamma0} easily follows. Finally, using that (see \cite[Formula~13.4.8 and Formula~13.1.8]{AS1972})
\begin{equation}\label{eq:AsymMUsLarge}
\begin{aligned}
\mathbf{M}(a,b,s) &= \tfrac{\Gamma(b)}{\Gamma(a)} e^s s^{a-b} \big[ 1 + (a-1)(a-b) s^{-1} + O(s^{-2}) \big], \\
\mathbf{U}(a,b,s) &= s^{-a} \big[ 1 + O(s^{-1}) \big], 
\end{aligned}    
\end{equation}
as $s \uparrow \infty$ and that (see \cite[Formula 13.2.34]{DLMF})
\begin{equation}\label{eq:WronskianMU}
\mathbf{W}(s) = - \frac{\Gamma(b)}{\Gamma(a)} \frac{e^s}{s^b},
\end{equation}
where $\Gamma$ denotes the Gamma function, the limit in \eqref{eq:AsymptBehavSS0} follows. 

\end{proof}
\begin{lem}\label{lem:SSProfileUnBdd1}
Let $\gamma = 1$. Then, for every $R > 0$, \eqref{eq:SSProfile0} has exactly one solution $U$ given by
\begin{equation}\label{eq:SelfSimProfGamma1}
U(r) = e^{-\frac{r^2}{4}} \big[ p(\tfrac{r^2}{4})\mathbf{M}(\tfrac{n+2}{2},\tfrac{n}{2},\tfrac{r^2}{4}) + q(\tfrac{r^2}{4})\mathbf{U}(\tfrac{n+2}{2},\tfrac{n}{2},\tfrac{r^2}{4}) \big],
\end{equation}
where $\mathbf{M}$ and $\mathbf{U}$ are as in \Cref{lem:SSProfileUnBdd0}, and
\begin{equation}
p(s) = \frac{\Gamma(\frac{n+2}{2})}{\Gamma(\frac{n}{2})} \int_{R^2/4}^s \tau^{\frac{n}{2}-1}\mathbf{U}(\tfrac{n+2}{2},\tfrac{n}{2},\tau) \, \rd \tau, \qquad q(s) = - \frac{\Gamma(\frac{n+2}{2})}{\Gamma(\frac{n}{2})} \int_{R^2/4}^s \tau^{\frac{n}{2}-1}\mathbf{M}(\tfrac{n+2}{2},\tfrac{n}{2},\tau) \, \rd \tau.    
\end{equation}
Furthermore, $U$ is positive and increasing in $(R,\infty)$, and
\begin{equation}\label{eq:AsymptBehavSS1}
\lim_{r\to\infty} \frac{U(r)}{r^2} = \frac{1}{4}  \int_{R^2/4}^\infty \tau^{\frac{n}{2}-1}\mathbf{U}(\tfrac{n+2}{2},\tfrac{n}{2},\tau) \, \rd \tau.
\end{equation}
\end{lem}
\begin{proof} We first notice that, for every $R > 0$, \eqref{eq:SSProfile} (with $\gamma = 1$ and $\beta = 2$) has a unique solution $U$ which is positive and increasing in $(R,\infty)$. Further, it satisfies $U''(R) = 1$ and thus $U(r) \sim \frac{1}{2}(r-R)^2$ as $r \downarrow R$, which is exactly the FB condition in \eqref{eq:SSProfile} when $\beta = 2$.

Now, as above, we set $s := \tfrac{r^2}{4}$, $e^{-s}V(s) := U(r)$ and $S = R^2/4$ and deduce that
\[
\begin{cases}
sV'' + (\tfrac{n}{2} - s) V' - \tfrac{n+2}{2} V = e^s \quad \text{ in } (S,\infty) \\
V(S) = V'(S) = 0
\end{cases}
\]
that is, $V$ satisfies a non-homogeneous Confluent Hypergeometric Equation with parameter $a = \frac{n+2}{2}$ and $b = \frac{n}{2}$. By \cite[Table 1]{Mathews2022:art} (see Case 1.A. and Case 5.C.), $s \mapsto \mathbf{M}(\tfrac{n+2}{2},\tfrac{n}{2},s)$ and $s \mapsto \mathbf{U}(\tfrac{n+2}{2},\tfrac{n}{2},s)$ are two independent solutions to the associated homogeneous equation and thus, by the method of ``variation of parameters'', we find
\[
V(s) = p(s)\mathbf{M}(\tfrac{n+2}{2},\tfrac{n}{2},s) + q(s)\mathbf{U}(\tfrac{n+2}{2},\tfrac{n}{2},s),
\]
where $p$ and $q$ are uniquely determined by $p(S) = q(S) = 0$ (this easily follows by imposing the initial conditions) and
\[
p'(s) = \frac{\Gamma(\frac{n+2}{2})}{\Gamma(\frac{n}{2})} s^{\frac{n}{2}-1}\mathbf{U}(\tfrac{n+2}{2},\tfrac{n}{2},s), \qquad q'(s) = - \frac{\Gamma(\frac{n+2}{2})}{\Gamma(\frac{n}{2})} s^{\frac{n}{2}-1}\mathbf{M}(\tfrac{n+2}{2},\tfrac{n}{2},s), 
\]
where we have used \eqref{eq:WronskianMU} and the equation of $V$. Furthermore, in light of \eqref{eq:AsymMUsLarge}, we have
\[
p'(s) \sim \frac{\Gamma(\frac{n+2}{2})}{\Gamma(\frac{n}{2})} s^{-2}, \qquad q'(s) \sim - s^{\frac{n}{2}} e^s,
\]
as $s \uparrow \infty$ and therefore,
\[
p(s) \to \frac{\Gamma(\frac{n+2}{2})}{\Gamma(\frac{n}{2})} \int_S^\infty s^{\frac{n}{2}-1}\mathbf{U}(\tfrac{n+2}{2},\tfrac{n}{2},s) \, \rd s =: \frac{\Gamma(\frac{n+2}{2})}{\Gamma(\frac{n}{2})} \, p_\infty        \quad \text{ as } s\uparrow\infty,
\]
while, by the de l'H\^opital's theorem, $q(s) \sim - s^{\frac{n}{2}}e^s$ as $s \uparrow \infty$. Consequently, by \eqref{eq:AsymMUsLarge} again, 
\[
V(s) \sim \big( p_\infty s - \tfrac{1}{s} \big) e^s,
\]
as $s \uparrow \infty$. Re-writing everything in terms of $U$ and $r$, both \eqref{eq:SelfSimProfGamma1} and \eqref{eq:AsymptBehavSS1} follow.
\end{proof}
\begin{rem} The solution $U$ found above can be made even more explicit (see \cite[Section 2]{CafPetSha04}):
\[
U(r) = (2n+r^2) \bigg( \frac{1}{2n+R^2} - 2 R^n e^{\frac{R^2}{4}} \int_R^r \frac{e^{-\frac{\tau^2}{4}}}{\tau^{n-1}(2n+\tau^2)^2} \, \rd \tau  \bigg) - 1.
\]
\end{rem}
Finally, we consider the range $\gamma \in (0,1)$. As anticipated, the proof of  \Cref{thm:ExSSProfiles} is split in two main steps. Respect to the limit cases $\gamma = 0,1$, we essentially proceed via qualitative methods.
\begin{lem}\label{lem:ApproxSS} Let $R > 0$ and $\gamma \in (0,1)$. Then there exists a solution $U$ to \eqref{eq:SSProfile}.
\end{lem}
\begin{proof} Let us fix $R > 0$ and $\gamma \in (0,1)$.

\

\emph{Step 1: Approximation.} Let us consider the family $\{V_\delta\}_{\delta \in (0,1)}$ made by the solutions to the Cauchy problem
\begin{equation}\label{eq:SSProfileDelta}
\begin{cases}
V'' + \big( \frac{n-1}{r} + \frac{r}{2} \big) V' - \frac{\beta}{2}V = \gamma (V + \delta)^{\gamma-1} \quad \text{ in } (R,\infty) \\
V(R) = V'(R) = 0.
\end{cases}
\end{equation}
For every fixed $\delta > 0$, set $V := V_\delta$. Since $V''(R) = \gamma \delta^{\gamma-1} > 0$, we have $V,V' > 0$ in some maximal interval $(R,R_{\textrm{max}})$. Actually, $R_{\textrm{max}} = \infty$: if not, it must be $V(R_{\textrm{max}}) > 0$ and $V'(R_{\textrm{max}}) = 0$, and thus, using the equation, $V''(R_{\textrm{max}}) > 0$ which is impossible by definition of $R_{\textrm{max}}$. Consequently, $V,V' > 0$ in the maximal interval of definition of the solution, which by monotonicity and the decay of the right-hand side of the equation it can be easily shown to be $(R,\infty)$.

\

\emph{Step 2: Uniform $L^\infty$ bounds.} In this step, we show that for every $R_\star > R$, 
\begin{equation}\label{eq:LinftyUnifBnd}
\{V_\delta\}_{\delta \in (0,1)} \, \text{ is uniformly bounded in } \, L^\infty((R,R_\star)).     
\end{equation}
Multiplying the equation in \eqref{eq:SSProfileDelta} by $V'$, integrating between $R$ and $r$ and using that $(U+\delta)^\gamma \leq U^\gamma + \delta^\gamma$, we find
\begin{equation}\label{eq:DiffIneqUnif}
(V')^2 \leq \tfrac{\beta}{2} V^2 + 2V^\gamma  \quad \text{ in } (R,\infty),     
\end{equation}
and therefore 
\begin{equation}\label{eq:DifInqLinfinity}
\int_0^{V(r)} \big( \tfrac{\beta}{2} s^2 + 2s^\gamma \big)^{-\frac{1}{2}} \rd s \leq r-R  \qquad \forall \, r \in (R,\infty).     
\end{equation}
On the one hand, if $r \in \{V > 1\}$, we have
\[
\int_0^{V(r)} \big( \tfrac{\beta}{2} s^2 + 2s^\gamma \big)^{-\frac{1}{2}} \rd s \ge \int_1^{V(r)} \big( \tfrac{\beta}{2} s^2 + 2s^\gamma \big)^{-\frac{1}{2}} \rd s \ge
k \int_1^{V(r)} s^{-1} \rd s = k \log V(r),
\]
where we have set $k = (\tfrac{\beta}{2} + 2)^{-\frac{1}{2}}$ and used that $\tfrac{\beta}{2}s^2 + 2s^\gamma \leq ks^2$ whenever $s \geq 1$. As a consequence,
\begin{equation}\label{eq:LinfinitySS2}
V(r) \leq e^{(r-R)/k}  \qquad \forall \, r \in \{V > 1\}.
\end{equation}

On the other hand, if $r \in \{V \leq 1\}$, we have $0 < s < V(r) \leq 1$ and thus $\tfrac{\beta}{2}s^2 + 2s^\gamma \leq (\tfrac{\beta}{2} + 2)s^\gamma$. Consequently, 
\[
\int_0^{V(r)} \big( \tfrac{\beta}{2} s^2 + 2s^\gamma \big)^{-\frac{1}{2}} \rd s \ge k \int_0^{V(r)} s^{-\frac{\gamma}{2}} \rd s = \beta k V^{\frac{1}{\beta}}(r),
\]
which yields
\begin{equation}\label{eq:LinfinitySS1}
V(r) \leq ( \beta k )^\beta (r-R)^\beta \qquad \forall \, r \in \{V\leq 1\}.
\end{equation}
Combining \eqref{eq:LinfinitySS2} with \eqref{eq:LinfinitySS1} and using that $V \geq 0$, the uniform bound \eqref{eq:LinftyUnifBnd} follows. 

\

\emph{Step 3: Uniform positivity.} Now, we show that there exist $R_0 > R$ and $\delta_0 \in (0,1)$ such that
\begin{equation}\label{eq:UnifPositivity}
V_\delta(r) \geq (\min\{r-R,R_0-R\})^2 \qquad \forall\, r>R, \; \forall \delta \in (0,\delta_0).
\end{equation}
Notice that, since $V' > 0$, it is enough to prove that there exist $R_0 > R$ and $\delta_0 \in (0,1)$ such that
\[
V_\delta(r) \geq (r-R)^2 \qquad \forall\, r \in (R,R_0), \; \forall \delta \in (0,\delta_0).
\]
Assume not. Then, given any $R_0 > R$ and $\delta_0 \in (0,1)$, there are $\delta \in (0,\delta_0)$ and $r_0 \in (R,R_0)$ such that $V(r_0) \leq (r_0-R)^2$. On the other hand, since $V''(R) = \gamma\delta^{\gamma-1}$, there is $R_\delta > R$ such that $V(r) \geq \frac{\gamma}{4}\delta^{\gamma-1}(r-R)^2$ for every $r \in (R,R_\delta]$. Therefore, taking $\delta_0$ small enough, we may assume $R_\delta < r_0 < R_0$. But then $V(R_\delta) \geq \frac{\gamma}{4}\delta^{\gamma-1}(r-R)^2 > (r_0-R)^2 \geq V(r_0)$, which is in contradiction with $V' > 0$. 

\

\emph{Step 4: Power decay as $r\downarrow R$.} In this step, we prove first that there exists $R_1 > R$ and $C > 0$ such that
\begin{equation}\label{eq:LinearBdSS}
V_\delta(r) \leq C(r-R) \qquad \forall\, r \in (R,R_1), \; \forall \delta \in (0,1).
\end{equation}
This will be obtained via comparison principle and the uniform bounds obtained in \emph{Step 2} as follows. First we notice that, by \eqref{eq:LinfinitySS1} and \eqref{eq:LinfinitySS2}, if $R_1 = R + \vep$ and $\vep > 0$ is small enough (depending only $\gamma$), then $V(R_1) \leq 2$ (and $V(R) = 0$ by construction). Further, since $V,V'>0$, we have
\[
-V'' - c V' + \tfrac{\beta}{2}V \leq 0 \quad \text{ in } (R,R_1),
\]
where $c := \frac{n-1}{R} + \frac{R_1}{2}$. Thus, if $W$ is a solution to
\begin{equation}\label{eq:SuperSolUpperbnd}
\begin{cases}
W'' + c W' - \tfrac{\beta}{2}W = 0 \quad \text{ in } (R,R_1) \\
W(R)=0, \; W(R_1)=2,
\end{cases}
\end{equation}
we deduce $V \leq W$ in $[R,R_1]$ by comparison, and therefore \eqref{eq:LinearBdSS} follows if $W(r) \leq C(r-R)$ for every $r \in (R,R+\vep)$ and some $C > 0$. Now, \eqref{eq:SuperSolUpperbnd} can be explicitly solved: setting
\[
\lambda_1 := \frac{\sqrt{c^2+2\beta} + c}{2}, \qquad \lambda_2 := \frac{\sqrt{c^2+2\beta} - c}{2},
\]
it is not difficult to check that
\[
W(r) = \frac{2}{e^{\lambda_2\vep} - e^{-\lambda_1\vep}} \big[ e^{\lambda_2(r-R)} - e^{-\lambda_1(r-R)} \big], \qquad r \in (R,R_1).
\]
Noticing that $W'(R) = \frac{2(\lambda_1 + \lambda_2)}{e^{\lambda_2\vep} - e^{-\lambda_1\vep}} > 0$, our claim \eqref{eq:LinearBdSS} follows.

Notice that, by \eqref{eq:LinearBdSS}, we have $V_\delta \leq 1$ in $(R,\tilde{R}_0)$ for some $\tilde{R}_0 \in (R,R_1)$ independent of $\delta$ and thus, by \eqref{eq:LinfinitySS1} and \eqref{eq:DiffIneqUnif}, we deduce the existence of $C > 0$ (depending only on $\gamma$ and $R$) such that
\begin{equation}\label{eq:DecayBdsApprox}
V_\delta(r) \leq C(r-R)^\beta, \qquad  V_\delta'(r) \leq C(r-R)^{\beta-1} \qquad \forall \, r\in (R,\tilde{R}_0), \; \forall \, \delta \in (0,1).  
\end{equation}

\

\emph{Step 5: Compactness and passage to the limit.} Let us fix $\vep > 0$ small. 
Then \eqref{eq:LinftyUnifBnd} and \eqref{eq:UnifPositivity} imply that $C_\vep^{-1} \leq (V_\delta + \delta)^{\gamma-1} \leq C_\vep$ in $(R+\vep,R+\frac{1}{\vep})$ for some $C_\vep > 0$ and every $\delta \in (0,\delta_0)$, where $\delta_0$ is as in \eqref{eq:UnifPositivity}. Consequently, by elliptic regularity, we deduce that $\{V_\delta\}_{\delta\in(0,1)}$ is uniformly bounded in $C^{2,\alpha}((R+\frac{\vep}{2}, R + \frac{1}{2\vep}))$ and, so a standard diagonal argument combined with the Arzel\`a-Ascoli theorem shows that the exist $U \in C^2((R,\infty))$ and a sequence $\delta_j \downarrow 0$ such that
\[
U_{\delta_j} \to U \quad \text{ in } C_\loc^2((R,\infty)),
\]
as $j \uparrow \infty$. Passing to the limit as $j\uparrow\infty$ into \eqref{eq:DecayBdsApprox} (with $\delta = \delta_j$), we immediately see that $U$ and $U'$ can be continuously extended up to $r=R$ by setting $U(R) = U'(R) =0$. On the other hand, passing to the limit as $j\uparrow\infty$ into the equation of $V_{\delta_j}$, we deduce that $U$ is a solution to \eqref{eq:SSProfile} and our statement follows.
\end{proof}
\begin{lem}\label{lem:SelfSim2} 
Let $R > 0$, $\gamma \in (0,1)$, and let $U$ be a solution to \eqref{eq:SSProfile}. Then $U,U' > 0$ in $(R,\infty)$, and both the FB condition in \eqref{eq:SSProfile0} and \eqref{eq:AsymptBehavSS} hold.
\end{lem}
\begin{proof} Let $U$ be a solution to \eqref{eq:SSProfile}. Again we divide the proof in some steps.

\

\emph{Step 1: Positivity, monotonicity and limit at $r\uparrow \infty$.} Since $U''(r) > 0$ as $r \downarrow R$, we deduce $U,U' > 0$ in $(R,R_{\textrm{max}})$ for some maximal $R_{\textrm{max}} > R$. Then the same argument of \emph{Step 1} in the proof of  \Cref{lem:ApproxSS} shows that $U,U' > 0$ in $(R,\infty)$. In particular, it follows that $U(r) \to \ell$ as $r \uparrow \infty$, for some $\ell \in (0,\infty]$.

In fact, we have $\ell = \infty$. Indeed, assume by contradiction $\ell \in \RR$. Then, directly from the equation of $U$, we obtain
\[
r^{1-n} e^{-\frac{r^2}{4}} \big( r^{n-1} e^{\frac{r^2}{4}} U' \big)' = U'' + \big( \tfrac{n-1}{r} + \tfrac{r}{2} \big) U' \geq \gamma\ell^{\gamma-1} 
\]
for some $\ell > 0$. Integrating and using that $U(R) = U'(R) = 0$, it follows
\[
U(r) \geq \gamma\ell^{\gamma-1} \int_R^r s^{1-n} e^{-\frac{s^2}{4}} \int_R^s \rho^{n-1} e^{\frac{\rho^2}{4}} \rd\rho\, \rd s,
\]
and thus, using that $\int_R^s \rho^{n-1} e^{\frac{\rho^2}{4}} \rd\rho \sim 2 s^{n-2} e^{\frac{s^2}{4}}$ as $s \uparrow \infty$ (this is easily follows by the de l'H\^opital's theorem), we deduce $\ell = \infty$, a contradiction.

\

\emph{Step 2: Proof of the FB condition in \eqref{eq:SSProfile0}.} Let us fix $\vep > 0$. Then, the same argument of \emph{Step 2} in the proof of  \Cref{lem:ApproxSS} shows that $U$ satisfies \eqref{eq:DiffIneqUnif}. Consequently, since $U(R) = 0$ and $\gamma < 2$, there exists $R_\vep > R$ such that
\[
(U')^2 \leq 2(1+\vep)U^\gamma \quad \text{ in } (R,R_\vep).
\]
On the other hand, since $U(R) = U'(R) = 0$, the equation of $U$ yields $U'' \geq \gamma(1-\vep)U^{\gamma-1}$ in $(R,R_\vep)$ and therefore, multiplying by $U'> 0$, we deduce
\[
(U')^2 \geq 2(1-\vep)U^\gamma \quad \text{ in } (R,R_\vep).
\]
Combining the above two inequalities and integrating, we find
\[
\frac{\sqrt{2}}{\beta}(1-\vep) (r-R) \leq U^{1/\beta}(r) \leq \frac{\sqrt{2}}{\beta}(1+\vep) (r-R) \qquad \forall \, r \in (R,R_\vep)  
\]
and our claim follows by dividing by $r-R$ and using the arbitrariness of $\vep$. 

\

\emph{Step 3: Proof of \eqref{eq:AsymptBehavSS}.} We proceed as in the proof of  \Cref{lem:SSProfileUnBdd1}, setting $s := \tfrac{r^2}{4}$, $e^{-s}V(s) := U(r)$ and $S := R^2/4$, to obtain
\[
\begin{cases}
sV'' + (\tfrac{n}{2} - s) V' - \tfrac{n+\beta}{2} V = e^s g(s) \quad \text{ in } (S,\infty) \\
V(S) = V'(S) = 0,
\end{cases}
\]
where $g(s) := \gamma U^{\gamma-1}(2\sqrt{s})$. Since $U$ is smooth in $(R,\infty)$ and $U^{1/\beta}(r) \sim (\sqrt{2}/\beta)(r-R)$ as $r \downarrow R$ (see \emph{Step 2} above), it easily follows that $g \in L_\loc^1([S,\infty))\cap C^\infty((S,\infty))$, and thus $V$ satisfies a non-homogeneous Confluent Hypergeometric Equation with parameter $a = \frac{n+\beta}{2}$ and $b = \frac{n}{2}$. By \cite[Table 1]{Mathews2022:art} (see Case 1.A. and Case 1.C.) and the method of ``variation of parameters'', it must hold that
\[
V(s) = p(s)\mathbf{M}(\tfrac{n+\beta}{2},\tfrac{n}{2},s) + q(s)\mathbf{U}(\tfrac{n+\beta}{2},\tfrac{n}{2},s),
\]
where $p$ and $q$ satisfy 
\[
p'(s) = \frac{\Gamma(\frac{n+\beta}{2})}{\Gamma(\frac{n}{2})} s^{\frac{n}{2}-1}\mathbf{U}(\tfrac{n+\beta}{2},\tfrac{n}{2},s)g(s), \qquad q'(s) = - \frac{\Gamma(\frac{n+\beta}{2})}{\Gamma(\frac{n}{2})} s^{\frac{n}{2}-1}\mathbf{M}(\tfrac{n+\beta}{2},\tfrac{n}{2},s)g(s), 
\]
with $p(S) = q(S) = 0$; this is not obvious, but can be checked by imposing the initial conditions and using the fact that $g \in L_\loc^1([S,\infty))$. By \eqref{eq:AsymMUsLarge} again, it holds that 
\[
p'(s) \sim \frac{\Gamma(\frac{n+\beta}{2})}{\Gamma(\frac{n}{2})} s^{-\frac{2+\beta}{2}} g(s) \quad \text{ and } \quad q'(s) \sim - s^{\frac{n+\beta}{2}-1} e^s g(s) \quad \text{ as } s \uparrow \infty,
\]
and since $g$ is bounded in $[S+1,\infty)$ and $\beta > 0$, we deduce $p(s) \uparrow \big[ \Gamma(\tfrac{n+\beta}{2})/\Gamma(\tfrac{n}{2})\big]p_\infty > 0$ as $s \uparrow \infty$, while $q(s) \sim - s^{\frac{n+\beta}{2}-1} e^s g(s)$ as $s \uparrow \infty$, by the de l'H\^opital's theorem and the definition of $g$ (here we use the fact that $g(s) \downarrow 0$ as $s\uparrow\infty$).
Consequently, proceeding similarly to the end of the proof of  \Cref{lem:SSProfileUnBdd1}, we obtain
\[
V(s) \sim \big( p_\infty s^{\frac{\beta}{2}} - \tfrac{g(s)}{s} \big) e^s,
\]
as $s \uparrow \infty$ and \eqref{eq:AsymptBehavSS} follows with $c = \frac{p_\infty}{4}$.
\end{proof}

\subsection{Self-similar solutions with bounded support.} 
\label{Subsec:SSBackward}
As mentioned at the beginning of the section, we study the existence of self-similar solutions to \eqref{eq:EQSing} ``backward in time'' with bounded support, that is, solutions with form 
\begin{equation}\label{eq:SSAnsatzShrinking}
u(x,t) = |t|^\frac{\beta}{2} U \big(|t|^{-\frac{1}{2}} x \big), \qquad t < 0
\end{equation}
where the profile $U:\RR^n \to [0,\infty)$ has bounded support. Notice that $u$ is an \emph{ancient} solution (that is, defined for $t \in (-\infty,0)$) with \emph{extinction time} $t = 0$: as mentioned in the introduction, such solutions are closely related to the \emph{shrinkers} in the MCF theory.

Proceeding as in the previous section, it is natural to look for nontrivial profiles $U: \RR^n \to [0,\infty)$ satisfying 
\begin{equation}\label{eq:DerFBSelfSShrinking}
\begin{cases}
\Delta U - \frac{\xi}{2}\cdot\nb U + \frac{\beta}{2}U = \gamma U^{\gamma-1} \quad &\text{ in } \{U > 0\} \\
|\nb(U^{1/\beta})| =  \frac{\sqrt{2}}{\beta} \quad &\text{ in } \partial\{ U > 0 \}, 
\end{cases}
\end{equation}
where the gradient and the Laplacian are taken w.r.t. $\xi := |t|^{-\frac{1}{2}} x$. 
Precisely, we study the existence of \emph{radial} profiles $U = U(r)$ ($r := |\xi|$) obtained as follows: we look for $\ell,R > 0$ and $U \in C^2([0,R);\RR_+)$ such that
\begin{equation}\label{eq:SSProfileShriniking}
\begin{cases}
U'' + \big( \frac{n-1}{r} - \frac{r}{2} \big) U' + \frac{\beta}{2}U = \gamma U^{\gamma-1} \quad \text{ in } (0,R) \\
U(0) = \ell, \; U'(0) = 0 \\
U(R) = 0, \; (U^{1/\beta})'(R) = -\frac{\sqrt{2}}{\beta},
\end{cases}
\end{equation}
with $U$ extended to zero in $[R,\infty)$. Then,
\begin{equation}\label{eq:PosSetSSGammaPosShriniking}
\{u > 0\} = \{(x,t): t < 0, \,|x|^2 < R^2|t|\},
\end{equation}
and thus each time-slice of $\{u > 0\}$ is a ball of radius $R\sqrt{|t|}$, collapsing to a point at $t = 0$.

Establishing existence (or non-existence) of solutions to \eqref{eq:SSProfileShriniking} seems to be a highly nontrivial problem ---as mentioned, some methods of the previous section do not apply here. 
For this reason, we restrict ourselves to the limit cases $\gamma = 0,1$: we will see that, if $\gamma = 0$, \eqref{eq:SSProfileShriniking} has exactly one explicit solution while, if $\gamma = 1$, solutions do no exist. The methods we use involve special functions, linearity, and seem not easily adaptable to treat the full range $\gamma \in (0,1)$.

\

The case $\gamma = 0$, was first treated in \cite[Section 1]{CafVaz95}: the authors proved existence of a unique solution via qualitative methods. Here we complement their analysis giving an explicit formula for the solution in terms of special functions.

\begin{lem}\label{lem:ShrinkersGamma0}
Let $\gamma = 0$ and $n\geq 1$. Then \eqref{eq:SSProfileShriniking} has exactly one solution $(\ell,R,U)$ given by
\begin{equation}\label{eq:ShrinkProfGamma0}
U(r) = \ell \, \mathbf{M}(-\tfrac{1}{2},\tfrac{n}{2},\tfrac{r^2}{4}),
\end{equation}
where $s \mapsto \mathbf{M}(-\tfrac{1}{2},\tfrac{n}{2},s)$ is the Kummer's function, and
\begin{equation}\label{eq:RadiusHeightShrinkProfGamma0}
R = 2 \sqrt{s_\star} , \qquad \frac{1}{\ell} = - \sqrt{\frac{s_\star}{2}} \frac{\rd}{\rd s}\mathbf{M}(-\tfrac{1}{2},\tfrac{n}{2},s_\star), 
\end{equation}
where $s_\star > 0$ is the only positive zero of $\mathbf{M}$.
\end{lem}
\begin{proof}
Let $U$ be a solution to the ODE in \eqref{eq:SSProfileShriniking} with $\gamma = 0$. Setting $s := \tfrac{r^2}{4}$ and $V(s) := U(r)$, we easily find
\[
sV'' + (\tfrac{n}{2} - s) V' + \tfrac{1}{2} V = 0 \quad \text{ in } (0,S),
\]
where $S := R^2/4$, that is, $V$ is a solution to the Confluent Hypergeometric Equation with parameters $a = -\tfrac{1}{2}$ and $b = \tfrac{n}{2}$ (see \cite[Section 13]{AS1972} and \cite[Chapter 1]{Slater1960:book}). By \cite[Table 1]{Mathews2022:art} (see Case 1.A. and Case 1.C.), we have
\[
V(s) = p \mathbf{M}(-\tfrac{1}{2},\tfrac{n}{2},s) + q \mathbf{U}(-\tfrac{1}{2},\tfrac{n}{2},s),
\]
where $\mathbf{M}$ and $\mathbf{U}$ denote the Kummer's and Tricomi's functions, respectively, and $p,q \in \RR$ are free parameters, and therefore
\[
U(r) = p \mathbf{M}(-\tfrac{1}{2},\tfrac{n}{2},\tfrac{r^2}{4}) + q \mathbf{U}(-\tfrac{1}{2},\tfrac{n}{2},\tfrac{r^2}{4}).
\]

Now, let us assume first $n \geq 2$. Then, \cite[Table 1]{Mathews2022:art} and \cite[Section 1.5]{Slater1960:book} yield $|\mathbf{U}(-\tfrac{1}{2},\tfrac{n}{2},s)| \to +\infty$ as $s \downarrow 0$ which, in turn, forces $q = 0$. Further, since $\mathbf{M}(-\tfrac{1}{2},\tfrac{n}{2},0) = 1$ and we require $U(0)=\ell$, then necessarily $p = \ell$.
Thus $U$ is as in \eqref{eq:ShrinkProfGamma0}. Noticing that $\frac{\rd}{\rd s} \mathbf{M}(-\tfrac{1}{2},\tfrac{n}{2},0) = -\frac{1}{n}$, we also have $U'(0) = 0$ and thus it is sufficient to check that there are $R,\ell > 0$ such that the last two conditions in \eqref{eq:SSProfileShriniking} are satisfied. 
Since by \cite[Formula 13.9.1]{DLMF} (see also \cite[Chapter 6]{Slater1960:book}), $\mathbf{M}(-\tfrac{1}{2},\tfrac{n}{2},s)$ has exactly one positive zero $s_\star$, $U$ has exactly one zero at $R = 2 \sqrt{s_\star}$. Further, $s_\star$ is simple, that is, $\frac{\rd}{\rd s}\mathbf{M}(-\tfrac{1}{2},\tfrac{n}{2},s_\star) < 0$ (this can be easily checked by a contradiction argument, using that $s \mapsto \mathbf{M}(-\tfrac{1}{2},\tfrac{n}{2},s)$ is analytic and differentiating its equation). Consequently, taking $\ell$ as in \eqref{eq:RadiusHeightShrinkProfGamma0}, we find $U'(R) = -\sqrt{2}$ as claimed.

\

Lastly, when $n = 1$, we have $\mathbf{U}(-\tfrac{1}{2},\tfrac{1}{2},s) = \sqrt{s}$, i.e., $U(r) = p \mathbf{M}(-\tfrac{1}{2},\tfrac{1}{2},\tfrac{r^2}{4}) + \frac{q}{2} r$. Thus, since $U'(0) = 0$, we deduce $q = 0$ and the second part of the argument follows as above. Equivalently, one can directly integrate the equation of $U$ to find 
\[
U(r) = \ell 
\Big( 1 - \tfrac{1}{2} \int_0^r\int_0^\tau e^{\frac{\rho^2}{4}} \rd \rho \, \rd \tau \Big),
\]
where $R$ is the only solution to $\int_0^R\int_0^r e^{\tau^2/4} \rd \tau \rd r= 2$ and $\ell$ satisfies $\ell \int_0^R e^{r^2/4} \rd r = 2 \sqrt{2}$. 
\end{proof}
The case $\gamma = 1$ is easier. We show that any solution to \eqref{eq:SSProfileShriniking} must be a parabola with vertex at $r = 0$ which, clearly, cannot vanish quadratically at its only positive zero.
\begin{lem}\label{lem:ShrinkersGamma1}
Let $\gamma = 1$. Then \eqref{eq:SSProfileShriniking} does not have solutions.
\end{lem}
\begin{proof}
Assume by contradiction $U$ is a solution to \eqref{eq:SSProfileShriniking}. Then, $\tilde{U} := U - 1$ is a solution to 
\[
\tilde{U}'' + \Big( \frac{n-1}{r} - \frac{r}{2} \Big) \tilde{U}' + \tilde{U} = 0 \quad \text{ in } (0,R),
\]
with $\tilde{U}(0) = \ell - 1$ and $\tilde{U}'(0) = 0$. We set as before $s := \tfrac{r^2}{4}$ and $V(s) := \tilde{U}(r)$, to find
\[
sV'' + (\tfrac{n}{2} - s) V' + V = 0 \quad \text{ in } (0,S),
\]
where $S = R^2/4$, which is the Confluent Hypergeometric Equation with parameters $a = -1$ and $b = \tfrac{n}{2}$. Then, proceeding exactly as in the proof of  \Cref{lem:ShrinkersGamma0}, we find
\[
U(r) = (\ell - 1) \mathbf{M}(-1,\tfrac{n}{2},\tfrac{r^2}{4}) + 1.
\]
But since 
\[
\mathbf{M}(a,b,s) := \sum_{k=0}^\infty \frac{(a)_k}{(b)_k} \frac{s^k}{k!} = 1 + \frac{a}{b}s + \frac{a(a+1)}{b(b+1)} \frac{s^2}{2} + \frac{a(a+1)(a+2)}{b(b+1)(b+2)} \frac{s^3}{6} \ldots,
\]
for every $a \in \RR$ and every $b \in \RR\setminus\{0,-1,-2,\ldots\}$, it follows that $\mathbf{M}(-1,b,s) = 1 - s/b$ and thus
\[
U(r) = \ell - \frac{\ell - 1}{2n} r^2.
\]
Consequently, if $\ell \in (0,1]$, $U$ is positive everywhere while, if $\ell > 1$, $U'(R) < 0$, where $R > 0$ is the only positive zero of $U$, contradicting the last condition in \eqref{eq:SSProfileShriniking}.
\end{proof}
As already mentioned, the above methods seem not apply when $\gamma \in (0,1)$. Motivated by some partial analytic results and some numerical computations, we close the section with the following conjecture.
\begin{conj}
Let $\gamma \in (0,1)$. Then \eqref{eq:SSProfileShriniking} does not have solutions.    
\end{conj}
%
%
%
%
%
%
%
%
%
%
%
\section{Traveling waves}\label{sec:TravelingWaves}
In this final section, we study the existence of nonnegative weak solutions to \eqref{eq:EQSing}
%
%
with traveling wave (TW) form, that is, 
\begin{equation}\label{eq:TWAnsatz}
u(x,t) = \phi(e\cdot x-ct),
\end{equation}
where $\phi: \RR \to [0,\infty)$ is the solution's profile, $c \in \RR$ is the profile's speed, and $e$ is a fixed unit vector: $u$ is an \emph{eternal} solution (i.e., defined for all times $t\in \RR$) identified by a fixed profile which travels along the direction $e$ with speed $c$. Specifically, we look for nonzero profiles $\phi: \RR \to [0,\infty)$ satisfying 
\begin{equation}\label{eq:DerFBTW}
\begin{cases}
c\phi' + \phi'' = \gamma \phi^{\gamma-1} \quad &\text{ in } \{ \phi > 0 \} \\
|(\phi^{1/\beta})'| =  \frac{\sqrt{2}}{\beta} \quad &\text{ in } \partial\{ \phi > 0 \}, 
\end{cases}
\end{equation}
where $\phi'$ denotes differentiation w.r.t. $\xi := e\cdot x - ct$ and the derivative at FB points is taken from \emph{inside} $\{\phi > 0\}$, i.e., for every $\xi_0 \in \partial\{ \phi > 0 \}$, we require 
\[
\lim_{\substack{\xi \to \xi_0 \\ \xi \in \{\phi > 0\}}} |(\phi^{1/\beta})'(\xi)| =  \frac{\sqrt{2}}{\beta}.
\]
Notice that this is nothing more than the natural FB condition in \eqref{eq:FBCondIntro}. 
\begin{defn}
If $(\phi,c)$ satisfies \eqref{eq:DerFBTW}, we say that $\phi$ is an admissible profile and $c$ is an admissible speed. Implicitly, positive and identically zero profiles are not admissible profiles.   
\end{defn}
Let us begin with the limit cases $\gamma = 0$ and $\gamma = 1$, where the admissible profiles are fully explicit.

\

\emph{The case $\gamma = 0$.} If $\gamma = 0$ (and $\beta = 1$), the analysis of \eqref{eq:DerFBTW} is quite simple (see \cite{CafVaz95}), since the equation can be explicitly integrated. It can be easily checked that for every $c \in \RR$, there are exactly two admissible profiles $\phi_c^+$ and $\phi_c^-$ satisfying:

\

$\bullet$  $\phi_c^+ = 0$ in $(-\infty,0]$ and $(\phi_c^+)' > 0$ in $(0,\infty)$, given by
    \[
    \phi_c^+(\xi) =
    \begin{cases}
     \sqrt{2}\,\xi_+                     \quad &\text{ if } c = 0 \\
     \frac{\sqrt{2}}{c}(1-e^{-c\xi})_+ \quad &\text{ if } c > 0 \\
     \frac{\sqrt{2}}{|c|}(e^{|c|\xi}-1)_+ \quad &\text{ if } c < 0.
    \end{cases}
    \]
\

$\bullet$ $\phi_c^- = 0$ in $[0,\infty)$ and $(\phi_c^-)' < 0$ in $(-\infty,0)$, given by
    \[
    \phi_c^-(\xi) =
    \begin{cases}
     \sqrt{2}(-\xi)_+                     \quad &\text{ if } c = 0 \\
     \frac{\sqrt{2}}{c}(e^{-c\xi}-1)_+ \quad &\text{ if } c > 0 \\
     \frac{\sqrt{2}}{|c|}(1-e^{|c|\xi})_+ \quad &\text{ if } c < 0.
    \end{cases}
    \]
Furthermore, any other admissible profile with speed $c$ has the form $\psi_c (\xi) = a \phi_c^+(\xi - \xi_+) + b \phi_c^-(\xi - \xi_-)$, for some $a,b \in \{0,1\}$ with $a^2 + b^2 \not= 0$ and $\xi_\pm \in \RR$ with $\xi_- \leq \xi_+$. This last part is not obvious: we refer the reader to the final part of the proof of  \Cref{thm:ExTW} below, where we treat the range $\gamma \in (0,1]$, since the same argument applies to the case $\gamma = 0$.
%
%
%
%

\

\emph{The case $\gamma = 1$.} Also the case $\gamma = 1$ (that is, $\beta = 2$) can be treated easily, since we can essentially reduce the analysis of \eqref{eq:DerFBTW} to the case $\gamma = 0$. 
Indeed, for $c\neq 0$, $\tilde{\phi}$ is a solution to the equation in \eqref{eq:DerFBTW} with $\gamma = 0$ if and only if $\phi(\xi) = \tilde{\phi}(\xi) + \frac{1}{c}\xi$ is a solution to the equation in \eqref{eq:DerFBTW} with $\gamma = 1$ (the case $c=0$ deals with the simple equation $\phi'' = 1$). Then, imposing $\phi(0) = 0$ and $|(\phi^{1/2})'(0)| = \frac{\sqrt{2}}{2}$, we obtain that for every $c \in \RR$, there are exactly two admissible profiles $\phi_c^+$ and $\phi_c^-$ satisfying:

\

$\bullet$ $
\phi_c^+ = 0$ in $(-\infty,0]$ and $(\phi_c^+)' > 0$ in $(0,\infty)$, given by
    \[
    \phi_c^+(\xi) =
    \begin{cases}
     \frac{1}{2}\xi_+^2                     \quad &\text{ if } c = 0  \vspace{2mm}\\
     \begin{cases}
     \frac{1}{c^2}(e^{-c\xi} - 1 + c\xi)  &\text{ if } \xi > 0 \\
     0                                    &\text{ if } \xi \le 0 
     \end{cases}
    \quad &\text{ if } c \not= 0
         \end{cases}
    \]

\

$\bullet$ $\phi_c^- = 0$ in $[0,\infty)$ and $(\phi_c^-)' < 0$ in $(-\infty,0)$, given by
    \[
    \phi_c^-(\xi) =
    \begin{cases}
     \frac{1}{2}(-\xi)_+^2                     \quad &\text{ if } c = 0  \vspace{2mm}\\
     \begin{cases}
     0                                    &\text{ if } \xi \ge 0 \\
     \frac{1}{c^2}(e^{-c\xi} - 1 + c\xi)  &\text{ if } \xi < 0 
     \end{cases}
    \quad &\text{ if } c \not= 0.
         \end{cases}
    \]
As is the case $\gamma = 0$, any other admissible profile with speed $c$ has the form $\psi_c (\xi) = a \phi_c^+(\xi - \xi_+) + b \phi_c^-(\xi - \xi_-)$, for some $a,b \in \{0,1\}$ with $a^2 + b^2 \not= 0$ and $\xi_\pm \in \RR$ with $\xi_- \leq \xi_+$, and the proof works exactly as in  \Cref{thm:ExTW}. Notice that the analysis we present below works in the case $\gamma = 1$ as well, and yields essentially the same result using qualitative methods only.
%
%

\

Given this, we can state and prove the main result of the section (Theorem \ref{thm:ExTW}): it provides a full classification of the admissible profiles in the range $\gamma \in (0.1]$. As mentioned above, the final step of the proof (see \emph{Step 3} below) works in the case $\gamma = 0$ as well and thus, thanks to the independent analysis of the case $\gamma = 0$ presented above, the admissible profiles are classified for every $\gamma \in [0,1]$.
\begin{thm}\label{thm:ExTW}
Let $\gamma \in (0,1]$ and $c_\beta := (\tfrac{2}{\beta^2})^{\frac{\beta}{2}}$. Then, for every $c \in \RR$, there are exactly two admissible profiles $\phi_c^+$ and $\phi_c^-$ satisfying: 

\

$\bullet$ $\phi_c^+ = 0$ in $(-\infty,0]$, $(\phi_c^+)' > 0$ in $(0,\infty)$, and $\phi_c^+(\xi) \to \infty$ as $\xi \to \infty$. In addition:

    \smallskip

    \begin{itemize}
        \item[--] If $c = 0$, we have $\phi_0^+(\xi) = c_\beta \xi_+^\beta$. 
        \item[--] If $c > 0$, then $\phi_c^+(\xi) \sim \big( \tfrac{2\gamma}{\beta c} \xi \big)^{\beta/2}$ as $\xi \to \infty$. 
        \item[--] If $c < 0$, $\phi_c^+(\xi) \sim e^{-c\xi}$ as $\xi \to +\infty$, up to a multiplicative constant.
    \end{itemize}

\

$\bullet$ $\phi_c^- = 0$ in $[0,+\infty)$, $(\phi_c^-)' < 0$ in $(-\infty,0)$, and $\phi_c^-(\xi) \to \infty$ as $\xi \to -\infty$. In addition:

    \smallskip

    \begin{itemize}
        \item[--] If $c = 0$, we have $\phi_0^-(\xi) = c_\beta (-\xi)_+^\beta$. 
        \item[--] If $c > 0$, then $\phi_c^-(\xi) \sim e^{-c\xi}$ as $\xi \to -\infty$, up to a multiplicative constant.
        \item[--] If $c < 0$, $\phi_c^-(\xi) \sim \big( \tfrac{2\gamma}{ \beta c} |\xi| \big)^{\beta/2}$ as $\xi \to -\infty$.
    \end{itemize}
Furthermore, for every admissible profile $\psi_c$ with speed $c$ there are $a,b \in \{0,1\}$ with $a^2 + b^2 \not= 0$, and $\xi_\pm \in \RR$ with $\xi_- \leq \xi_+$, such that
\begin{equation}\label{eq:AllTWProfiles}
\psi_c (\xi) = a \phi_c^+(\xi - \xi_+) + b \phi_c^-(\xi - \xi_-).
\end{equation}
\end{thm}
\begin{rem}\label{rem:TWWeakSolutions}
Similar to what happens for the self-similar solutions constructed in Section \ref{sec:SelfSimilar}, thanks to \Cref{remark:solutionsAreWeakSol}, if $\psi_c$ is an admissible profile and $\gamma \in (0,1]$, then $u(x,t) = \psi_c(e\cdot x - ct)$ is a weak solution to \eqref{eq:EQSing} in $\RR^{n+1}$ in the sense of \Cref{def:WeakSolEIntro}. If $\gamma = 0$, $u$ is a solution in the sense of domain variations in $Q$, see \eqref{eq:InnerVarGamma0}.
\end{rem}
\begin{proof}[Proof of \Cref{thm:ExTW}]
Let $\phi$ be a solution to the equation in \eqref{eq:DerFBTW} such that, for $\xi_0 \in \RR$, it holds $(\phi(\xi_0),\phi'(\xi_0)) = (\phi_0,\phi_1)$ with $\phi_0 > 0$, $\phi_1 \in \RR$ ---note that any admissible profile must satisfy this for some values $\xi_0, \phi_0,\phi_1 \in \RR$ with $\phi_0>0$.
By the standard ODE theory, such solution exists (and it is positive) in a maximal nonempty interval $I := (m_-, m_+)$ with $\xi_0\in I$ for which, if $m_\pm \neq \pm \infty$, then $\phi(m_\pm) = 0$ or $\phi(\xi) \to +\infty$ as $\xi \to m_\pm$.
If $I= \RR$, $\phi$ is not an admissible profile. 

Hence, our goal is to find solutions $\phi$ attaining the level $0$ at some finite value (either $m_-$ or $m_+$), which can be assumed to be $\xi=0$ by translation invariance.
This will be accomplished with a phase-plane analysis as follows.

\

Let us set $U = \phi^{1/\beta}$. Whenever $U > 0$, we have
\begin{equation}
    \begin{split}
        \phi' &= \beta U^{\beta-1}U', \\
        \phi'' &= \beta(\beta-1)U^{\beta-2} (U')^2 + \beta U^{\beta-1} U''.
    \end{split}
\end{equation}
Plugging such expressions into the equation of $\phi$ and recalling that $\beta - 2 = \beta(\gamma-1)$, we obtain
\[
c U U' + (\beta-1)(U')^2 + U U'' = \frac{\gamma}{\beta}, 
\]
which, since $\beta-1 = \tfrac{\beta\gamma}{2}$, is equivalent to
\begin{equation}\label{eq:SysTW0}
\begin{cases}
U'= V \\
UV'= \tfrac{\gamma}{\beta} - c UV - \tfrac{\gamma\beta}{2} V^2.
\end{cases}
\end{equation}
Now, let $\xi = \xi(\tau)$ be the solution to $\frac{\rd \xi}{\rd \tau} = U(\xi)$ with $\xi(0) = \xi_0$ (local existence and uniqueness follow since $U$ is positive and smooth close to $\xi_0$). Using this re-parametrization, the above system becomes
\begin{equation}\label{eq:SysTW}
\begin{cases}
\dot{U} = UV \\
\dot{V} = \tfrac{2}{\beta\gamma} \big(  \tfrac{2}{\beta^2} - \nu UV - V^2\big),
\end{cases}
\end{equation}
where $\nu := \tfrac{2c}{\beta\gamma}$ and $\dot{U}$ denotes differentiation w.r.t. $\tau$. 
Whenever $UV\not=0$, the trajectories of \eqref{eq:SysTW} ---and thus the trajectories of \eqref{eq:SysTW0}--- are the graphs of the solutions to
\begin{equation}\label{eq:ODETW}
\frac{2}{\beta\gamma}\frac{\rd V}{\rd U} = \frac{2/\beta^2 - \nu UV - V^2}{UV}.
\end{equation}
To find admissible profiles, we look for trajectories in the region $\{U>0\}$ of the phase-plane $(U,V)$ that ``reach'' the critical points $(0, \pm\sqrt{2}/\beta) =:P^\pm$.
In the remaining part of the analysis, we should distinguish between the cases $c<0$, $c=0$, and $c>0$.
Note, however, that we may restrict to $c\geq 0$, since the case $c<0$ can be reduced to $c>0$ changing $c$ by $-c$ and $V$ by $-V$ ---i.e., the phase portrait in the case $c<0$ is simply the one for $c>0$ reflected evenly across $V=0$ (reversing the direction in which trajectories are followed). 
This symmetry is illustrated in \Cref{fig:Trajectories}, in which a summary of the phase plane analysis performed next is described.

\begin{figure}[!ht]
\includegraphics[scale=0.43]{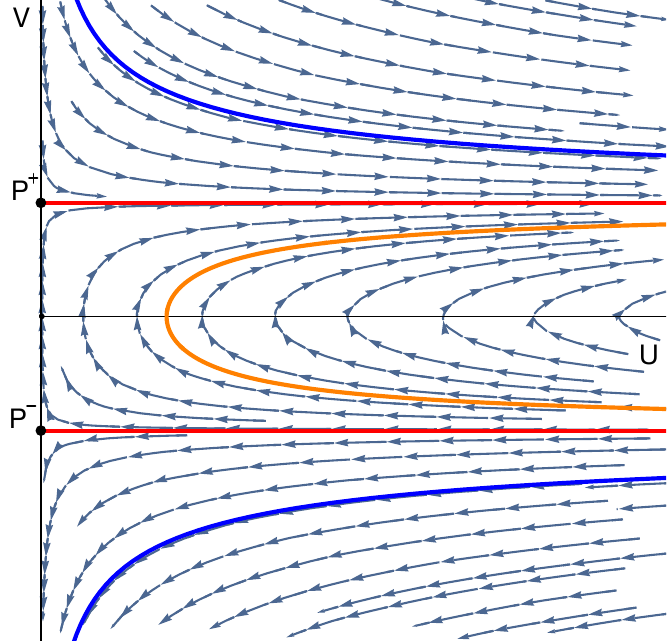} \quad
\includegraphics[scale=0.43]{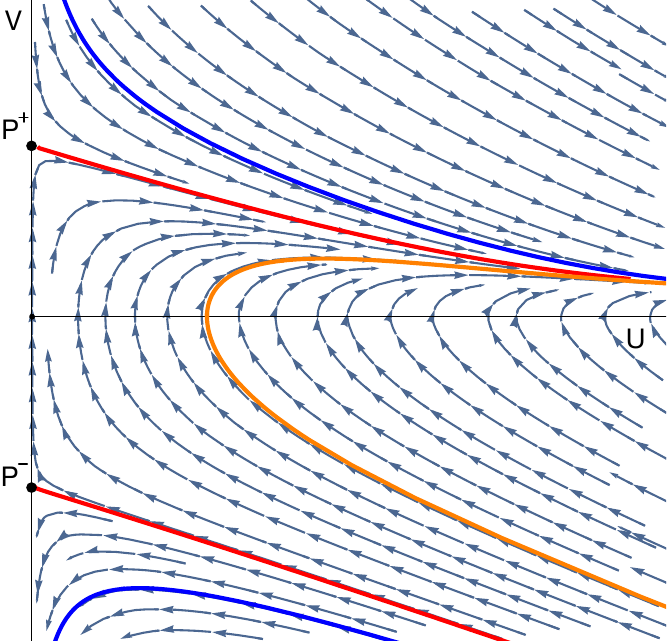} \quad
\includegraphics[scale=0.43]{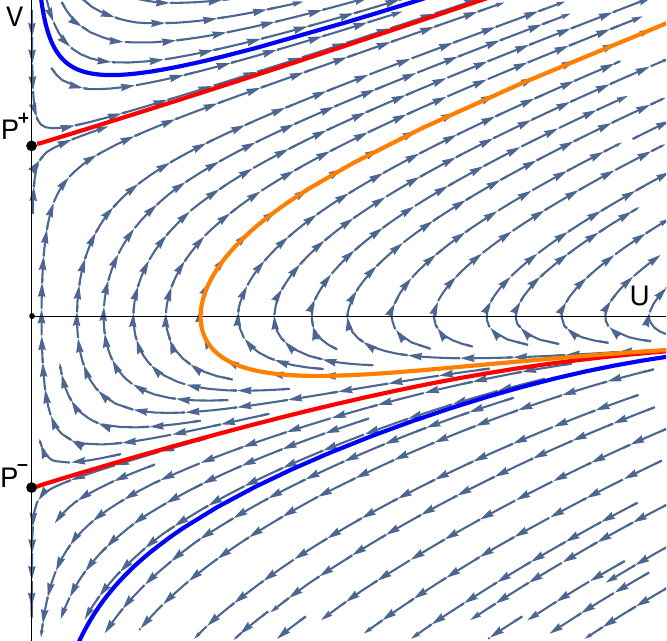}
\caption{Some trajectories in the phase-plane $(U,V)$ for $c=0$ (left), $c > 0$ (middle), and $c < 0$ (right).
The admissible profiles are painted in red. In orange, examples of positive profiles, and in blue, examples of profiles reaching zero linearly ---and thus, not satisfying the FB condition \eqref{eq:DerFBTW}.}
\label{fig:Trajectories}
\end{figure}

\

\emph{Step 1: Case $c = 0$.} If $c = 0$, then $\nu = 0$ and \eqref{eq:ODETW} has two critical points, $P^\pm$, and two constant solutions $V^\pm = \pm (\sqrt{2}/\beta)$. 
Substituting into the first equation of \eqref{eq:SysTW0} and recalling that we may assume $U(0) = 0$ by translation, we immediately find that $U(\xi) = \pm (\sqrt{2}/\beta) \xi$. Since $U := \phi^{1/\beta}$, such trajectories correspond to the stationary waves $\phi_0^\pm$ in the statement. 

\medskip

Whenever $V \not= V^{\pm}$, we may directly integrate \eqref{eq:ODETW}:
\[
\frac{-2V}{2/\beta^2 - V^2} \rd V = -\beta\gamma \frac{\rd U}{U},
\]
to obtain
\[
V^2 = \frac{2}{\beta^2} \pm k U^{-\beta\gamma},
\]
where $k > 0$ is a free constant. With the above explicit formula, we can classify all the remaining trajectories and the corresponding profiles, which are not admissible:

\

(i) The trajectories of the branch $V^2 = 2/\beta^2 - k_0 U^{-\beta\gamma}$ satisfy $V^- < V < V^+$ and correspond to \emph{positive} profiles satisfying $\phi(\xi) \sim c_\beta |\xi|^\beta$ as $\xi \to \pm\infty$. Up to translations, one may assume that each $\phi$ attains its global minimum at $\xi = 0$ and that $\phi' > 0$ in $(0,\infty)$, while $\phi' < 0$ in $(-\infty,0)$.

\

(ii) The trajectories of the branch $V^2 = 2/\beta^2 + k_0 U^{-\beta\gamma}$ satisfy $V^2 > 2/\beta^2$ and $|V|\to +\infty$ as $U\downarrow 0$: they are not admissible profiles.
    Indeed, since $V^2 \sim k_0 U^{-\beta\gamma}$ as $U \downarrow 0$, from the first equation of \eqref{eq:SysTW0} and using that $\beta-1 = \tfrac{\beta\gamma}{2}$, we deduce $U' \sim U^{1-\beta}$ (up to a multiplicative constant) which implies $\phi(\xi) = U^\beta(\xi) \sim \pm \xi$ as $\xi \to 0^\pm$. Thus, the FB law in \eqref{eq:DerFBTW} is not satisfied.

\

\emph{Step 2: Case $c \not= 0$.} Recall that if $c \not= 0$, then $\nu \not= 0$. 
As before, the system \eqref{eq:SysTW} has two critical points $P^\pm = (0,\pm \sqrt{2}/\beta)$. Linearizing around $P^\pm$, it is not difficult to check that both $P^+$ and $P^-$ are \emph{saddle}-type points. Thus, there is a trajectory $V^+ = V^+(U)$ ``going out'' from $P^+$ and a trajectory $V^- = V^-(U)$ ``entering'' in $P^-$. 
For the remaining analysis, we restrict ourselves to the case $\nu>0$ since the conclusions for $\nu<0$ will follow taking into account the symmetry mentioned before \emph{Step 1}.

To complete the phase-plane analysis, we study the \emph{null-isoclines}, i.e., the solutions to  
\[
\nu UV + V^2 = \frac{2}{\beta^2}, \quad  \text{ with } \nu >0.
\]
From this one easily shows that there are two curves in the region $\{U>0\}$ where $\rd V/\rd U = 0$.
The first one is given by a decreasing function $v^+_0 = v^+_0(U) > 0$ with $v^+_0(0) = P^+$ and $v^+_0(U) \to 0$ as $U \to \infty$.
The second one is given by a decreasing function $v^-_0 = v^-_0(U) < 0$ with $v^-_0(0) = P^-$ and $v^-_0(U) \to -\infty$ as $U \to \infty$. 
Having this, with standard ODE arguments we can easily show that $V^+ = V^+(U)$ is positive and decreasing, with $V^+(U) > v_0^+(U)$ for $U>0$, and $V^+(U) \downarrow 0$ as $U \to \infty$.
Similarly, $V^- = V^-(U)$ is negative and decreasing, with $V^-(U) > v_0^-(U)$ and $V^-(U) \to -\infty$ as $U \to \infty$. 
The remaining trajectories (which will not be admissible) behave as follows:

\

(i) There is a family of trajectories lying between the graphs of $V^-$ and $V^+$. Similarly as in the case $c=0$, they correspond to \emph{positive} profiles satisfying $\phi(\xi) \to \pm \infty$ as $\xi \to \pm \infty$ (clearly, when $c\not=0$ the behavior at $\xi = \pm \infty$ changes, according to the analysis above).

\

(ii) There is a family of trajectories satisfying $V > V^+$ and one satisfying $V < V^-$.
    Proceeding similarly as in the case $c=0$, taking into account that
    and $\rd V/\rd U \to \pm \infty$ as $U \downarrow 0$ and
    \begin{equation}
        \frac{2}{\beta\gamma}\frac{\rd V}{\rd U} = \frac{2/\beta^2 - \nu UV - V^2}{UV} \sim \frac{2/\beta^2 - V^2}{UV},
    \end{equation} 
    one can see that the profiles behave \emph{linearly} close to the FB. 
    The growth at infinity is the same as for the positive profiles.

\

Recall that the qualitative behavior of the trajectories in the phase-plane $(U,V)$ is given in \Cref{fig:Trajectories}. It remains to analyze the admissible profiles given by $V^\pm$.

By translation invariance, we may assume that $U(0)=0$ and therefore, proceeding as in the case $\nu=0$, we obtain $V^\pm \sim \pm \sqrt{2}/\beta$ as $U \downarrow  0$, which implies $U(\xi) \sim \pm (\sqrt{2}/\beta) \xi$ as $\xi \to 0^{\pm}$. 
We now examine the behavior of $V^\pm$ as $U \to \infty$. 
An elementary argument shows that $\rd V^+/\rd U \to 0$ as $U \to \infty$ and thus, passing to the limit into \eqref{eq:ODETW}, we deduce $V^+ \sim (2/(\beta^2\nu)) U^{-1}$ as $U \to \infty$ which, in turn, shows $U(\xi) \sim (2/(\beta\sqrt{\nu})) \xi^{1/2}$ as $\xi \to \infty$. 
For what concerns $V^-$, since $UV^-(U) \to -\infty$ as $U \to \infty$, \eqref{eq:ODETW} yields
\[
\frac{2}{\beta\gamma}\frac{\rd V^-}{\rd U} \sim - \nu - \frac{V^-}{U} \quad \text{ as } U \to \infty.
\]
Then, direct integration (recalling that $\beta-1 = \tfrac{\beta\gamma}{2} = c/ \nu$) gives $V^-(U) \sim -(c/\beta) U$ as $U \to \infty$, which, in light of the first equation of \eqref{eq:SysTW0}, gives $U(\xi) \sim e^{-\frac{c}{\beta}\xi}$ as $\xi \to -\infty$ up to a multiplicative constant. 
Therefore, the trajectories $V^\pm$ correspond to the profiles $\phi_c^\pm$ for $c>0$ described in the statement. 

\

\emph{Step 3: Proof of \eqref{eq:AllTWProfiles}.} The proof of \eqref{eq:AllTWProfiles} is an almost immediate consequence of the analysis above. Indeed, if $\psi_c$ is an admissible profile (with speed $c$), then it is positive at some $\xi_0$ and  either $\psi_c'(\xi_0) > 0$ and $\psi_c(\xi) = \phi_c^+(\xi-\xi_+)$ in $[\xi_+,\infty)$ for some $\xi_+ < \xi_0$, or $\psi_c'(\xi_0) < 0$ and $\psi_c(\xi) = \phi_c^-(\xi-\xi_-)$ in $(-\infty,\xi_-]$ for some $\xi_- > \xi_0$ (if $\psi_c'(x_0) = 0$, $\psi_c$ is a positive profile, i.e., not admissible).

If $\psi_c(\xi) = \phi_c^+(\xi-\xi_+)$ in $[\xi_+,\infty)$, then either $\psi_c = 0$ in $(-\infty,\xi_+]$ or $\psi_c > 0$ somewhere in $(-\infty,\xi_+)$. In the first case, we conclude. In the second, we must have $\psi_c(\xi) = \phi_c^-(\xi-\xi_-)$ in $(-\infty,\xi_-]$ for some $\xi_- \leq \xi_+$.
Indeed, it cannot be $\xi_- > \xi_+$: if so, $\psi_c(\xi) = \phi_c^+(\xi-\xi_+) = \phi_c^-(\xi-\xi_-)$ in $(\xi_+,\xi_-)$ which is impossible by definition of $\phi_c^+$ and $\phi_c^-$. 
\end{proof}
\begin{rem}\label{rem:CollTW}
Besides their self-interest, TWs allow us to build weak solutions with nontrivial \emph{singular} FB points. Below, we present two families of solutions exhibiting singular FB: the first one is characterized by a ``lower dimensional'' singular set, while the other has a FB made by singular points only. As mentioned in the introduction, such kind of singularities do \emph{not} appear in the Mean Curvature Flow theory (this is because surfaces flowing by mean curvature enjoy the Strong Maximum Principle, see for instance \cite{Mantegazza:book}). Notice that, by Remark \ref{remark:solutionsAreWeakSol}, when $\gamma = 0$, the solutions we construct below are weak solutions in the sense of \eqref{eq:FirstOVLimIntro}.

\begin{figure}[!ht]\label{fig:CollTW}
\includegraphics[scale=0.45]{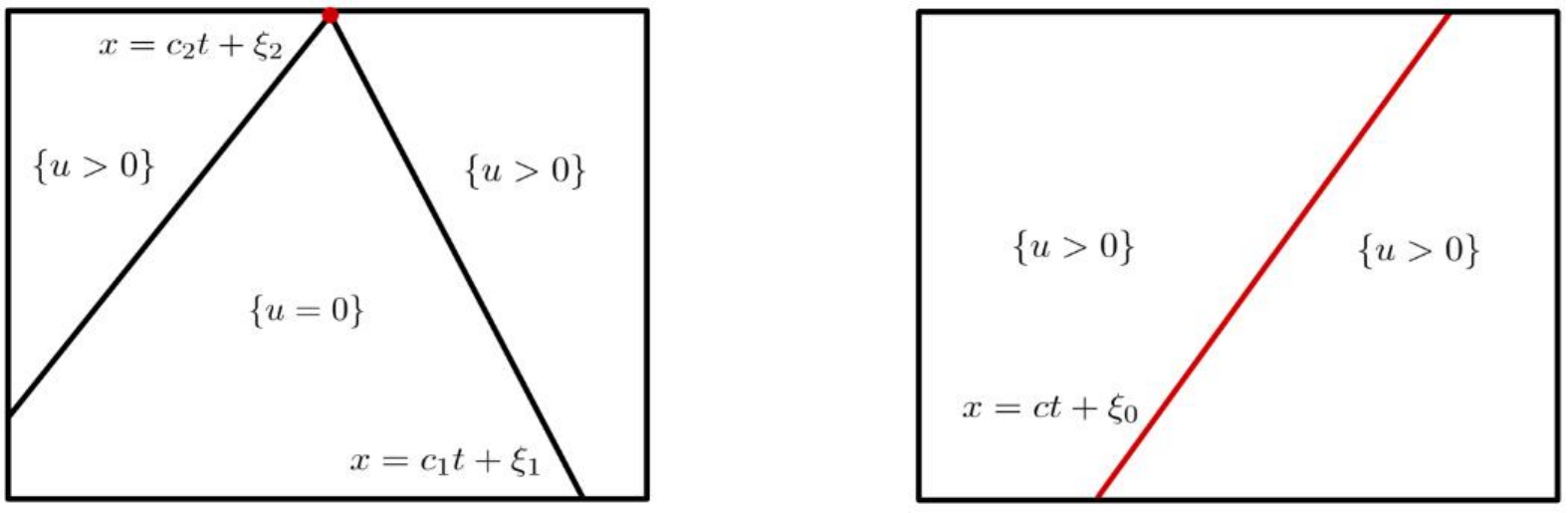}
\caption{The positivity sets of the ``colliding TWs'' (left) and the ``TWs sliding on a line'' (right).}
\end{figure}

The main examples are the so-called ``colliding traveling waves'' (see \cite{CafVaz95,art:Weiss2003} in the case $\gamma=0$). They can be constructed as follows. Fix $n=1$, $c_1 < 0 < c_2$, $\xi_2 < \xi_1$ and let $\phi_{c_1}^+$, $\phi_{c_2}^-$ as in  \Cref{thm:ExTW}. Then, by Remark \ref{rem:TWWeakSolutions}, 
\[
u(x,t) = \phi_{c_2}^-(x-c_2t - \xi_2) + \phi_{c_1}^+(x - c_1t - \xi_1)
\]
is a weak solution to \eqref{eq:EQSing} (in the sense of \Cref{def:WeakSolEIntro}) in $\RR\times (-\infty,t_\star)$, where $t_\star := \frac{\xi_1-\xi_2}{c_2-c_1}$. Further,
\[
\{u > 0\} = \big\{(x,t)\in \RR \times \RR: t_\star > t > \min\big\{\tfrac{x-\xi_1}{c_1},\tfrac{x-\xi_2}{c_2} \big\} \big\},
\]
and thus the FB is the \emph{cone} 
\begin{equation}\label{eq:ConicalFBCollTW}
t = \min\big\{\tfrac{x-\xi_1}{c_1},\tfrac{x-\xi_2}{c_2} \big\}, \qquad t \leq t_\star,
\end{equation}
with vertex $V = (x_\star,t_\star)$ (where $x_\star := \frac{\xi_1c_2 - \xi_2c_1}{c_2-c_1}$) and opening $\alpha := \pi - |\arctan(1/c_1)| - |\arctan(1/c_2)|$: thus $V$ is a ``conical'' \emph{singular} FB point, corresponding to the point where the waves $\phi_{c_1}^+$ and $\phi_{c_2}^-$ ``collide''. Notice that if $n \ge 2$, a similar construction works as well. The function 
\begin{equation}\label{eq:ConicalFBCollTWNDim}
u(x,t) = \phi_{c_2}^-(x_1-c_2t - \xi_2) + \phi_{c_1}^+(x_1 - c_1t - \xi_1)
\end{equation}
is still a weak solution, and its FB is given by \eqref{eq:ConicalFBCollTW}: depending on $n$, the \emph{singular} set $\{t = t_\star\} \cap \partial\{u > 0\}$ is a \emph{line} ($n=2$), a \emph{plane} ($n=3$) and so on. In all this cases, it is not difficult to check that the blow-up at singular FB points is $u_0(x,t) = (\sqrt{2}/\beta)^\beta |x_1|^\beta$ while, at \emph{regular} FB points (that is, $\{t < t_\star\} \cap \partial\{u > 0\}$), the blow-up is either $u_0(x,t) = (\sqrt{2}/\beta)^\beta (x_1)_+^\beta$ or $u_0(x,t) = (\sqrt{2}/\beta)^\beta (-x_1)_+^\beta$.
 
\

The second family of solutions with singular FB is even easier to construct. For $n=1$, fix $c,\xi_0 \in \RR$ and let $\phi_c^+$, $\phi_c^-$ as in  \Cref{thm:ExTW}. Then, by   \Cref{rem:TWWeakSolutions} again,
\[
u(x,t) = \phi_c^-(x-ct - \xi_0) + \phi_c^+(x - ct - \xi_0)
\]
is a weak solution to \eqref{eq:EQSing} (in the sense of \Cref{def:WeakSolEIntro}) in $\RR\times \RR$, with
\[
\{u > 0\} = \big\{(x,t): x \not= ct + \xi_0 \big\},
\]
and thus the FB is the line $x = ct +\xi_0$. Since the exterior normal vector at FB points is not defined, \emph{all} FB points are singular; we call such solutions ``TWs sliding on a line''. Even in this case, one can easily generalize the construction to higher dimensions by adding fictitious variables as in \eqref{eq:ConicalFBCollTWNDim}. In all this cases, the blow-up at FB points is $u_0(x,t) = (\sqrt{2}/\beta)^\beta |x_1|^\beta$.
\end{rem}
%
%
%
%
%
%
%
%
%
%
%
%
%
%
%
%


%
%

\end{document}